\documentclass[bj,preprint]{imsart}
\usepackage[numbers]{natbib}
\usepackage{amsthm,amsmath,natbib,amssymb,latexsym,euscript,enumerate,color,
cancel}
\RequirePackage[colorlinks,citecolor=blue,urlcolor=blue]{hyperref}

\startlocaldefs

\newtheorem{thm}{Theorem}[section]
\newtheorem{cor}{Corollary}[section]
\newtheorem{lem}{Lemma}[section]
\newtheorem{prop}{Proposition}[section]
\theoremstyle{definition}
\newtheorem{defn}{Definition}[section]
\theoremstyle{remark}
\newtheorem{rem}{Remark}[section]
\numberwithin{equation}{section}

\def \N{\mathbb{N}}
\def \R{\mathbb{R}}
\def \E{\mathbb{E}}
\def \F{\mathbb{F}}
\def \P{\mathbb{P}}

\def \eps{\varepsilon}
\def \trans{^{\scriptscriptstyle{\intercal}}}

\endlocaldefs

\allowdisplaybreaks

\begin{document}

\begin{frontmatter}

\title{Crandall-Lions Viscosity Solutions for Path-Dependent PDEs: the Case of Heat Equation}

\runtitle{Viscosity Solutions for Path-Dependent PDEs}

\author{
\textsc{Andrea Cosso}\thanks{University of Bologna, Dipartimento di Matematica, Piazza di Porta S. Donato, 5, 40126 Bologna, Italy, \sf andrea.cosso at unibo.it}
\qquad\quad
\textsc{Francesco Russo}\thanks{ENSTA Paris, Institut Polytechnique de Paris,
 Unit\'e de Math\'ematiques Appliqu\'ees, 828, boulevard des Mar\'echaux, 91120 Palaiseau, France, \sf francesco.russo at ensta-paris.fr}
}





\begin{abstract}
We address our interest to the development of a 
theory of viscosity solutions \`a la Crandall-Lions for path-dependent partial differential equations (PDEs), namely PDEs in the space of continuous paths $C([0,T];\R^d)$. Path-dependent PDEs can play a central role in the study of certain classes of optimal control problems, as for instance optimal control problems with delay. Typically, they do not admit a smooth solution satisfying the
corresponding HJB equation in a classical sense, it is therefore natural to search for a weaker notion of solution. While other notions of generalized solution have been proposed in the literature, the extension of the Crandall-Lions framework to the path-dependent setting is still an open problem. The question of uniqueness of the solutions, which is the most delicate issue, will be based on early ideas from the theory of viscosity solutions and a suitable variant of Ekeland's variational principle. This latter is  based on 
the construction of a smooth gauge-type function,
where smooth is meant in the horizontal/vertical (rather than Fr\'echet) sense. In order to make the presentation more readable, we address the path-dependent heat equation, which in particular simplifies the smoothing of its natural ``candidate'' solution. Finally, concerning the existence part, we provide a functional It\^o formula under general assumptions, extending earlier results in the literature.\end{abstract}



\end{frontmatter}

\section{Introduction}

Path-dependent heat equation refers to the second-order partial differential equation in the space of continuous paths
\begin{equation}\label{PPDE_Intro}
\begin{cases}
\vspace{2mm}
- \partial_t^H v(t,\boldsymbol x) - \dfrac{1}{2} \textup{tr}\big[\partial_{\boldsymbol x\boldsymbol x}^V v(t,\boldsymbol x)\big] \ = \ 0, \qquad\qquad &(t,\boldsymbol x)\in[0,T)\times C([0,T];\R^d), \\
v(T,\boldsymbol x) \ = \ \xi(\boldsymbol x), &\,\boldsymbol x\in C([0,T];\R^d).
\end{cases}
\end{equation}
Here $C([0,T];\R^d)$ denotes the Banach space of continuous paths $\boldsymbol x\colon[0,T]\rightarrow\R^d$ equipped with the supremum norm $\|\boldsymbol x\|_\infty:=\sup_{t\in[0,T]}|\boldsymbol x(t)|$, with $|\cdot|$ denoting the Euclidean norm on $\R^d$. The terminal condition $\xi\colon C([0,T];\R^d)\rightarrow\R$ is assumed to be continuous and bounded. We refer to equation \eqref{PPDE_Intro} as path-dependent \emph{heat} equation.
Similarly as for the usual heat equation, it admits
the following
Feynman-Kac representation formula in terms of the $d$-dimensional Brownian motion $\boldsymbol W=(\boldsymbol W_s)_{s\in[0,T]}$:
\begin{equation}\label{v_Intro}
v(t,\boldsymbol x) \ = \ \E\big[\xi(\boldsymbol W^{t,\boldsymbol x})\big], \qquad \forall\,(t,\boldsymbol x)\in[0,T]\times C([0,T];\R^d),
\end{equation}
where
\[
\boldsymbol W_s^{t,\boldsymbol x} \ := \
\begin{cases}
\boldsymbol x(s), & s\leq t, \\
\boldsymbol x(t) + \boldsymbol W_s - \boldsymbol W_t, \qquad\quad & s>t.
\end{cases}
\]
In the case of the classical heat equation $\xi$ only depends on the terminal value $\boldsymbol W_T^{t,\boldsymbol x}$.

The peculiarity of equation
\eqref{PPDE_Intro}
is the presence of the so-called functional or pathwise derivatives $\partial_t^H v$, $\partial_{\boldsymbol x\boldsymbol x}^V v$, where $\partial_t^H v$ is known as horizontal derivative, while $\partial_{\boldsymbol x\boldsymbol x}^V v$ is the matrix of second-order vertical derivatives.
Those derivatives appeared in \cite{K99,K15} (under the name of coinvariant derivatives) as building block of the so-called $i$-smooth analysis, and independently in \cite{AH02}, where they were denoted Clio derivatives; later, they were rediscovered by \cite{dupire} (from which we borrow terminology and definitions), which adopted a slightly different definition based on the space of c\`adl\`ag paths and in addition developed a related stochastic calculus, known as functional It\^o calculus, including in particular the so-called functional It\^o formula. Differently from the classical Fr\'echet derivative on $C([0,T];\R^d)$, the distinguished features of the pathwise derivatives are their finite-dimensional nature and the property of being non-anticipative, which follow from the interpretation of $t$ in $\boldsymbol x(t)$ as time variable. This means that $v(t,\boldsymbol x)$ only depends on the values of the path $\boldsymbol x$ up to time $t$; moreover, the horizontal and vertical derivatives at time $t$ are computed keeping the past values frozen, while only the present value of the path (that is $\boldsymbol x(t)$) can vary. The related functional It\^o calculus was rigorously investigated in \cite{contfournie10,ContFournieAP,CF16}. \cite{CGR,cosso_russoStrict}  also gave a contribution in this direction, exploring the relation between pathwise derivatives and Banach space stochastic calculus, built on an appropriate notion of Fr\'echet type derivative and firstly conceived in \cite{DGR}, see also \cite{DGRnote,DGR2,digirrusso12,digirfabbrirusso13}.

Partial differential equations in the space of continuous paths (also known as functional or Clio or path-dependent partial differential equations) are mostly motivated by optimal control problems of deterministic and stochastic systems with delay (or path-dependence) in the state variable. Such control systems arise in many fields, as for instance optimal advertising theory \cite{GM06,GMS09}, chemical engineering \cite{GKM09}, financial management \cite{F11,OSZ11}, economic growth theory \cite{BFG12}, mean field game theory \cite{BCLY17}, biomedicine \cite{GM18,RSTM18}, systemic risk \cite{CFMS18}. The underlying deterministic or stochastic controlled differential equations with delay can be studied
in two ways: first using a direct approach (see for instance \cite{IN64,S84,K99,IKS03,KI15}),
second by lifting them into a suitable infinite-dimensional framework, leading to evolution equations in Hilbert (as in \cite{cho,DZ96,flandoli_zanco13}) or Banach spaces (as in \cite{M84,M98,DGR}). The latter methodology turned out to be preferable to address general optimal control problems with delay (see for instance \cite{VK81,I82,GM06,FGG10a,FGG10b,FMT10,fabbrigozziswiech}), although such an infinite-dimensional reformulation may require some additional artificial assumptions to be imposed on the original control problem.
On the other hand, the direct approach was  adopted for special problems where the Hamilton-Jacobi-Bellman equation reduces to a finite-dimensional differential equation, as in \cite{EOS00,LR03}. This approach can
now regain relevance thanks to a well-grounded theory of path-dependent partial differential equations. To this regard, the path-dependent heat equation represents the primary test for such a theory, it indeed requires the main building blocks of the methodology, without overloading the proofs with additional technicalities.

Path-dependent partial differential equations represent a quite recent area of research. Typically, they do not admit a smooth solution satisfying the equation in a classical sense, mainly because of the awkward nature of the underlying space $C([0,T];\R^d)$. This happens also for the path-dependent heat equation, which in particular does not have the smoothing effect characterizing the classical heat equation,
except in some specific cases (as shown in \cite{DGR, DGRClassical})
with $\xi$ belonging to the class of so-called cylinder or tame functions (therefore depending specifically on a finite number of integrals with respect to the path) or $\xi$ being smoothly Fr\'echet differentiable. It is indeed quite easy, relying on the probabilistic representation formula \eqref{v_Intro}, to see that the function $v$ is not smooth (in the horizontal/vertical sense mentioned above) for terminal conditions of the form
\[
\xi(\boldsymbol x) \ = \ \sup_{0\leq t\leq T} \boldsymbol x(t), \qquad\qquad\qquad \xi(\boldsymbol x) \ = \ \boldsymbol x(t_0),
\]
for some fixed $t_0\in(0,T)$. For a detailed analysis of the first case above we refer to Section 3.2 in \cite{cosso_russoNorway}, see also Remark 3.8 in
\cite{cosso_russoStrict}. Concerning the second case, see for instance Example 11.1.3 of \cite{Zhang}. It is however worth mentioning that some positive results on smooth solutions were obtained in \cite{cosso_russoStrict,peng_wang}. We also refer to Chapter 9 of \cite{DGR} and \cite{DGRClassical}, where smooth solutions were investigated using a Fr\'echet type derivative formulation.

It is therefore natural to search for a weaker notion of solution, as the notion of viscosity solution, commonly used in the standard finite-dimensional case. The theory of viscosity solutions, firstly introduced in \cite{CL81,CL83} for first-order equations in finite dimension and later extended to the second-order case in \cite{L83c,L83a,L83b}, provides a well-suited framework guaranteeing the desired existence, uniqueness, and stability properties (for a comprehensive account see \cite{crandishiilions92}). The extension of such a theory to equations in infinite dimension was initiated by \cite{CL1,L88,L89a,L89b,S88,swiech94}. One of the structural assumption is that the state space has to be a Hilbert space or, slightly more general, certain Banach space with smooth norm, not including for instance the Banach space $C([0,T];\R^d)$ (notice however that in this paper we do not directly generalize those results to $C([0,T];\R^d)$, as we adopt horizontal/vertical, rather than Fr\'echet, derivatives on $C([0,T];\R^d)$).

First-order path-dependent partial differential equations were deeply investigated in \cite{Lu07} using a viscosity type notion of solution, which differs from the Crandall-Lions definition as the maximum/minimum condition is formulated on the subset of absolutely continuous paths. Such a modification does not affect existence in the first-order case, however it is particularly convenient for uniqueness, which is indeed established under general conditions. Other notions of generalized solution designed for first-order equations were adopted in \cite{AH02} as well as in \cite{Lu01,Lu03a,Lu03b}, where the minimax framework introduced in \cite{S80,SS78} was implemented. We also mention  \cite{BK18}, where such a minimax approach was extended to first-order path-dependent Hamilton-Jacobi-Bellman equations in infinite dimension. Concerning the second-order case, a first attempt to extend the Crandall-Lions framework to the path-dependent case was carried out in \cite{peng12}, even though a technical condition on the semi-jets was imposed, namely condition (16) in \cite{peng12}, which narrows down the applicability of such a result. In the literature, this was perceived as an almost insurmountable obstacle, so that the Crandall-Lions definition was not further investigated, while other notions of generalized solution were devised, see \cite{EKTZ,PS15,tangzhang13,cosso_russoStrong-Visc,Ohashi,paperPathDep,BKMZ}. We mention in particular the framework designed in \cite{EKTZ} and further investigated in \cite{etzI,etzII,R16,rtz1,rtz3,CFGRT}, where the notion of sub/supersolution adopted differs from the Crandall-Lions definition as the tangency condition is not pointwise but in the sense of expectation with respect to an appropriate class of probability measures. On the other hand, in \cite{cosso_russoStrong-Visc} we introduced the so-called strong-viscosity solution, which is quite similar to the notion of good solution for partial differential equations in finite dimension, that in turn is known to be equivalent to the definition of $L^p$-viscosity solution, see for instance \cite{JKS02}. We also mention \cite{paperPathDep}, where the authors deal with semilinear path-dependent equations and propose the notion of decoupled mild solution, formulated in terms of generalized transition 
semigroups; such a notion also adapts to path-dependent equations with integro-differential terms.

In the present paper we adopt the natural generalization of the well-known definition of viscosity solution \`a la Crandall-Lions given in terms of test functions and, under this notion, we establish existence and uniqueness for the path-dependent heat equation \eqref{PPDE_Intro}. The uniqueness property is derived, as usual, from the comparison theorem. The proof of this latter, which is the
most delicate issue, is known to be quite involved even in the classical finite-dimensional case (see for instance \cite{crandishiilions92}), and in its latest form is based on Ishii's lemma.  Here we follow instead an earlier approach (see for instance Theorem II.1 in \cite{L83b} or Theorem IV.1 in \cite{L88}), which in principle can be applied to any path-dependent equation admitting a ``candidate'' solution $v$, for which a probabilistic representation formula holds. This is the case for equation \eqref{PPDE_Intro}, where the candidate solution is given by formula \eqref{v_Intro}, but it is also the case for Kolmogorov type equations or, more generally, for Hamilton-Jacobi-Bellman equations. This latter is the class of equations studied in \cite{L83b} and \cite{L88}, whose methodology in a nutshell can be described as follows. Let $u$ (resp. $w$) be a viscosity subsolution (resp. supersolution) of the same path-dependent equation. The desired inequality $u\leq w$ follows if we compare both $u$ and $w$ to the ``candidate'' solution $v$, that is if we prove the two inequalities $u\leq v$ and $v\leq w$. Let us consider for instance the first inequality $u\leq v$. In the non-path-dependent and finite-dimensional case (as in \cite{L83b}), this is proved proceeding as follows: firstly, performing a smoothing of $v$ through its probabilistic representation formula; secondly, taking a local maximum of $u-v_n$ (here it is used the local compactness of the finite-dimensional underlying space), with $v_n$ being a smooth approximation of $v$; finally, the inequality $u\leq v_n$ is proved proceeding as in the so-called partial comparison theorem (comparison between a viscosity subsolution/supersolution and a smooth supersolution/subsolution), namely exploiting the viscosity subsolution property of $u$ with $v_n$ playing the role of test function. In \cite{L88}, where such a methodology was extended to the infinite-dimensional case, the existence of a maximum of $u-v_n$ is achieved relying on Ekeland's variational principle, namely exploiting the completeness of the space instead of the missing local compactness.

In this paper we generalize the methodology sketched above to the path-dependent case. There are however 
at least two crucial mathematical issues required by such a proof, still not at disposal in the path-dependent framework.

Firstly, given a candidate solution $v$, it is not a priori obvious how to perform a smooth approximation of $v$ itself starting from its probabilistic representation formula. Here we exploit the results proved in \cite{cosso_russoStrict} (Theorem 3.5) and \cite{cosso_russoStrong-Visc} (Theorem 3.12), which are reported and adapted to the present framework in Appendix \ref{AppC} (Lemma \ref{L:Classical} and Lemma \ref{L:Smoothing}, respectively). Notice that such results apply to the case of the path-dependent heat equation \eqref{PPDE_Intro}, where there is only the terminal condition $\xi$ in the probabilistic representation formula \eqref{v_Intro} for $v$. More general results are at disposal in \cite{cosso_russoStrict} and \cite{cosso_russoStrong-Visc}, which cover the case of semilinear path-dependent partial differential equations, characterized by the presence of four coefficients $b$, $\sigma$, $F$, $\xi$ (see, in particular, Theorem 3.16 in \cite{cosso_russoStrong-Visc} for more details). However, when those other coefficients appear in the path-dependent partial differential equation, we need more information on the sequence $\{v_n\}_n$ approximating $v$. For instance, we also have to estimate the derivatives of $v_n$ in order to proceed as in \cite{L83b} or \cite{L88}. Since such results are still not at disposal in the path-dependent setting, in order to make the paper more readable and not excessively lengthy, here we address the case of the path-dependent heat equation.

Secondly, concerning the existence of a maximum of $u-v_n$, we rely on a generalized version of Ekeland's variational principle for which we need a smooth gauge-type function with bounded derivatives, as explained below. Our equation is in fact formulated on the non-locally compact space $[0,T]\times C([0,T];\R^d)$ endowed with the pseudometric
\[
d_\infty\big((t,\boldsymbol x),(t',\boldsymbol x')\big) \ := \ |t - t'| + \|\boldsymbol x(\cdot\wedge t) - \boldsymbol x'(\cdot\wedge t')\|_\infty.
\]
Recall that Ekeland's variational principle, in its original form, applied to $([0,T]\times C([0,T];\R^d), d_\infty$) states that a perturbation $u(\cdot,\cdot)-v_n(\cdot,\cdot)-\delta d_\infty((\cdot,\cdot),(\bar t,\boldsymbol{\bar x}))$ of $u(\cdot,\cdot)-v_n(\cdot,\cdot)$ has a strict global maximum, with the perturbation being expressed in terms of the distance $d_\infty$ (the point $(\bar t,\boldsymbol{\bar x})$ is fixed). As the map $(t,\boldsymbol x)\mapsto d_\infty((t,\boldsymbol x),(\bar t,\boldsymbol{\bar x}))$ is not smooth, it cannot be a test function. In order to have a smooth map instead of $d_\infty$, we need a smooth variational principle on $[0,T]\times C([0,T];\R^d)$. To this end, the starting point is a generalization of the so-called Borwein-Preiss smooth variant of Ekeland's variational principle (see for instance \cite{BZ05}), which works when $d_\infty$ is replaced by a so-called gauge-type function (see Definition \ref{D:Ekeland}). For the proof of the comparison theorem, we have to construct a gauge-type function which is also smooth and with bounded derivatives, recalling that \emph{smooth} in the present context means in the horizontal/vertical (rather than in the Fr\'echet) sense. In Section \ref{S:Smoothing} such a gauge-type function is built through a smoothing of $d_\infty$ itself (more precisely, of the part concerning the supremum norm). This latter smoothing is performed by convolution, firstly in the vertical direction, that is in the direction of the map $1_{[t,T]}$ (Lemma \ref{L:rho_infty}), then in the horizontal direction (Lemma \ref{L:Smoothing_rho_infty}), the ordering of smoothings being crucial. Notice in particular that the supremum norm is already smooth in the horizontal direction; however, after the vertical smoothing, we lose in general the horizontal regularity because of the presence of the term $1_{[t,T]}$; for this reason we have also to perform the horizontal smoothing. The resulting smooth gauge-type function with bounded derivatives corresponds to the function $\rho_\infty$ defined in \eqref{rho}.

Regarding existence, we prove that the candidate solution $v$ in \eqref{v_Intro} solves in the viscosity sense equation \eqref{PPDE_Intro}. We proceed essentially as in the classical non-path-dependent case, relying as usual on It\^o's formula, which in the present context corresponds to the functional It\^o formula. Such a formula was firstly stated in \cite{dupire} and then rigorously proved in \cite{contfournie10,ContFournieAP}, see also \cite{CF16,fournie,cosso_russoStrict,Ohashi,Oberhauser}. In the present paper we provide a functional It\^o formula under general assumptions (Theorem \ref{T:Ito}). In particular, we do not require any boundedness assumption on the functional $u\colon[0,T]\times C([0,T];\R^d)\rightarrow\R$, thus improving (when the semimartingale process is continuous) the results stated in \cite{contfournie10,ContFournieAP}.

The paper is organized as follows. Section \ref{S:FunctItoCalc} is devoted to pathwise derivatives and functional It\^o calculus. In particular, there is the functional It\^o formula (Theorem \ref{T:Ito}) whose complete proof is reported in Appendix \ref{AppA}. In Section \ref{S:Smoothing} we prove the smooth variational principle on $[0,T]\times C([0,T];\R^d)$, constructing the smooth gauge-type function with bounded derivatives. In Section \ref{S:Visc} we provide the (path-dependent) Crandall-Lions definition of viscosity solution for a general path-dependent partial differential equation. We then study in detail the path-dependent heat equation. In particular, we prove existence showing that the so-called candidate solution $v$ solves in the viscosity sense the path-dependent heat equation (Theorem \ref{T:Existence}). We conclude Section \ref{S:Visc} proving the comparison theorem (Theorem \ref{T:Comparison}) and uniqueness (Corollary \ref{C:Uniqueness}).

\section{Pathwise derivatives and functional It\^o calculus}
\label{S:FunctItoCalc}

In the present section we define the pathwise derivatives and prove the functional It\^o formula under general assumptions.

\subsection{Maps on c\`adl\`ag paths}

Given $T>0$ and $d\in\N^*$, we denote by $D([0,T];\R^d)$ the set of c\`adl\`ag functions $\boldsymbol{\hat x}\colon[0,T]\rightarrow\R^d$. We denote by $\boldsymbol{\hat x}(t)$ the value of $\boldsymbol{\hat x}$ at $t\in[0,T]$. We also denote by $\boldsymbol 0$ the function $\boldsymbol{\hat x}\colon[0,T]\rightarrow\R^d$ identically equal to zero. We consider on $D([0,T];\R^d)$ the supremum norm $\|\cdot\|_\infty$, namely $\|\boldsymbol{\hat x}\|_\infty:=\sup_{t\in[0,T]}|\boldsymbol{\hat x}(t)|$, where $|\cdot|$ denotes the Euclidean norm on $\R^d$ (we use the same symbol $|\cdot|$ to denote the Euclidean norm on $\R^k$, for any $k\in\N$). We refer to Chapter V in \cite{Pollard} and to Section 15 of Chapter 3 in \cite{Billingsley} for a study of the set of c\`adl\`ag functions endowed with the uniform metric and a comparison with the Skorokhod space.

We set $\boldsymbol{\hat\Lambda}:=[0,T]\times D([0,T];\R^d)$ and define $\hat d_\infty\colon\boldsymbol{\hat\Lambda}\times\boldsymbol{\hat\Lambda}\rightarrow[0,\infty)$ as 
\[
\hat d_\infty\big((t,\boldsymbol{\hat x}),(t',\boldsymbol{\hat x}')\big) \ := \ |t - t'| + \big\|\boldsymbol{\hat x}(\cdot\wedge t) - \boldsymbol{\hat x}'(\cdot\wedge t')\big\|_\infty.
\]
Notice that $\hat d_\infty$ is a \emph{pseudometric} on $\boldsymbol{\hat\Lambda}$, that is $\hat d_\infty$ is not a true metric because one may have $\hat d_\infty((t,\boldsymbol{\hat x}),(t',\boldsymbol{\hat x}'))=0$ even if $(t,\boldsymbol{\hat x})\neq(t',\boldsymbol{\hat x}')$. We recall that one can construct a true metric space $(\boldsymbol{\hat\Lambda}^{\!*},\hat d_\infty^{\,*})$, called the metric space induced by the pseudometric space $(\boldsymbol{\hat\Lambda},\hat d_\infty)$, by means of the equivalence relation which follows from the vanishing of the pseudometric. We also observe that $(\boldsymbol{\hat\Lambda},\hat d_\infty)$ is a complete pseudometric space. Finally, we denote by $\mathcal B(\boldsymbol{\hat\Lambda})$ the Borel $\sigma$-algebra on $\boldsymbol{\hat\Lambda}$ induced by $\hat d_\infty$.

\begin{defn}
A map (or functional) $\hat u\colon\boldsymbol{\hat\Lambda}\rightarrow\R$ is said to be \textbf{non-anticipative} (on $\boldsymbol{\hat\Lambda}$) if it satisfies
\[
\hat u(t,\boldsymbol{\hat x}) \ = \ \hat u(t,\boldsymbol{\hat x}(\cdot\wedge t)),
\]
for all $(t,\boldsymbol{\hat x})\in\boldsymbol{\hat\Lambda}$.
\end{defn}

\begin{rem}
(i) The property of being non-anticipative is crucial and automatically true if the map $\hat u\colon\boldsymbol{\hat\Lambda}\rightarrow\R$ is continuous with respect to $\hat d_\infty$.\\
(ii) More generally, it holds that whenever $\hat u\colon\boldsymbol{\hat\Lambda}\rightarrow\R$ is Borel measurable, namely $\hat u$ is measurable with respect to $\mathcal B(\boldsymbol{\hat\Lambda})$, then $\hat u$ is non-anticipative on $\boldsymbol{\hat\Lambda}$. As a matter of fact, notice that every open subset $B$ of $\boldsymbol{\hat\Lambda}$, endowed with $\hat d_\infty$, satisfies the following property: if $(t,\boldsymbol{\hat x})\in B$ then $(t,\boldsymbol{\hat x}(\cdot\wedge t))\in B$ (this follows from the fact that $\hat d_\infty((t,\boldsymbol{\hat x}),(t,\boldsymbol{\hat x}(\cdot\wedge t)))=0$). As a consequence, by a monotone class argument, the same property holds true for every Borel subset of $\boldsymbol{\hat\Lambda}$. Now, let $\hat u\colon\boldsymbol{\hat\Lambda}\rightarrow\R$ be Borel measurable. For every $(t,\boldsymbol{\hat x})\in\boldsymbol{\hat\Lambda}$, denote
\[
B_{\hat u(t,\boldsymbol{\hat x})} \ := \ \big\{(s,\boldsymbol{\hat y})\in\boldsymbol{\hat\Lambda}\colon\hat u(s,\boldsymbol{\hat y})=\hat u(t,\boldsymbol{\hat x})\big\}.
\]
Notice that $B_{\hat u(t,\boldsymbol{\hat x})}\in\mathcal B(\boldsymbol{\hat\Lambda})$ and since $(t,\boldsymbol{\hat x})\in B_{\hat u(t,\boldsymbol{\hat x})}$ we deduce that $(t,\boldsymbol{\hat x}(\cdot\wedge t))\in B_{\hat u(t,\boldsymbol{\hat x})}$. This means that $\hat u(t,\boldsymbol{\hat x}(\cdot\wedge t))=\hat u(t,\boldsymbol{\hat x})$, namely the map $\hat u$ is non-anticipative.
\end{rem}

\begin{defn}
We denote by $\boldsymbol C(\boldsymbol{\hat\Lambda})$ the set of maps $\hat u\colon\boldsymbol{\hat\Lambda}\rightarrow\R$ which are continuous on $\boldsymbol{\hat\Lambda}$ with respect to $\hat d_\infty$.
\end{defn}

\begin{defn}[\textbf{Pathwise derivatives}]\label{D:FunctionalDerivatives}
Let $\hat u\colon\boldsymbol{\hat\Lambda}\rightarrow\R$ be non-anticipative.
\begin{enumerate}[\upshape (i)]
\item Given $(t,\boldsymbol{\hat x})\in\boldsymbol{\hat\Lambda}$, with $t<T$, the \textbf{horizontal derivative} of $\hat u$ at $(t,\boldsymbol{\hat x})$ (if the corresponding limit exists) is defined as 
\[
\partial_t^H\hat u(t,\boldsymbol{\hat x}) \ := \ \lim_{\delta\rightarrow0^+}\frac{\hat u(t+\delta,\boldsymbol{\hat x}(\cdot\wedge t)) - \hat u(t,\boldsymbol{\hat x})}{\delta}.
\]
At $t=T$ the horizontal derivative is defined as
\[
\partial_t^H\hat u(T,\boldsymbol{\hat x}) \ := \ \lim_{t\rightarrow T^-} \partial_t^H\hat u(t,\boldsymbol{\hat x}).
\]
\item Given $(t,\boldsymbol{\hat x})\in\boldsymbol{\hat\Lambda}$, the \textbf{vertical derivatives} of first and second-order of $\hat u$ at $(t,\boldsymbol{\hat x})$ (if the corresponding limits exist) are defined as 
\begin{align*}
\partial_{x_i}^V\hat u(t,\boldsymbol{\hat x}) \ &:= \ \lim_{h\rightarrow0}\frac{\hat u(t,\boldsymbol{\hat x} + h\,\mathbf e_i\,1_{[t,T]}) - \hat u(t,\boldsymbol{\hat x})}{h}, \\
\partial_{x_i x_j}^V\hat u(t,\boldsymbol{\hat x}) \ &:= \ \partial_{x_j}^V(\partial_{x_i}^V\hat u)(t,\boldsymbol{\hat x}),
\end{align*}
where $\mathbf e_1,\ldots,\mathbf e_d$ is the standard orthonormal basis of $\R^d$.\\
Finally, we denote $\partial_{\boldsymbol x}^V\hat u=(\partial_{x_1}^V\hat u,\ldots,\partial_{x_d}^V\hat u)$ and $\partial_{\boldsymbol x\boldsymbol x}^V\hat u=(\partial_{x_i x_j}^V\hat u)_{i,j=1,\ldots,d}$.
\end{enumerate}
\end{defn}

\begin{defn}
We denote by $\boldsymbol C^{1,2}(\boldsymbol{\hat\Lambda})$ the set of $\hat u\in\boldsymbol C(\boldsymbol{\hat\Lambda})$ such that $\partial_t^H\hat u$, $\partial_{\boldsymbol x}^V\hat u$, $\partial_{\boldsymbol x\boldsymbol x}^V\hat u$ exist everywhere on $\boldsymbol{\hat\Lambda}$ and are continuous.
\end{defn}

For later use, we also introduce the following set of maps on c\`adl\`ag paths.

\begin{defn}
We denote by $\boldsymbol C^{0,2}(\boldsymbol{\hat\Lambda})$ the set of $\hat u\in\boldsymbol C(\boldsymbol{\hat\Lambda})$ such that $\partial_{\boldsymbol x}^V\hat u$, $\partial_{\boldsymbol x\boldsymbol x}^V\hat u$ exist everywhere on $\boldsymbol{\hat\Lambda}$ and are continuous.
\end{defn}

We can finally state the functional It\^o formula for maps on c\`adl\`ag paths, whose proof is reported in Appendix \ref{AppA}.

\begin{thm}\label{T:ItoLifted}
Let $\hat u\in\boldsymbol C^{1,2}(\boldsymbol{\hat\Lambda})$. Then, for every $d$-dimensional continuous semimartingale $\boldsymbol X=(\boldsymbol X_t)_{t\in[0,T]}$, where $\boldsymbol X=(X^1,\ldots,X^d)$, defined on some filtered probability space $(\Omega,{\mathcal F},({\mathcal F}_t)_{t\in[0,T]},\P)$, with $({\mathcal F}_t)_{t\in[0,T]}$ satisfying the usual conditions, the following \textbf{functional It\^o formula} holds:
\begin{align*}
\hat u(t,\boldsymbol X) \ &= \ \hat u(0,\boldsymbol X) + \int_0^t \partial_t^H\hat u(s,\boldsymbol X)\,ds + \frac{1}{2}\sum_{i,j=1}^d \int_0^t \partial_{x_i x_j}^V \hat u(s,\boldsymbol X)\,d[X^i,X^j]_s \\
&\quad \ + \sum_{i=1}^d \int_0^t \partial_{x_i}^V\hat u(s,\boldsymbol X)\,dX_s^i, \hspace{2.05cm} \text{for all }\,0\leq t\leq T,\,\,\P\text{-a.s.}
\end{align*}
\end{thm}
\begin{proof}
See Appendix \ref{AppA}.
\end{proof}

\subsection{Maps on continuous paths}
\label{SubS:FunctionalDerivativesCont}

Let $C([0,T];\R^d)$ denote the set of continuous functions $\boldsymbol x\colon[0,T]\rightarrow\R^d$. Notice that $ C([0,T];\R^d)$ is a subset of $ D([0,T];\R^d)$. We set ${\boldsymbol\Lambda}:=[0,T]\times C([0,T];\R^d)$ and denote $\boldsymbol d_\infty$ the restriction of $\boldsymbol{\hat d}_\infty$ to ${\boldsymbol\Lambda}\times{\boldsymbol\Lambda}$. Then, $\boldsymbol d_\infty$ is a pseudometric on ${\boldsymbol\Lambda}$ and $({\boldsymbol\Lambda},\boldsymbol d_\infty)$ is a complete pseudometric space. We denote by $\mathcal B({\boldsymbol\Lambda})$ the Borel $\sigma$-algebra on ${\boldsymbol\Lambda}$ induced by $\boldsymbol d_\infty$.

\begin{defn}\label{D:Lift}
Let $\hat u\colon\boldsymbol{\hat\Lambda}\rightarrow\R$ be non-anticipative and consider $u\colon\boldsymbol\Lambda\rightarrow\R$. We say that $\hat u$ is \textbf{consistent} with $u$ if
\[
u(t,\boldsymbol x) \ = \ \hat u(t,\boldsymbol x),
\]
for all $(t,\boldsymbol x)\in\boldsymbol\Lambda$.
\end{defn}

The following consistency property is crucial as it implies that, given $u$ admitting two maps $\hat u_1$ and $\hat u_2$, both being consistent with $u$, their pathwise derivatives coincide on continuous paths (see also Remark \ref{R:Consistency}).

\begin{lem}\label{L:Consistency}
If $\hat u_1,\hat u_2\in\boldsymbol C^{1,2}(\boldsymbol{\hat\Lambda})$ satisfy
\[
\hat u_1(t,\boldsymbol x) \ = \ \hat u_2(t,\boldsymbol x), \qquad \forall\,(t,\boldsymbol x)\in\boldsymbol\Lambda,
\]
then
\begin{align*}
\partial_t^H\hat u_1(t,\boldsymbol x) \ &= \ \partial_t^H\hat u_2(t,\boldsymbol x), \\
\partial^V_{\boldsymbol x} \hat u_1(t,\boldsymbol x) \ &= \ \partial^V_{\boldsymbol x} \hat u_2(t,\boldsymbol x), \\
\partial^V_{\boldsymbol x\boldsymbol x} \hat u_1(t,\boldsymbol x) \ &= \ \partial^V_{\boldsymbol x\boldsymbol x} \hat u_2(t,\boldsymbol x),
\end{align*}
for all $(t,\boldsymbol x)\in\boldsymbol\Lambda$.
\end{lem}
\begin{proof}
See Appendix \ref{AppB}.
\end{proof}

Thanks to Lemma \ref{L:Consistency} we can now give the following definition (see also Remark \ref{R:Consistency}).

\begin{defn}\label{C^1,2(Lambda)}
Let $u\colon\boldsymbol\Lambda\rightarrow\R$. We say that $u\in\boldsymbol C^{1,2}(\boldsymbol\Lambda)$ if there exists $\hat u\colon\boldsymbol{\hat\Lambda}\rightarrow\R$ consistent with $u$ and satisfying $\hat u\in\boldsymbol  C^{1,2}(\boldsymbol{\hat\Lambda})$. Moreover, we define
\begin{align*}
\partial_t^H u(t,\boldsymbol x) \ &:= \ \partial_t^H\hat u(t,\boldsymbol x), \\
\partial^V_{\boldsymbol x} u(t,\boldsymbol x) \ &:= \ \partial^V_{\boldsymbol x} \hat u(t,\boldsymbol x), \\
\partial^V_{\boldsymbol x\boldsymbol x} u(t,\boldsymbol x) \ &:= \ \partial^V_{\boldsymbol x\boldsymbol x} \hat u(t,\boldsymbol x),
\end{align*}
for all $(t,\boldsymbol x)\in\boldsymbol\Lambda$.
\end{defn}

\begin{rem}\label{R:Consistency}
Notice that, by Lemma \ref{L:Consistency}, if $u\in\boldsymbol  C^{1,2}(\boldsymbol\Lambda)$ then the definition of the pathwise derivatives of $u$ is independent of the map $\hat u\in\boldsymbol C^{1,2}(\boldsymbol{\hat\Lambda})$ consistent with $u$.
\end{rem}

\begin{thm}\label{T:Ito}
Let $u\in\boldsymbol  C^{1,2}(\boldsymbol\Lambda)$. Then, for every $d$-dimensional continuous semimartingale $\boldsymbol X=(\boldsymbol X_t)_{t\in[0,T]}$, where $\boldsymbol X=(X^1,\ldots,X^d)$, defined on some filtered probability space $(\Omega,{\mathcal F},({\mathcal F}_t)_{t\in[0,T]},\P)$, with $({\mathcal F}_t)_{t\in[0,T]}$ satisfying the usual conditions, the following \textbf{functional It\^o formula} holds:
\begin{align}
u(t,\boldsymbol X) \ &= \ u(0,\boldsymbol X) + \int_0^t \partial_t^H u(s,\boldsymbol X)\,ds + \frac{1}{2}\sum_{i,j=1}^d \int_0^t \partial^V_{x_i x_j} u(s,\boldsymbol X)\,d[X^i,X^j]_s \label{Ito_formula} \\
&\quad \ + \sum_{i=1}^d \int_0^t \partial^V_{x_i}u(s,\boldsymbol X)\,dX_s^i, \qquad\qquad\quad\;\;\ \text{for all }\,0\leq t\leq T,\,\,\P\text{-a.s.} \notag
\end{align}
\end{thm}
\begin{proof}
Since $u\in\boldsymbol  C^{1,2}(\boldsymbol\Lambda)$, by Definition \ref{C^1,2(Lambda)} there exists a map $\hat u\colon\boldsymbol{\hat\Lambda}\rightarrow\R$ consistent with $u$ and satisfying $\hat u\in\boldsymbol C^{1,2}(\boldsymbol{\hat\Lambda})$. Then, by Theorem \ref{T:ItoLifted}, the following functional It\^o formula holds:
\begin{align*}
\hat u(t,\boldsymbol X) \ &= \ \hat u(0,\boldsymbol X) + \int_0^t \partial_t^H\hat u(s,\boldsymbol X)\,ds + \frac{1}{2}\sum_{i,j=1}^d \int_0^t \partial^V_{x_i x_j} \hat u(s,\boldsymbol X)\,d[X^i,X^j]_s \\
&\quad \ + \sum_{i=1}^d \int_0^t \partial^V_{x_i}\hat u(s,\boldsymbol X)\,dX_s^i, \hspace{2.1cm} \text{for all }\,0\leq t\leq T,\,\,\P\text{-a.s.}
\end{align*}
The claim follows identifying the pathwise derivatives of $\hat u$ with those of $u$.
\end{proof}

\section{Smooth variational principle on $\boldsymbol\Lambda$}
\label{S:Smoothing}

The goal of the present section is the proof of a smooth variational principle on $\boldsymbol\Lambda$, which plays a crucial role in the proof of the comparison theorem (Theorem \ref{T:Comparison}). To this end, we begin recalling a generalization of the so-called Borwein-Preiss smooth variant (\cite{BP87}) of Ekeland's variational principle (\cite{Ekeland}), corresponding to Theorem \ref{T:Ekeland} below. We state it for the case of real-valued (rather than $\mathbb R\cup\{+\infty\}$-valued as in \cite{BZ05}) maps on $\boldsymbol\Lambda$. We firstly recall the definition of gauge-type function for the specific set $\boldsymbol\Lambda$.

\begin{defn}\label{D:Ekeland}
We say that $\Psi\colon\boldsymbol\Lambda\times\boldsymbol\Lambda\rightarrow[0,+\infty)$ is a \textbf{gauge-type function} provided that the properties below hold:
\begin{enumerate}[\upshape a)]
\item $\Psi$ is continuous on $\boldsymbol\Lambda\times\boldsymbol\Lambda$;
\item $\Psi((t,\boldsymbol x),(t,\boldsymbol x))=0$, for every $(t,\boldsymbol x)\in\boldsymbol\Lambda$;
\item for every $\eps>0$ there exists $\eta>0$ such that, for all $(t',\boldsymbol x'),(t'',\boldsymbol x'')\in\boldsymbol\Lambda$, the inequality $\Psi((t',\boldsymbol x'),(t'',\boldsymbol x''))\leq\eta$ implies $d_\infty((t',\boldsymbol x'),(t'',\boldsymbol x''))<\eps$.
\end{enumerate}
\end{defn}

\begin{thm}\label{T:Ekeland}
  Let $G\colon\boldsymbol\Lambda\rightarrow\R$ be an upper semicontinuous map, bounded from above. Suppose that $\Psi\colon\boldsymbol\Lambda\times\boldsymbol\Lambda\rightarrow[0,+\infty)$ is a gauge-type function (according to Definition \ref{D:Ekeland}) and $\{\delta_n\}_{n\geq0}$ is a sequence of strictly positive real numbers. For every $\eps>0$, let $(t_0,\boldsymbol x_0)\in\boldsymbol\Lambda$
  such that
\[
\sup G - \eps \ \leq \ G(t_0,\boldsymbol x_0).
\]
Then, there exists a sequence $\{(t_n,\boldsymbol x_n)\}_{n\geq1}\subset\boldsymbol\Lambda$ which converges to some $(\bar t,\boldsymbol{\bar x})\in\boldsymbol\Lambda$ satisfying the following properties.
\begin{enumerate}[\upshape i)]
\item $\Psi((\bar t,\boldsymbol{\bar x}),(t_n,\boldsymbol x_n))\leq\frac{\eps}{2^n\delta_0}$, for every $n\geq0$.
\item $G(t_0,\boldsymbol x_0)\leq G(\bar t,\boldsymbol{\bar x}) - \sum_{n=0}^{+\infty} \delta_n\,\Psi((\bar t,\boldsymbol{\bar x}),(t_n,\boldsymbol x_n))$.
\item For every $(t,\boldsymbol x)\neq(\bar t,\boldsymbol{\bar x})$,
\[
G(t,\boldsymbol x) - \sum_{n=0}^{+\infty} \delta_n\,\Psi\big((t,\boldsymbol x),(t_n,\boldsymbol x_n)\big) \ < \ G(\bar t,\boldsymbol{\bar x}) - \sum_{n=0}^{+\infty} \delta_n\,\Psi\big((\bar t,\boldsymbol{\bar x}),(t_n,\boldsymbol x_n)\big).
\]
\end{enumerate}
\end{thm}
\begin{proof}
Theorem \ref{T:Ekeland} follows trivially from Theorem 2.5.2 in \cite{BZ05}, the only difference being that the latter result is stated on complete \emph{metric} spaces, while here $\boldsymbol\Lambda$ is a complete \emph{pseudometric} space.
\end{proof}

The main ingredient of Theorem \ref{T:Ekeland} is the gauge-type function $\Psi$. In the proof of the comparison theorem we need such a gauge-type function to be also smooth as a map of its first pair, namely $(t,\boldsymbol x)\mapsto\Psi((t,\boldsymbol x),(t_0,\boldsymbol x_0))$, and with bounded derivatives. The most important example of gauge-type function is the pseudometric $d_\infty$ itself, which unfortunately is not smooth enough. The major contribution of the present section is the construction of such a smooth gauge-type function with bounded derivatives, which corresponds to the function $\rho_\infty$ in \eqref{rho}. In order to do it, we perform a smoothing of the pseudometric $d_\infty$ itself (more precisely of the part concerning the supremum norm), first in the vertical direction, and then in the horizontal direction. In particular, the next result concerns the smoothing in the vertical direction. The precise form of the mollifier $\zeta$ in \eqref{zeta} is used to get explicit bounds on $\hat\kappa_\infty^{(t_0,\boldsymbol x_0)}$ and its derivatives.

\begin{lem}\label{L:rho_infty}
  Let $\zeta\colon\R^d\rightarrow\R$ be the probability density function of the standard normal multivariate distribution
\begin{equation}\label{zeta}
\zeta(\mathbf z) \ := \ \frac{1}{(2\pi)^{\frac{d}{2}}}\,\textup{e}^{-\frac{1}{2}|\mathbf z|^2}, \qquad \forall\,\mathbf z\in\R^d.
\end{equation}
For every fixed $(t_0,\boldsymbol x_0)\in\boldsymbol\Lambda$, define the map $\hat\kappa_\infty^{(t_0,\boldsymbol x_0)}\colon\boldsymbol{\hat\Lambda}\rightarrow\R$ as
\begin{equation}\label{kappa}
\hat\kappa_\infty^{(t_0,\boldsymbol x_0)}(t,\boldsymbol{\hat x}) \ := \ \int_{\R^d}\big\|\boldsymbol{\hat x}(\cdot\wedge t) - \boldsymbol x_0(\cdot\wedge t_0) - \mathbf z\,1_{[t,T]}\big\|_\infty\,\zeta(\mathbf z)\,d\mathbf z - \int_{\R^d}|\mathbf z|\,\zeta(\mathbf z)\,d\mathbf z,
\end{equation}
for all $(t,\boldsymbol{\hat x})\in\boldsymbol{\hat\Lambda}$. Moreover, let
$\kappa_\infty^{(t_0,\boldsymbol x_0)}\colon\boldsymbol\Lambda\rightarrow\R$ be given by
\[
\kappa_\infty^{(t_0,\boldsymbol x_0)}(t,\boldsymbol x) \ := \ \hat\kappa_\infty^{(t_0,\boldsymbol x_0)}(t,\boldsymbol x),
\]
for every $(t,\boldsymbol x)\in\boldsymbol\Lambda$.
Then, the following properties hold.
\begin{enumerate}[\upshape 1)]
\item For every $(t,\boldsymbol{\hat x})\in\boldsymbol{\hat\Lambda}$, the vertical derivatives of first and second-order of $\hat\kappa_\infty^{(t_0,\boldsymbol x_0)}$ at $(t,\boldsymbol{\hat x})$ (namely $\partial^V_{x_i}\hat\kappa_\infty^{(t_0,\boldsymbol x_0)}(t,\boldsymbol{\hat x})$ and $\partial^V_{x_i x_j}\hat\kappa_\infty^{(t_0,\boldsymbol x_0)}(t,\boldsymbol{\hat x})$, for every $i,j=1,\ldots,d$) exist.
\item For every $i,j=1,\ldots,d$, $\partial^V_{x_i}\hat\kappa_\infty^{(t_0,\boldsymbol x_0)}$ is bounded by the constant $1$ and $\partial^V_{x_i x_j}\hat\kappa_\infty^{(t_0,\boldsymbol x_0)}$ is bounded by the constant $\sqrt{\frac{2}{\pi}}$.
\item $\hat\kappa_\infty^{(t_0,\boldsymbol x_0)}\geq-C_\zeta$ and $\kappa_\infty^{(t_0,\boldsymbol x_0)}(t,\boldsymbol x)\geq\|\boldsymbol x(\cdot\wedge t)-\boldsymbol x_0(\cdot\wedge t_0)\|_\infty-C_\zeta$, for every $(t,\boldsymbol x)\in\boldsymbol\Lambda$, with
\begin{equation}\label{C_zeta}
C_\zeta \ := \ \int_{\R^d}|\mathbf z|\zeta(\mathbf z)\,d\mathbf z \ = \ \sqrt{2}\,\frac{\Gamma\big(\frac{d}{2} + \frac{1}{2}\big)}{\Gamma\big(\frac{d}{2}\big)} \ > \ 0,
\end{equation}
where $\Gamma(\cdot)$ is the Gamma function.
\item For every fixed $d$, there exists some constant $\alpha_d>0$ such that
\begin{align}\label{Ineq12}
\alpha_d\,\big(\|\boldsymbol x(\cdot\wedge t) - \boldsymbol x_0(\cdot\wedge t_0)\|_\infty^{d+1}\wedge\|&\boldsymbol x(\cdot\wedge t) - \boldsymbol x_0(\cdot\wedge t_0)\|_\infty\big) \\
&\leq \ \kappa_\infty^{(t_0,\boldsymbol x_0)}(t,\boldsymbol x) \ \leq \ \|\boldsymbol x(\cdot\wedge t) - \boldsymbol x_0(\cdot\wedge t_0)\|_\infty, \notag
\end{align}
for all $(t,\boldsymbol x)\in\boldsymbol\Lambda$. In particular, it holds that $\kappa_\infty^{(t_0,\boldsymbol x_0)}\geq0$.
\end{enumerate}
\end{lem}
\begin{proof}
See Appendix \ref{AppVar}, Section \ref{AppVar1}.
\end{proof}

We now address the problem of smoothing the map $\hat\kappa_\infty^{(t_0,\boldsymbol x_0)}$ in the horizontal direction. This is required by the fact that the presence of $1_{[t,T]}$ in the definition of $\hat\kappa_\infty^{(t_0,\boldsymbol x_0)}$ is an obstruction to horizontal regularity, therefore a further convolution in the time variable $t$ is needed. The latter convolution also provides the continuity on $\boldsymbol{\hat\Lambda}$ (notice that the map $(t,\boldsymbol{\hat x})\mapsto\hat\kappa_\infty^{(t_0,\boldsymbol x_0)}(t,\boldsymbol{\hat x})$ is not continuous on $\boldsymbol{\hat\Lambda}$, see Remark \ref{R:Example}).

We perform such a horizontal smoothing to $\hat\kappa_\infty^{(t_0,\boldsymbol x_0)}/(1+C_\zeta+\hat\kappa_\infty^{(t_0,\boldsymbol x_0)})$. We apply it to
such a map (rather than to $\hat\kappa_\infty^{(t_0,\boldsymbol x_0)}$ directly) in order to have bounded derivatives (see item 3 of Lemma \ref{L:Smoothing_rho_infty}). Moreover, we consider $1+C_\zeta+\hat\kappa_\infty^{(t_0,\boldsymbol x_0)}$ (instead of $1+\hat\kappa_\infty^{(t_0,\boldsymbol x_0)}$) in order to have a denominator greater than or equal to $1$ (this follows from inequality $\hat\kappa_\infty^{(t_0,\boldsymbol x_0)}\geq-C_\zeta$, see item 3 of Lemma \ref{L:rho_infty}). The precise form of the mollifier $\eta$ in \eqref{eta} is used to get explicit bounds on $\hat\chi_\infty^{(t_0,\boldsymbol x_0)}$ and its derivatives.

\begin{rem}[\cite{Gomoyunov}]\label{R:Example}
Notice that the map $(t,\boldsymbol{\hat x})\mapsto\hat\kappa_\infty^{(t_0,\boldsymbol x_0)}(t,\boldsymbol{\hat x})$ is not continuous on $\boldsymbol{\hat\Lambda}$. As a matter of fact, consider the following example. Take $d=1$, $T=2$, $t_0=0$, $\boldsymbol x_0\equiv0$, $t=1$, $\boldsymbol{\hat x}=1_{[1,2]}$. Then, it holds that
\begin{align*}
\hat\kappa_\infty^{(t_0,\boldsymbol x_0)}(t,\boldsymbol{\hat x}) \ &= \ \int_\R\big\|\boldsymbol{\hat x}(\cdot\wedge t) - \boldsymbol x_0(\cdot\wedge t_0) - \mathbf z\,1_{[t,T]}\big\|_\infty\,\zeta(\mathbf z)\,d\mathbf z - \int_\R|\mathbf z|\,\zeta(\mathbf z)\,d\mathbf z \\
&= \ \int_\R |1 - \mathbf z|\,\zeta(\mathbf z)\,d\mathbf z - \int_\R|\mathbf z|\,\zeta(\mathbf z)\,d\mathbf z.
\end{align*}
Now, take $\delta\in(0,1)$, then
\begin{align*}
\hat\kappa_\infty^{(t_0,\boldsymbol x_0)}(t+\delta,\boldsymbol{\hat x}) \ &= \ \int_\R\big\|\boldsymbol{\hat x}(\cdot\wedge(t+\delta)) - \boldsymbol x_0(\cdot\wedge t_0) - \mathbf z\,1_{[t+\delta,T]}\big\|_\infty\,\zeta(\mathbf z)\,d\mathbf z - \int_\R|\mathbf z|\,\zeta(\mathbf z)\,d\mathbf z \\
&= \ \int_\R \max\big\{1,|1 - \mathbf z|\big\}\,\zeta(\mathbf z)\,d\mathbf z - \int_\R|\mathbf z|\,\zeta(\mathbf z)\,d\mathbf z.
\end{align*}
In conclusion, we have
\begin{align*}
\big|\hat\kappa_\infty^{(t_0,\boldsymbol x_0)}(t+\delta,\boldsymbol{\hat x}) - \hat\kappa_\infty^{(t_0,\boldsymbol x_0)}(t,\boldsymbol{\hat x})\big| \ &= \ \int_\R \Big\{\max\big\{1,|1 - \mathbf z|\big\} - |1 - \mathbf z|\Big\}\,\zeta(\mathbf z)\,d\mathbf z \\
&= \ \int_0^2 \big(1 - |1 - \mathbf z|\big)\,\zeta(\mathbf z)\,d\mathbf z \ =: \ \eps_* \ > \ 0,
\end{align*}
where $\eps_*$ is a constant independent of $\delta$. This proves that $|\hat\kappa_\infty^{(t_0,\boldsymbol x_0)}(t+\delta,\boldsymbol{\hat x}) - \hat\kappa_\infty^{(t_0,\boldsymbol x_0)}(t,\boldsymbol{\hat x})|\not\rightarrow0$ as $\delta\rightarrow0^+$ and shows that $\hat\kappa_\infty^{(t_0,\boldsymbol x_0)}$ is not continuous on $\boldsymbol{\hat\Lambda}$.
\end{rem}

\begin{lem}\label{L:Smoothing_rho_infty}
Let $\eta\colon\R\rightarrow\R$ be given by
\begin{equation}\label{eta}
\eta(s) \ := \ s\,\textup{e}^{-s}, \qquad \forall\,s\in\R.
\end{equation}
For every fixed $(t_0,\boldsymbol x_0)\in\boldsymbol\Lambda$, let $\hat\kappa_\infty^{(t_0,\boldsymbol x_0)}$ be the map defined in Lemma \ref{L:rho_infty} and define the map $\hat\chi_\infty^{(t_0,\boldsymbol x_0)}\colon\boldsymbol{\hat\Lambda}\rightarrow\R$ as
\[
\hat\chi_\infty^{(t_0,\boldsymbol x_0)}(t,\boldsymbol{\hat x}) \ := \ \int_0^{+\infty} \frac{\hat\kappa_\infty^{(t_0,\boldsymbol x_0)}\big((t + s)\wedge T,\boldsymbol{\hat x}(\cdot\wedge t)\big)}{1 + C_\zeta + \hat\kappa_\infty^{(t_0,\boldsymbol x_0)}\big((t + s)\wedge T,\boldsymbol{\hat x}(\cdot\wedge t)\big)}\,\eta(s)\,ds,
\]
for all $(t,\boldsymbol{\hat x})\in\boldsymbol{\hat\Lambda}$, with $C_\zeta$ as in \eqref{C_zeta}, where we recall that $1+C_\zeta+\hat\kappa_\infty^{(t_0,\boldsymbol x_0)}\geq1$ (see item 3 of Lemma \ref{L:rho_infty}). Moreover, let $\chi_\infty^{(t_0,\boldsymbol x_0)}\colon\boldsymbol\Lambda\rightarrow\R$ be given by
\begin{equation}\label{gamma}
\chi_\infty^{(t_0,\boldsymbol x_0)}(t,\boldsymbol x) \ := \ \hat\chi_\infty^{(t_0,\boldsymbol x_0)}(t,\boldsymbol x), \qquad \forall\,(t,\boldsymbol x)\in\boldsymbol\Lambda.
\end{equation}
Then, the following properties hold.
\begin{enumerate}[\upshape 1)]
\item For every $(t,\boldsymbol{\hat x})\in\boldsymbol{\hat\Lambda}$, the horizontal and vertical derivatives of first and second-order of $\hat\chi_\infty^{(t_0,\boldsymbol x_0)}$ at $(t,\boldsymbol{\hat x})$ (namely $\partial^H_t\hat\chi_\infty^{(t_0,\boldsymbol x_0)}(t,\boldsymbol{\hat x})$, $\partial^V_{x_i}\hat\chi_\infty^{(t_0,\boldsymbol x_0)}(t,\boldsymbol{\hat x})$ and $\partial^V_{x_i x_j}\hat\chi_\infty^{(t_0,\boldsymbol x_0)}(t,\boldsymbol{\hat x})$, for every $i,j=1,\ldots,d$) exist.
\item $\hat\chi_\infty^{(t_0,\boldsymbol x_0)}\in\boldsymbol C^{1,2}(\boldsymbol{\hat\Lambda})$ and the map $\big((t_0,\boldsymbol x_0),(t,\boldsymbol{\hat x})\big)\mapsto\hat\chi_\infty^{(t_0,\boldsymbol x_0)}(t,\boldsymbol{\hat x})$ is continuous on $\boldsymbol\Lambda\times\boldsymbol{\hat\Lambda}$.
\item The horizontal derivative of $\hat\chi_\infty^{(t_0,\boldsymbol x_0)}$ is bounded by the constant $\frac{2}{\textup{e}}$; the first-order vertical derivatives of $\hat\chi_\infty^{(t_0,\boldsymbol x_0)}$ are bounded by the constant $1+C_\zeta$; the second-order vertical derivatives of $\hat\chi_\infty^{(t_0,\boldsymbol x_0)}$ are bounded by the constant $(1+C_\zeta)\Big(\sqrt{\frac{2}{\pi}}+2\Big)$.
\item For every $(t,\boldsymbol x)\in\boldsymbol\Lambda$,
\begin{align}\label{Ineq12_bis}
&\alpha_d\,\frac{\|\boldsymbol x(\cdot\wedge t) - \boldsymbol x_0(\cdot\wedge t_0)\|_\infty^{d+1}\wedge\|\boldsymbol x(\cdot\wedge t) - \boldsymbol x_0(\cdot\wedge t_0)\|_\infty}{1 + C_\zeta + \|\boldsymbol x(\cdot\wedge t) - \boldsymbol x_0(\cdot\wedge t_0)\|_\infty} \\
&\hspace{5cm}\leq \ \chi_\infty^{(t_0,\boldsymbol x_0)}(t,\boldsymbol x) \ \leq \ \|\boldsymbol x(\cdot\wedge t) - \boldsymbol x_0(\cdot\wedge t_0)\|_\infty\wedge1, \notag
\end{align}
with the same constant $\alpha_d$ as in \eqref{Ineq12}. In particular, it holds that $\chi_\infty^{(t_0,\boldsymbol x_0)}\geq0$.
\end{enumerate}
\end{lem}
\begin{proof}
See Appendix \ref{AppVar}, Section \ref{AppVar2}.
\end{proof}

In conclusion, by Lemma \ref{L:Smoothing_rho_infty} it follows that the map $\rho_\infty\colon\boldsymbol\Lambda\times\boldsymbol\Lambda\rightarrow[0,+\infty)$ given by
\begin{equation}\label{rho}
\rho_\infty\big((t,\boldsymbol x),(t_0,\boldsymbol x_0)\big) \ = \ |t - t_0|^2 + \chi_\infty^{(t_0,\boldsymbol x_0)}(t,\boldsymbol x), \qquad \forall\,(t,\boldsymbol x),(t_0,\boldsymbol x_0)\in\boldsymbol\Lambda,
\end{equation}
with $\chi_\infty$ as in \eqref{gamma}, is a
gauge-type function, which is also smooth as a map of the first pair, namely $(t,\boldsymbol x)\mapsto\rho_\infty((t,\boldsymbol x),(t_0,\boldsymbol x_0))$, and with bounded derivatives.

We now apply
Theorem \ref{T:Ekeland} to the smooth gauge-type function $\rho_\infty$ with bounded derivatives defined by \eqref{rho}, taking $\delta_0:=\delta>0$ and $\delta_n:=\delta/2^n$, for every $n\geq1$. 

\begin{thm}[Smooth variational principle on $\boldsymbol\Lambda$]\label{T:EkelandBis} 
Let $\delta>0$ and $G\colon\boldsymbol\Lambda\rightarrow\R$ be an upper semicontinuous map, bounded from above. For every $\eps>0$, let $(t_0,\boldsymbol x_0)\in\boldsymbol\Lambda$ satisfy
\[
\sup G - \eps \ \leq \ G(t_0,\boldsymbol x_0).
\]
Then, there exists a sequence $\{(t_n,\boldsymbol x_n)\}_{n\geq1}\subset\boldsymbol\Lambda$ which converges to some $(\bar t,\boldsymbol{\bar x})\in\boldsymbol\Lambda$ fulfilling the properties below.
\begin{enumerate}[\upshape i)]
\item $\rho_\infty((\bar t,\boldsymbol{\bar x}),(t_n,\boldsymbol x_n))\leq\frac{\eps}{2^n\delta}$, for every $n\geq0$.
\item $G(t_0,\boldsymbol x_0)\leq G(\bar t,\boldsymbol{\bar x})-\delta\varphi_\eps(t,\boldsymbol x)$, where the map $\varphi_\eps\colon\boldsymbol\Lambda\rightarrow[0,+\infty)$ is defined as
\[
\varphi_\eps(t,\boldsymbol x) \ := \ \sum_{n=0}^{+\infty} \frac{1}{2^n}\,\rho_\infty\big((t,\boldsymbol x),(t_n,\boldsymbol x_n)\big), \qquad \forall\,(t,\boldsymbol x)\in\boldsymbol\Lambda.
\]
\item For every $(t,\boldsymbol x)\neq(\bar t,\boldsymbol{\bar x})$, $G(t,\boldsymbol x) - \delta\,\varphi_\eps(t,\boldsymbol x)<G(\bar t,\boldsymbol{\bar x}) - \delta\,\varphi_\eps(\bar t,\boldsymbol{\bar x})$.
\end{enumerate}
Finally, the map $\varphi_\eps$ satisfies the following properties.
\begin{enumerate}[\upshape 1)]
\item $\varphi_\eps\in\boldsymbol C^{1,2}(\boldsymbol\Lambda)$ and is bounded.
\item $\partial_t^H\varphi_\eps$ is bounded by the constant
  $2\left(2\hspace{0.5mm}T+\frac{2}{\textup{e}}\right)$.
\item For every $i,j=1,\ldots,d$, $\partial^V_{x_i}\varphi_\eps$ is bounded by the constant $2\,(1+C_\zeta)$ and $\partial^V_{x_i x_j}\varphi_\eps$ is bounded by the constant $2\,(1+C_\zeta)\Big(\sqrt{\frac{2}{\pi}}+2\Big)$.
\end{enumerate}
\end{thm}
\begin{proof}
Items i)-ii)-iii) follow directly from Theorem \ref{T:Ekeland}, while items 1)-2)-3) follow easily from items 2)-3)-4) of Lemma \ref{L:Smoothing_rho_infty}.
\end{proof}

\section{Crandall-Lions (path-dependent) viscosity solutions}
\label{S:Visc}

\subsection{Viscosity solutions}

In the present section we consider the second-order path-dependent partial differential equation
\begin{equation}\label{PPDE_general}
\hspace{-5mm}\begin{cases}
\vspace{2mm}
\partial_t^H u(t,\boldsymbol x) = F\big(t,\boldsymbol x,u(t,\boldsymbol x),\partial^V_{\boldsymbol x}u(t,\boldsymbol x),\partial^V_{\boldsymbol x\boldsymbol x} u(t,\boldsymbol x)\big), &(t,\boldsymbol x)\in[0,T)\times C([0,T];\R^d), \\
u(T,\boldsymbol x) = \xi(\boldsymbol x), &\,\boldsymbol x\in C([0,T];\R^d),
\end{cases}
\end{equation}
with $F\colon[0,T]\times C([0,T];\R^d)\times\R^d\times\mathcal S(d)\rightarrow\R$ and $\xi\colon C([0,T];\R^d)\rightarrow\R$, where $\mathcal S(d)$ is the set of symmetric $d\times d$ matrices.
\begin{defn}
We denote by $\boldsymbol C_{\textup{pol}}^{1,2}(\boldsymbol\Lambda)$ the set of $\varphi\in\boldsymbol C^{1,2}(\boldsymbol\Lambda)$ such that $\varphi$, $\partial_t^H\varphi$, $\partial^V_{\boldsymbol x}\varphi$, $\partial^V_{\boldsymbol x\boldsymbol x}\varphi$ satisfy a polynomial growth condition.
\end{defn}

\begin{defn}\label{D:Visc}
We say that an upper semicontinuous map $u\colon\boldsymbol\Lambda\rightarrow\R$ is a (\textbf{path-dependent}) \textbf{viscosity subsolution} of equation \eqref{PPDE_general} if the following holds.
\begin{itemize}
\item $u(T,\boldsymbol x)\leq\xi(\boldsymbol x)$, for all $\boldsymbol x\in C([0,T];\R^d)$;
\item for any $(t,\boldsymbol x)\in[0,T)\times C([0,T];\R^d)$ and $\varphi\in\boldsymbol C_{\textup{pol}}^{1,2}(\boldsymbol\Lambda)$, satisfying
\[
(u-\varphi)(t,\boldsymbol x)=\sup_{(t',\boldsymbol x')\in\boldsymbol\Lambda}(u-\varphi)(t',\boldsymbol x'),
\]
we have
\[
- \partial_t^H\varphi(t,\boldsymbol x) + F\big(t,\boldsymbol x,u(t,\boldsymbol x),\partial^V_{\boldsymbol x}\varphi(t,\boldsymbol x),\partial^V_{\boldsymbol x\boldsymbol x}\varphi(t,\boldsymbol x)\big) \ \leq \ 0.
\]
\end{itemize}

\noindent We say that a lower semicontinuous map $u\colon\boldsymbol\Lambda\rightarrow\R$ is a (\textbf{path-dependent}) \textbf{viscosity supersolution} of equation \eqref{PPDE_general} if:
\begin{itemize}
\item $u(T,\boldsymbol x)\geq\xi(\boldsymbol x)$, for all $\boldsymbol x\in C([0,T];\R^d)$;
\item for any $(t,\boldsymbol x)\in[0,T)\times C([0,T];\R^d)$ and $\varphi\in\boldsymbol C_{\textup{pol}}^{1,2}(\boldsymbol\Lambda)$, satisfying:
\[
(u-\varphi)(t,\boldsymbol x)=\inf_{(t',\boldsymbol x')\in\boldsymbol\Lambda}(u-\varphi)(t',\boldsymbol x'),
\]
we have
\[
- \partial_t^H\varphi(t,\boldsymbol x) + F\big(t,\boldsymbol x,u(t,\boldsymbol x),\partial^V_{\boldsymbol x}\varphi(t,\boldsymbol x),\partial^V_{\boldsymbol x\boldsymbol x}\varphi(t,\boldsymbol x)\big) \ \geq \ 0.
\]
\end{itemize}

\noindent We say that a continuous map $u\colon\boldsymbol\Lambda\rightarrow\R$ is a (\textbf{path-dependent}) \textbf{viscosity solution} of equation \eqref{PPDE_general} if $u$ is both a (path-dependent) viscosity subsolution and a (path-dependent) viscosity supersolution of \eqref{PPDE_general}.
\end{defn}

\subsection{Path-dependent heat equation}

In the present section we focus on the path-dependent heat equation, namely when $F(t,\boldsymbol x,r,p,M)=-\frac{1}{2}\text{tr}[M]$
\begin{equation}\label{PPDE}
\begin{cases}
\vspace{2mm}
\partial_t^H u(t,\boldsymbol x) + \dfrac{1}{2} \textup{tr}\big[\partial^V_{\boldsymbol x\boldsymbol x} u(t,\boldsymbol x)\big] \ = \ 0, \qquad\qquad &(t,\boldsymbol x)\in[0,T)\times C([0,T];\R^d), \\
u(T,\boldsymbol x) \ = \ \xi(\boldsymbol x), &\,\boldsymbol x\in C([0,T];\R^d).
\end{cases}
\end{equation}
In the sequel we denote
\begin{equation}\label{23bis}
\mathcal L u(t,\boldsymbol x) \ := \ \partial_t^H u(t,\boldsymbol x) + \frac{1}{2} \textup{tr}[\partial^V_{\boldsymbol x\boldsymbol x} u(t,\boldsymbol x)].
\end{equation}
 On the terminal condition $\xi$, we impose the assumption

\vspace{2mm}

\noindent{\bf (A)} \emph{The function $\xi\colon C([0,T];\R^d)\rightarrow\R$ is continuous and bounded.}

\subsubsection{Existence}

\noindent The ``candidate solution'' to equation \eqref{PPDE} is
\begin{equation}\label{v}
v(t,\boldsymbol x) \ := \ \E\big[\xi(\boldsymbol W^{t,\boldsymbol x})\big], \qquad \text{for all }(t,\boldsymbol x)\in\boldsymbol\Lambda,
\end{equation}
where $\boldsymbol W=(\boldsymbol W_s)_{s\in[0,T]}$ is a $d$-dimensional Brownian motion on some complete probability space $(\Omega,{\mathcal F},\P)$, and the stochastic process $\boldsymbol W^{t,\boldsymbol x}=(\boldsymbol W_s^{t,\boldsymbol x})_{s\in[0,T]}$ is given by
\begin{equation}\label{W^t,x}
\boldsymbol W_s^{t,\boldsymbol x} \ := \
\begin{cases}
\boldsymbol x(s), & s\leq t, \\
\boldsymbol x(t) + \boldsymbol W_s - \boldsymbol W_t, \qquad\quad & s>t.
\end{cases}
\end{equation}
\begin{rem} \label{R421}
  The boundedness of $\xi$ in Assumption {\bf (A)}
   will be used in the proof of (the comparison) Theorem \ref{T:Comparison}. On the other hand, the proof that the function $v$ in \eqref{v} is continuous and is a viscosity solution of equation \eqref{PPDE} (see the proof of Theorem \ref{T:Existence}) holds under weaker growth condition on $\xi$ (for instance, $\xi$ having polynomial growth).	
\end{rem}

\begin{thm}\label{T:Existence}
  Under Assumption {\bf (A)}, the function $v$ in \eqref{v} is continuous and bounded.
  Moreover, $v$ is a (path-dependent) viscosity solution of equation \eqref{PPDE}.
\end{thm}
\begin{proof}
  \textsc{Step I.} \emph{Continuity of $v$.} Given $(t,\boldsymbol x),(t',\boldsymbol x')
  \in\boldsymbol\Lambda$, with $t\leq t'$, from \eqref{W^t,x} we have
\[ 
\boldsymbol W_s^{t,\boldsymbol x} - \boldsymbol W_s^{t',\boldsymbol x'} \ = \
\begin{cases}
\boldsymbol x(s) - \boldsymbol x'(s), & s\leq t, \\
\boldsymbol x(t) - \boldsymbol x'(s) + \boldsymbol W_s - \boldsymbol W_t, \qquad & t<s\leq t', \\
\boldsymbol x(t) - \boldsymbol x'(t') + \boldsymbol W_{t'} - \boldsymbol W_t, \qquad\quad & s>t'.
\end{cases}
\]
Hence
\begin{align*}
\sup_{s\in[0,T]}\big|\boldsymbol W_s^{t,\boldsymbol x} - \boldsymbol W_s^{t',\boldsymbol x'}\big| \ &\leq \ \|\boldsymbol x(\cdot\wedge t) - \boldsymbol x'(\cdot\wedge t')\|_\infty + \sup_{s\in[t,t']}\big|\boldsymbol W_s - \boldsymbol W_t\big| \\
&\leq \ \|\boldsymbol x(\cdot\wedge t) - \boldsymbol x'(\cdot\wedge t')\|_\infty + \sum_{i=1}^d \sup_{s\in[t,t']}\big|W_s^i - W_t^i\big|,
\end{align*}
where $\boldsymbol W=(W^1,\ldots,W^d)$ and the second inequality follows from the fact the Euclidean norm on $\R^d$ is estimated by the $1$-norm. By the reflection principle, $\sup_{s\in[t,t']}|W_s^i-W_t^i|$ has the same law as $|W_{t'}^i-W_t^i|$, therefore
\begin{align*}
\E\Big[\sup_{s\in[0,T]}\big|\boldsymbol W_s^{t,\boldsymbol x} - \boldsymbol W_s^{t',\boldsymbol x'}\big|\Big] \ &\leq \ \|\boldsymbol x(\cdot\wedge t) - \boldsymbol x'(\cdot\wedge t')\|_\infty + \sum_{i=1}^d\E\big[|W_{t'}^i - W_t^i|\big] \\
&= \ \|\boldsymbol x(\cdot\wedge t) - \boldsymbol x'(\cdot\wedge t')\|_\infty + d\sqrt{\frac{2}{\pi}}\sqrt{|t-t'|}.
\end{align*}
Then, since $\xi$ is bounded and continuous, the continuity of $v$ follows from the above estimate together with the Lebesgue dominated convergence theorem.

\vspace{2mm}

\noindent\textsc{Step II.} \emph{$v$ is a viscosity solution of equation \eqref{PPDE}.} For every $t\in[0,T]$, let $\F^t=({\mathcal F}_s^t)_{s\in[t,T]}$ be the completion of the filtration generated by $(\boldsymbol W_s-\boldsymbol W_t)_{s\in[t,T]}$. Now, fix $(t,\boldsymbol x)\in\boldsymbol\Lambda$ and $t'\in[t,T]$. We first prove that 
\begin{equation}\label{DPP_v}
v(t,\boldsymbol x) \ = \ \E\big[v\big(t',\boldsymbol W^{t,\boldsymbol x}\big)\big].
\end{equation}
To this end, we begin noticing that by \eqref{W^t,x} we have
\begin{equation} \label{E23}
\boldsymbol W_\cdot^{t,\boldsymbol x} \ = \ \boldsymbol x(\cdot\wedge t) + \boldsymbol W_{\cdot\vee t} - \boldsymbol W_t.
\end{equation}
Therefore
\begin{equation}\label{v_proof}
v(t,\boldsymbol x) \ = \ \E[\xi(\boldsymbol x(\cdot\wedge t) + \boldsymbol W_{\cdot\vee t} - \boldsymbol W_t)].	
\end{equation}
Now, notice that, by \eqref{E23},
\[
\boldsymbol W_\cdot^{t',\boldsymbol W^{t,\boldsymbol x}} \ = \ \boldsymbol W_{\cdot\wedge t'}^{t,\boldsymbol x} + \boldsymbol W_{\cdot\vee t'} - \boldsymbol W_{t'} \ = \ \boldsymbol W_\cdot^{t,\boldsymbol x}.
\]
This proves the flow property $\boldsymbol W_\cdot^{t,\boldsymbol x}=\boldsymbol W_\cdot^{t',\boldsymbol W^{t,\boldsymbol x}}$. Then, by the freezing lemma
for conditional expectation
 and formula \eqref{v_proof}, we obtain
\begin{align*}
v(t,\boldsymbol x) \ = \ \E\big[\xi(\boldsymbol W^{t,\boldsymbol x})\big] \ &= \ \E\big[\xi\big(\boldsymbol W^{t',\boldsymbol W^{t,\boldsymbol x}}\big)\big] \ = \ \E\big[\xi\big(\boldsymbol W_{\cdot\wedge t'}^{t,\boldsymbol x} + \boldsymbol W_{\cdot\vee t'} - \boldsymbol W_{t'}\big)\big] \\
&= \ \E\big[\E\big[\xi\big(\boldsymbol W_{\cdot\wedge t'}^{t,\boldsymbol x} + \boldsymbol W_{\cdot\vee t'} - \boldsymbol W_{t'}\big)\big|{\mathcal F}_{t'}^t\big]\big] \ = \ \E\big[v\big(t',\boldsymbol W_{\cdot\wedge t'}^{t,\boldsymbol x}\big)\big].
\end{align*}
Finally, recalling that $v$ is non-anticipative we deduce that $v(t',\boldsymbol W_{\cdot\wedge t'}^{t,\boldsymbol x})=v(t',\boldsymbol W^{t,\boldsymbol x})$, which
 concludes the proof of formula \eqref{DPP_v}.

Let us now prove that $v$ is a viscosity solution of equation \eqref{PPDE}. We only prove the viscosity subsolution property, as the supersolution property can be proved in a similar way. We proceed along the same lines as in the proof of the subsolution property in Theorem 3.66 of \cite{fabbrigozziswiech}. Let $(t,\boldsymbol x)\in[0,T)\times C([0,T];\R^d)$ and $\varphi\in\boldsymbol C_{\textup{pol}}^{1,2}(\boldsymbol\Lambda)$, satisfying:
\[
(v-\varphi)(t,\boldsymbol x)=\sup_{(t',\boldsymbol x')\in\boldsymbol\Lambda}(v-\varphi)(t',\boldsymbol x').
\]
We suppose that $(v-\varphi)(t,\boldsymbol x)=0$ (if this is not the case, we replace $\varphi$ by $\psi(\cdot,\cdot):=\varphi(\cdot,\cdot)+v(t,\boldsymbol x)-\varphi(t,\boldsymbol x)$). Take
\begin{equation}\label{v_proof_2}
\varphi(t,\boldsymbol x) \ = \ v(t,\boldsymbol x) \ = \ \E\big[v\big(t+\eps,\boldsymbol W^{t,\boldsymbol x}\big)\big] \ \leq \ \E\big[\varphi\big(t+\eps,\boldsymbol W^{t,\boldsymbol x}\big)\big],
\end{equation}
where the latter inequality follows from the fact that $\sup(v-\varphi)=0$, so that $v\leq\varphi$ on $\boldsymbol\Lambda$. Notice that the last expectation in \eqref{v_proof_2} is finite, as $\varphi$ has polynomial growth. Now, by the functional It\^o formula \eqref{Ito_formula}, we have
\[
\varphi(t+\eps,\boldsymbol W^{t,\boldsymbol x}) \ = \ \varphi(t,\boldsymbol x) + \int_t^{t+\eps} \mathcal L\varphi(s,\boldsymbol W^{t,\boldsymbol x})\,ds + \sum_{i=1}^d \int_t^{t+\eps} \partial^V_{x_i}\varphi(s,\boldsymbol W^{t,\boldsymbol x})\,d W_s^i,
\]
where $\mathcal L$ was defined in \eqref{23bis}.
Since $\partial^V_{x_i}\varphi$ has polynomial growth, the corresponding stochastic integral is a martingale. Then, plugging the above formula into \eqref{v_proof_2} and dividing by $\eps$, we find
\[
-\,\E\bigg[\frac{1}{\eps} \int_t^{t+\eps} \mathcal L\varphi(s,\boldsymbol W^{t,\boldsymbol x})\,ds\bigg] \ \leq \ 0.
\]
Letting $\eps\rightarrow0^+$, we conclude that
\[
-\mathcal L\varphi(t,\boldsymbol x) \ \leq \ 0,
\]
which proves the viscosity subsolution property.
\end{proof}

\subsubsection{Comparison theorem and uniqueness}

\begin{thm}\label{T:Comparison}
Suppose that Assumption {\bf (A)} holds. Let $u,w\colon\boldsymbol\Lambda\rightarrow\R$ be respectively upper and lower semicontinuous, satisfying
\[
\sup u \ < \ +\infty, \qquad\qquad
\inf w \ > \ -\infty.
\]
Suppose that $u$ $($resp. $w$$)$ is a (path-dependent) viscosity subsolution $($resp. supersolution$)$ of equation \eqref{PPDE}. Then $u\leq w$ on $\boldsymbol\Lambda$.
\end{thm}
\begin{proof}
The proof consists in showing that $u\leq v$ and $v\leq w$ on $\boldsymbol\Lambda$ (with $v$ given by \eqref{v}), from which we immediately deduce the claim. In what follows, we only report the proof of the inequality $u\leq v$, as the other inequality (that is $v\leq w$) can be deduced from the first one replacing $u$, $v$, $\xi$ with $-w$, $-v$, $-\xi$, respectively.

\vspace{1mm}

We proceed by contradiction and assume that $\sup(u-v)>0$. Then, there exists $(t_0,\boldsymbol x_0)\in\boldsymbol\Lambda$ such that
\[
(u - v)(t_0,\boldsymbol x_0) \ > \ 0.
\]
Notice that $t_0<T$, since $u(T,\cdot)\leq\xi(\cdot)=v(T,\cdot)$. We split the rest of the proof into five steps.

\vspace{1mm}

\noindent\textsc{Step I}. Let $\{\xi_N\}_N$ be the sequence given by Lemma \ref{L:Smoothing}. Since $\xi$ is bounded, we have that $\xi_N$ is bounded uniformly with respect to $N$. Now, denote
\[
v_N(t,\boldsymbol x) \ := \ \E\big[\xi_N(\boldsymbol W^{t,\boldsymbol x})\big], \qquad \text{for all }(t,\boldsymbol x)\in\boldsymbol\Lambda.
\]
Then, $v_N$ is bounded uniformly with respect to $N$. Moreover, by Lemma \ref{L:Classical} it follows that, for every $N$, $v_N\in\boldsymbol C^{1,2}(\boldsymbol\Lambda)$ and is a classical (smooth) solution of equation \eqref{PPDE} with terminal condition $\xi_N$. Finally, recalling from Lemma \ref{L:Smoothing} that $\{\xi_N\}_N$ converges pointwise to $\xi$ as $N\rightarrow+\infty$, it follows from the Lebesgue dominated convergence theorem that $\{v_N\}_N$ converges pointwise to $v$ as $N\rightarrow+\infty$. Then, we notice that there exists $N_0\in\N$ such that
\begin{equation}\label{Contradiction}
(u - v_{N_0})(t_0,\boldsymbol x_0) \ > \ 0.
\end{equation}
We also suppose that (possibly enlarging $N_0$)
\begin{equation}\label{xi-xi_N_0}
|\xi(\boldsymbol x_0) - \xi_{N_0}(\boldsymbol x_0)| \ \leq \ \frac{1}{2}(u - v_{N_0})(t_0,\boldsymbol x_0).
\end{equation}

\vspace{1mm}

\noindent\textsc{Step II}. For every $\lambda>0$, we set
\[
u^\lambda(t,\boldsymbol x) := \text{e}^{\lambda t}u(t,\boldsymbol x), \; \xi^\lambda(\boldsymbol x) := \text{e}^{\lambda T}\xi(\boldsymbol x), \; v_{N_0}^\lambda(t,\boldsymbol x) := \text{e}^{\lambda t}v_{N_0}(t,\boldsymbol x), \; \xi_{N_0}^\lambda(\boldsymbol x) := \text{e}^{\lambda T}\xi_{N_0}(\boldsymbol x).
\]
for all $(t,\boldsymbol x)\in\boldsymbol\Lambda$. Notice that $u^\lambda$ is a (path-dependent) viscosity subsolution of the path-dependent partial differential equation
\begin{equation}\label{PPDE_lambda}
\begin{cases}
\vspace{2mm}
\partial_t^H u^\lambda(t,\boldsymbol x) + \frac{1}{2} \textup{tr}[\partial^V_{\boldsymbol x} u^\lambda(t,\boldsymbol x)] = \lambda\,u^\lambda(t,\boldsymbol x), \quad &(t,\boldsymbol x)\in[0,T)\times C([0,T];\R^d), \\
u^\lambda(T,\boldsymbol x) \ = \ \xi^\lambda(\boldsymbol x), &\,\boldsymbol x\in C([0,T];\R^d).
\end{cases}
\end{equation}
Similarly, $v_{N_0}^\lambda$ is a classical (smooth) solution of equation \eqref{PPDE_lambda} with $\xi^\lambda$ replaced by $\xi_{N_0}^\lambda$. We finally notice that by \eqref{Contradiction} we have
\[
(u^\lambda - v_{N_0}^\lambda)(t_0,\boldsymbol x_0) \ > \ 0.
\]
So, in particular,
\begin{equation}\label{ContradictionN_0}
\sup (u^\lambda - v_{N_0}^\lambda) - \eps \ = \ (u^\lambda - v_{N_0}^\lambda)(t_0,\boldsymbol x_0) \ \leq \ \sup (u^\lambda - v_{N_0}^\lambda),
\end{equation}
where $\eps:=\sup (u^\lambda - v_{N_0}^\lambda) - (u^\lambda - v_{N_0}^\lambda)(t_0,\boldsymbol x_0)$.

\vspace{1mm}

\noindent\textsc{Step III.} Notice that $u^\lambda-v_{N_0}^\lambda$ is upper semicontinuous and bounded from above. Then, by \eqref{ContradictionN_0} and the smooth variational principle (Theorem \ref{T:EkelandBis}) with $G = u^\lambda - v_{N_0}^\lambda$,
we deduce that for every $\delta>0$ there exists a sequence $\{(t_n,\boldsymbol x_n)\}_{n\geq1}\subset\boldsymbol\Lambda$ converging to some $(\bar t,\boldsymbol{\bar x})\in\boldsymbol\Lambda$ (possibly depending
on $\eps, \delta, \lambda, N_0$) such that the following holds.
\begin{enumerate}[\upshape i)]
\item $\rho_\infty((t_n,\boldsymbol x_n),(\bar t,\boldsymbol{\bar x}))\leq\frac{\eps}{2^n\delta}$, for every $n\geq0$, where $\rho_\infty$ is the smooth gauge-type function with bounded derivatives defined by \eqref{rho}.
\item $(u^\lambda-v_{N_0}^\lambda)(t_0,\boldsymbol x_0)\leq \big(u^\lambda-(v_{N_0}^\lambda+\delta\varphi_\eps)\big)(\bar t,\boldsymbol{\bar x})$, where
\[
\varphi_\eps(t,\boldsymbol x) \ := \ \sum_{n=0}^{+\infty} \frac{1}{2^n}\,\rho_\infty\big((t,\boldsymbol x),(t_n,\boldsymbol x_n)\big) \qquad \forall\,(t,\boldsymbol x)\in\boldsymbol\Lambda.
\]
\item It holds that
\begin{equation}\label{maximum_proof}
\big(u^\lambda - (v_{N_0}^\lambda + \delta\varphi_\eps)\big)(\bar t,\boldsymbol{\bar x}) \ = \ \sup_{(t,\boldsymbol x)\in\boldsymbol\Lambda}\big(u^\lambda - (v_{N_0}^\lambda + \delta\varphi_\eps)\big)(t,\boldsymbol x).
\end{equation}
\end{enumerate}
We also recall from Theorem \ref{T:EkelandBis}
 that $\varphi_\eps$ satisfies the following properties.
\begin{enumerate}[\upshape 1)]
\item $\varphi_\eps\in\boldsymbol C^{1,2}(\boldsymbol\Lambda)$ and is bounded.
\item $\left|\partial_t^H\varphi_\eps(t,\boldsymbol x)\right|\leq2\left(2\hspace{0.5mm}T+\frac{2}{\textup{e}}\right)$, for every $(t,\boldsymbol x)\in[0,T)\times C([0,T];\R^d)$.
\item For every $i,j=1,\ldots,d$, $\partial^V_{x_i}\varphi_\eps$ is bounded by the constant $2\,(1+C_\zeta)$ and $\partial^V_{x_i x_j}\varphi_\eps$ is bounded by the constant $2\,(1+C_\zeta)\Big(\sqrt{\frac{2}{\pi}}+2\Big)$.
\end{enumerate}
In particular, $\varphi_\eps\in\boldsymbol C_{\textup{pol}}^{1,2}(\boldsymbol\Lambda)$.

\vspace{1mm}

\noindent\textsc{Step IV}. We prove below that $\bar t<T$. As a matter of fact, by item ii) of \textsc{Step III} we have
\begin{equation} \label{E11}
\big(u^\lambda-(v_{N_0}^\lambda+\delta\varphi_\eps)\big)(\bar t,\boldsymbol{\bar x}) \ \geq \ (u^\lambda-v_{N_0}^\lambda)(t_0,\boldsymbol x_0).
\end{equation}
On the other hand, if $\bar t=T$ we obtain
\begin{equation} \label{E12}
\big(u^\lambda-(v_{N_0}^\lambda+\delta\varphi_\eps)\big)(\bar t,\boldsymbol{\bar x}) \ = \ \text{e}^{\lambda T}(\xi(\boldsymbol{\bar x}) - \xi_{N_0}(\boldsymbol{\bar x})) - \delta\varphi_\eps(T,\boldsymbol{\bar x}) \ \leq \ \text{e}^{\lambda T}(\xi(\boldsymbol{\bar x}) - \xi_{N_0}(\boldsymbol{\bar x})),
\end{equation}
where the latter inequality comes from the fact that $\varphi_\eps\geq0$. Hence,
by \eqref{E11} and \eqref{E12} we get
\[
\text{e}^{\lambda t_0}(u - v_{N_0})(t_0,\boldsymbol x_0) \ \leq \ \text{e}^{\lambda T}(\xi(\boldsymbol{\bar x}) - \xi_{N_0}(\boldsymbol{\bar x})).
\]
Letting $\eps\rightarrow0$, it follows from item i) above with
$n= 0$ and \eqref{Ineq12_bis}
that $d_\infty((\bar t,\boldsymbol{\bar x}),(t_0,\boldsymbol x_0))\rightarrow0$. Therefore, letting $\eps\rightarrow0$ in the previous inequality, we obtain
\[
\text{e}^{\lambda t_0}(u - v_{N_0})(t_0,\boldsymbol x_0) \ \leq \ \text{e}^{\lambda T}(\xi(\boldsymbol x_0) - \xi_{N_0}(\boldsymbol x_0)).
\]
By \eqref{xi-xi_N_0}, we end up with $\text{e}^{\lambda t_0}\leq\frac{1}{2}\,\text{e}^{\lambda T}$. Letting $\lambda\rightarrow0$, we find a contradiction.

\vspace{1mm}
\noindent\textsc{Step V}.
Here again $\lambda > 0$ is fixed.
By \eqref{maximum_proof} and the definition of viscosity subsolution of \eqref{PPDE_lambda}
applied to $u^\lambda$
at the point $(\bar t,\boldsymbol{\bar x})$ with test function $v_{N_0}^\lambda+\delta\varphi_\eps$, we obtain
\[
- \mathcal L(v_{N_0}^\lambda + \delta\varphi_\eps)(\bar t,\boldsymbol{\bar x}) + \lambda\,u^\lambda(\bar t,\boldsymbol{\bar x}) \ \leq \ 0.
\]
Recalling that $v_{N_0}^\lambda$ is a classical (smooth) solution of equation \eqref{PPDE_lambda} with $\xi^\lambda$ replaced by $\xi_{N_0}^\lambda$, we find
\[
\lambda\,(u^\lambda - v_{N_0}^\lambda)(\bar t,\boldsymbol{\bar x}) \ \leq \ \delta\,\mathcal L \varphi_\eps(\bar t,\boldsymbol{\bar x}).
\]
By item ii) in \textsc{Step III} (namely \eqref{E11}),
subtracting from both sides the quantity $\lambda\delta\varphi_\eps
(\bar t,\boldsymbol{\bar x}),$ we obtain
\[
\lambda\,(u^\lambda - v_{N_0}^\lambda)(t_0,\boldsymbol x_0) \ \leq \ \lambda\,\big(u^\lambda - (v_{N_0}^\lambda + \delta\varphi_\eps)\big)(\bar t,\boldsymbol{\bar x}) \ \leq \ \delta\,\mathcal L \varphi_\eps(\bar t,\boldsymbol{\bar x}) - \lambda\,\delta\,\varphi_\eps(\bar t,\boldsymbol{\bar x}).
\]
Recalling that $\varphi_\eps\geq0$, we see that
\[
\lambda\,(u^\lambda - v_{N_0}^\lambda)(t_0,\boldsymbol x_0) \ \leq \ \lambda\,\big(u^\lambda - (v_{N_0}^\lambda + \delta\varphi_\eps)\big)(\bar t,\boldsymbol{\bar x}) \ \leq \ \delta\,\mathcal L \varphi_\eps(\bar t,\boldsymbol{\bar x}).
\]
From items 2) and 3) above, it follows that $\mathcal L \varphi_\eps(\bar t,\boldsymbol{\bar x})$ is bounded by a constant (not depending on $\eps$, $\delta$, $\lambda$). Therefore, letting $\delta\rightarrow0^+$, taking into account
the notations of \textsc{Step II}, we have
\[
  \lambda\,\text{e}^{\lambda t_0}\,(u - v_{N_0})(t_0,\boldsymbol x_0) \ =
  \lambda \, (u^\lambda - v_{N_0}^\lambda)(t_0,\boldsymbol x_0) \ \leq \ 0,
\]
which gives a contradiction to \eqref{Contradiction}.
\end{proof}

As a direct consequence of the comparison theorem (Theorem \ref{T:Comparison}), we obtain the following uniqueness result.

\begin{cor}\label{C:Uniqueness}
Under Assumption {\bf (A)}, the function $v$ in \eqref{v} is the unique (path-dependent) viscosity solution of equation \eqref{PPDE}, where uniqueness holds in the class of all continuous and bounded functions from $\boldsymbol\Lambda$ to $\R$.
\end{cor}
\begin{proof}
By Theorem \ref{T:Existence} we know that $v$ is continuous and bounded, moreover it is a (path-dependent) viscosity solution of equation \eqref{PPDE}.

Now, let $u\colon\boldsymbol\Lambda\rightarrow\R$ be a continuous and bounded function such that $u$ is a (path-dependent) viscosity solution of equation \eqref{PPDE}. Then, in particular, $u$ (resp. $v$) is a (path-dependent) viscosity subsolution $($resp. supersolution$)$ of equation \eqref{PPDE}. As a consequence, by the comparison theorem (Theorem \ref{T:Comparison}) we deduce that $u\leq v$ on $\boldsymbol\Lambda$. Changing the roles of $u$ and $v$ we get the opposite inequality, from which we conclude that $u\equiv v$.
\end{proof}

\appendix

\section{Functional It\^o's formula}
\label{AppA}

We start with a definition and a technical result.

\begin{defn}
We denote by $\boldsymbol C^{1,0}(\boldsymbol{\hat\Lambda})$ the set of $\hat u\in\boldsymbol C(\boldsymbol{\hat\Lambda})$ such that $\partial_t^H\hat u$ exists everywhere on $\boldsymbol{\hat\Lambda}$ and is continuous.
\end{defn}

\begin{lem}\label{L:C1,0}
Let $\hat u\in\boldsymbol C^{1,0}(\boldsymbol{\hat\Lambda})$. Then, for every $(t,\boldsymbol{\hat x})\in[0,T)\times D([0,T];\R^d)$, the map $\phi\colon[0,T-t]\rightarrow\R$, defined as
\[
\phi(a) \ := \ \hat u(t+a,\boldsymbol{\hat x}(\cdot\wedge t)), \qquad \forall\,a\in[0,T-t],
\]
is in $C^1([0,T-t])$ and
\[
\phi'(a) \ = \ \partial_t^H\hat u(t+a,\boldsymbol{\hat x}(\cdot\wedge t)), \qquad \forall\,a\in[0,T-t].
\]
\end{lem}
\begin{proof}
Let $a\in[0,T-t)$. We have, for any $\delta\in(0,T-t-a]$,
\[
\frac{\phi(a+\delta)-\phi(a)}{\delta} = \frac{\hat u(t+a+\delta,\boldsymbol{\hat x}(\cdot\wedge t))-\hat u(t+a,\boldsymbol{\hat x}(\cdot\wedge t))}{\delta} \underset{\delta\rightarrow0^+}{\longrightarrow} \partial_t^H\hat u(t+a,\boldsymbol{\hat x}(\cdot\wedge t)).
\]
This shows that $\phi$ is right-differentiable on $[0,T-t)$ and that such a right-derivative is continuous on $[0,T-t)$. Then, it follows for instance from Corollary 1.2, Chapter 2, in \cite{pazy83} that $\phi\in C^1([0,T-t))$. Finally, at $a=T-t$, we have
\[
\lim_{a\rightarrow(T-t)^-} \phi'(a) \ = \ \lim_{a\rightarrow(T-t)^-} \partial_t^H\hat u(t+a,\boldsymbol{\hat x}(\cdot\wedge t)) \ = \ \partial_t^H\hat u(T,\boldsymbol{\hat x}(\cdot\wedge t)).
\]
This implies that $\phi\in C^1([0,T-t])$ and concludes the proof.
\end{proof}

We now report the proof of Theorem \ref{T:ItoLifted}.

\begin{proof}[Proof of Theorem \ref{T:ItoLifted}]
Fix $t\in(0,T]$. For every $\boldsymbol{\hat x}\in D([0,T];\R^d)$, denote by $\boldsymbol{\hat x}(\cdot\wedge t-)$ the path
\begin{equation}\label{xt-}
\boldsymbol{\hat x}(s\wedge t-) \ := \
\begin{cases}
\boldsymbol{\hat x}(s), \qquad &0\leq s<t, \\
\boldsymbol{\hat x}(t-), \qquad &t\leq s\leq T.
\end{cases}
\end{equation}
When $t=0$, we set $\boldsymbol{\hat x}(\cdot\wedge 0-):=\boldsymbol{\hat x}(\cdot\wedge0)$. We split the rest of the proof into several steps.

\vspace{1mm}

\noindent\textbf{\emph{Step 1.}} \emph{Piecewise constant approximation of $\boldsymbol X$.} For every $n\in\N$, consider the dyadic partition $\{t_0^n,\ldots,t_\ell^n,\ldots,t_{2^n}^n\}$ of $[0,t]$, with $t_\ell^n=\frac{\ell}{2^n}t$, for every $\ell=0,\ldots,2^n$. Then, we consider the piecewise constant approximation of $\boldsymbol X$
\begin{equation}\label{X^n}
\boldsymbol X_s^n \ = \ \sum_{\ell=0}^{2^n-1} \boldsymbol X_{t_\ell^n}\,1_{[t_\ell^n,t_{\ell+1}^n)}(s) + \boldsymbol X_t\,1_{[t,T]}(s), \qquad \forall\,s\in[0,T].
\end{equation}
Now, recalling the definition of $\boldsymbol{\hat x}(\cdot\wedge t-)$ in \eqref{xt-}, we notice that
\begin{equation}\label{ut-u0}
\hat u(t,\boldsymbol X_{\cdot\wedge t-}^n) - \hat u(0,\boldsymbol X) \ = \ \sum_{\ell=0}^{2^n-1} \Big(\hat u\big(t_{\ell+1}^n,\boldsymbol X_{\cdot\wedge t_{\ell+1}^n-}^n\big) - \hat u\big(t_\ell^n,\boldsymbol X_{\cdot\wedge t_\ell^n-}^n\big)\Big),
\end{equation}
where we have used the telescoping property and the equality $\hat u(0,\boldsymbol X)=\hat u(0,\boldsymbol X^n)$. For every $\ell=0,\ldots,2^n-1$, we have
\[
\hat u\big(t_{\ell+1}^n,\boldsymbol X_{\cdot\wedge t_{\ell+1}^n-}^n\big) - \hat u\big(t_\ell^n,\boldsymbol X_{\cdot\wedge t_\ell^n-}^n\big) \ = \ I_\ell^n + J_\ell^n,
\]
where
\begin{align*}
I_\ell^n \ &:= \ \hat u\big(t_{\ell+1}^n,\boldsymbol X_{\cdot\wedge t_{\ell+1}^n-}^n\big) - \hat u\big(t_\ell^n,\boldsymbol X_{\cdot\wedge t_{\ell+1}^n-}^n\big) \ = \ \hat u\big(t_{\ell+1}^n,\boldsymbol X_{\cdot\wedge t_\ell^n}^n\big) - \hat u\big(t_\ell^n,\boldsymbol X_{\cdot\wedge t_\ell^n}^n\big), \\
J_\ell^n \ &:= \ \hat u\big(t_\ell^n,\boldsymbol X_{\cdot\wedge t_{\ell+1}^n-}^n\big) - \hat u\big(t_\ell^n,\boldsymbol X_{\cdot\wedge t_\ell^n-}^n\big).
\end{align*}
Notice that when $\ell=0$ we have $J_\ell^n=0$. Moreover, when $\ell=1,\ldots,2^n-1$, $J_\ell^n$ can be written as
\[
J_\ell^n \ = \ \hat u\big(t_\ell^n,\boldsymbol X_{\cdot\wedge t_\ell^n}^n\big) - \hat u\big(t_\ell^n,\boldsymbol X_{\cdot\wedge t_{\ell-1}^n}^n\big).
\]
\noindent\textbf{\emph{Step 2.}} \emph{Definitions of $I^n$, $J^{n,1}$, $J^{n,2}$, $J^{n,3}$, $J^{n,4}$.} By Lemma \ref{L:C1,0} we have
\[
I_\ell^n \ = \ \int_{t_\ell^n}^{t_{\ell+1}^n} \partial_t^H \hat u\big(s,\boldsymbol X_{\cdot\wedge t_\ell^n}^n\big)\,ds \ = \ \int_{t_\ell^n}^{t_{\ell+1}^n} \partial_t^H \hat u\big(s,\boldsymbol X_{\cdot\wedge s}^n\big)\,ds.
\]
Then, we define
\begin{equation}\label{I^n}
I^n \ := \ \sum_{\ell=0}^{2^n-1} I_\ell^n \ = \ \sum_{\ell=0}^{2^n-1} \int_{t_\ell^n}^{t_{\ell+1}^n} \partial_t^H \hat u\big(s,\boldsymbol X_{\cdot\wedge s}^n\big)\,ds \ = \ \int_0^t \partial_t^H \hat u\big(s,\boldsymbol X_{\cdot\wedge s}^n\big)\,ds.
\end{equation}
On the other hand, when $\ell=1,\ldots,2^n-1$ the term $J_\ell^n$ can be written as
\[
J_\ell^n \ = \ J_\ell^{n,1} + J_\ell^{n,2} + J_\ell^{n,3} + J_\ell^{n,4},
\]
where (notice that $\boldsymbol X^n=(X^{n,1},\ldots,X^{n,d})$)
\begin{align*}
J_\ell^{n,1} &= \sum_{i=1}^d \partial_{x_i}^V \hat u\big(t_\ell^n,\boldsymbol X_{\cdot\wedge t_{\ell-1}^n}^n\big)\,\big(X_{t_\ell^n}^{n,i} - X_{t_{\ell-1}^n}^{n,i}\big) \ = \ \sum_{i=1}^d \partial_{x_i}^V \hat u\big(t_\ell^n,\boldsymbol X_{\cdot\wedge t_{\ell-1}^n}^n\big)\,\big(X_{t_\ell^n}^i - X_{t_{\ell-1}^n}^i\big), \\
J_\ell^{n,2} &= \frac{1}{2}\sum_{i,j=1}^d \partial_{x_ix_j}^V \hat u\big(t_\ell^n,\boldsymbol X_{\cdot\wedge t_{\ell-1}^n}\big)\,\big(X_{t_\ell^n}^i - X_{t_{\ell-1}^n}^i\big)\big(X_{t_\ell^n}^j - X_{t_{\ell-1}^n}^j\big), \\
J_\ell^{n,3} &= \frac{1}{2}\sum_{i,j=1}^d \Big(\partial_{x_ix_j}^V \hat u\big(t_\ell^n,\boldsymbol X_{\cdot\wedge t_{\ell-1}^n}^n\big) - \partial_{x_ix_j}^V \hat u\big(t_\ell^n,\boldsymbol X_{\cdot\wedge t_{\ell-1}^n}\big)\Big)\big(X_{t_\ell^n}^i - X_{t_{\ell-1}^n}^i\big)\big(X_{t_\ell^n}^j - X_{t_{\ell-1}^n}^j\big), \\
J_\ell^{n,4} &= \frac{1}{2}\sum_{i,j=1}^d \int_0^1 \Big(\partial_{x_ix_j}^V \hat u\big(t_\ell^n,\boldsymbol X_{\cdot\wedge t_{\ell-1}^n}^n + a\big(\boldsymbol X_{t_\ell^n} - \boldsymbol X_{t_{\ell-1}^n}\big)1_{[t_\ell^n,T]}\big) \\
&\quad \ - \partial_{x_ix_j}^V \hat u\big(t_\ell^n,\boldsymbol X_{\cdot\wedge t_{\ell-1}^n}^n\big)\Big)\,\big(X_{t_\ell^n}^i - X_{t_{\ell-1}^n}^i\big)\big(X_{t_\ell^n}^j - X_{t_{\ell-1}^n}^j\big)\,da.
\end{align*}
Then, we define
\begin{equation}\label{J^n,k}
J^{n,k} \ := \ \sum_{\ell=1}^{2^n-1} J_\ell^{n,k},
\end{equation}
for every $k=1,2,3,4$. As a consequence, by \eqref{ut-u0}, \eqref{I^n}, \eqref{J^n,k}, we obtain
\begin{equation}\label{u=IJ}
\hat u(t,\boldsymbol X_{\cdot\wedge t-}^n) - \hat u(0,\boldsymbol X) \ = \ I^n + J^{n,1} + J^{n,2} + J^{n,3} + J^{n,4}.
\end{equation}
Notice that
\begin{equation}\label{J^n,1}
J^{n,1} \ = \ \sum_{\ell=1}^{2^n-1} \sum_{i=1}^d \partial_{x_i}^V \hat u\big(t_\ell^n,\boldsymbol X_{\cdot\wedge t_{\ell-1}^n}^n\big)\,\big(X_{t_\ell^n}^i - X_{t_{\ell-1}^n}^i\big) \ = \ \sum_{i=1}^d \int_0^t Z_s^{n,i}\,dX_s^i,
\end{equation}
where
\[
Z_s^{n,i} \ := \ \sum_{\ell=1}^{2^n-1} \partial_{x_i}^V \hat u\big(t_\ell^n,\boldsymbol X_{\cdot\wedge s}^n\big) 1_{[t_{\ell-1}^n,t_\ell^n)}(s), \qquad \forall\,s\in[0,t].
\]
\noindent\textbf{\emph{Step 3.}} \emph{Convergence.} Before studying the convergence of the various terms in \eqref{u=IJ}, we make some observations. Recall that for a.e. $\omega\in\Omega$ the trajectory $\boldsymbol X(\omega)$ is continuous on $[0,T]$, so that it admits a modulus of continuity $\rho_{\boldsymbol X(\omega)}\colon[0,+\infty)\rightarrow[0,+\infty)$. Then, it holds that
\begin{equation}\label{d_infty_1}
\hat d_\infty\big((t,\boldsymbol X_{\cdot\wedge t-}^n(\omega)),(t,\boldsymbol X_{\cdot\wedge t}(\omega))\big) \ = \ \big\|\boldsymbol X_{\cdot\wedge t-}^n(\omega) - \boldsymbol X_{\cdot\wedge t}(\omega)\big\|_\infty \ \leq \ \rho_{\boldsymbol X(\omega)}(2^{-n}),
\end{equation}
for a.e. $\omega\in\Omega$. Similarly, we have
\begin{equation}\label{d_infty_2}
\hat d_\infty\big((s,\boldsymbol X_{\cdot\wedge s}^n(\omega)),(s,\boldsymbol X_{\cdot\wedge s}(\omega))\big) \ = \ \big\|\boldsymbol X_{\cdot\wedge s}^n(\omega) - \boldsymbol X_{\cdot\wedge s}(\omega)\big\|_\infty \ \leq \ \rho_{\boldsymbol X(\omega)}(2^{-n}),
\end{equation}
for a.e. $\omega\in\Omega$ and for every $s\in[0,t]$. Moreover, for every $\ell=1,\ldots,2^{n}-1$ and $t_{\ell-1}^n\leq s\leq t_\ell^n$, it holds that
\begin{align}\label{d_infty_3}
\hat d_\infty\big((t_\ell^n,\boldsymbol X_{\cdot\wedge s}^n(\omega)),(s,\boldsymbol X_{\cdot\wedge s}(\omega))\big) \ &= \ |t_\ell^n - s| + \big\|\boldsymbol X_{\cdot\wedge s}^n(\omega) - \boldsymbol X_{\cdot\wedge s}(\omega)\big\|_\infty \notag \\
&\leq \ 2^{-n} + \rho_{\boldsymbol X(\omega)}(2^{-n}),
\end{align}
for a.e. $\omega\in\Omega$. Now, for every $\boldsymbol x\in C([0,T];\R^d)$ and $n\in\N$, set (as in \eqref{X^n})
\[
\boldsymbol x^n(s) \ = \ \sum_{\ell=0}^{2^n-1} \boldsymbol x(t_\ell^n)\,1_{[t_\ell^n,t_{\ell+1}^n)}(s) + \boldsymbol x(t)\,1_{[t,T]}(s), \qquad \forall\,s\in[0,T].
\]
For such an $\boldsymbol x$, define
\begin{align}
K_1(\boldsymbol x) \ &:= \ \big\{(s,\boldsymbol{\hat x})\in\boldsymbol{\hat\Lambda}\colon \boldsymbol{\hat x}\,=\,\boldsymbol x^n(\cdot\wedge s),\text{ for some }n\in\N\big\}, \label{K_1(x)} \\
K_2(\boldsymbol x) \ &:= \ \big\{(r,\boldsymbol{\hat x})\in\boldsymbol{\hat\Lambda}\colon \big(r,\boldsymbol{\hat x}\big)\,=\,\big(t_\ell^n,\boldsymbol x^n(\cdot\wedge s)\big),\text{ for some }n\in\N, \notag \\
&\hspace{4cm}\ell=1,\ldots,2^n-1,\,s\in[t_{\ell-1}^n,t_\ell^n)\big\}, \label{K_2(x)} \\
K_3(\boldsymbol x) \ &:= \ \big\{(r,\boldsymbol{\hat x})\in\boldsymbol{\hat\Lambda}\colon \big(r,\boldsymbol{\hat x}\big)\,=\,\big(t_\ell^n,\boldsymbol x^n(\cdot\wedge t_{\ell-1}^n)+a(\boldsymbol x(t_\ell^n)-\boldsymbol x(t_{\ell-1}^n))1_{[t_\ell^n,T]}(\cdot)\big), \notag \\
&\hspace{4cm}\text{for some }n\in\N,\,\ell=1,\ldots,2^n-1,\,a\in[0,1]\big\}. \label{K_3(x)}
\end{align}
Notice that $K_2(\boldsymbol x)\subset K_3(\boldsymbol x)$ (observe that $\boldsymbol x^n(\cdot\wedge s)=\boldsymbol x^n(\cdot\wedge t_{\ell-1}^n)$, whenever $s\in[t_{\ell-1}^n,t_\ell^n)$, and consider the pairs in $K_3(\boldsymbol x)$ with $a=0$). It is easy to see that $K_1(\boldsymbol x)$ and $K_3(\boldsymbol x)$ (and a fortiori $K_2(\boldsymbol x)$) are relatively compact subsets of $(\boldsymbol{\hat\Lambda},\hat d_\infty)$. Now, we study separately the convergence of the various terms appearing in \eqref{u=IJ}.

\vspace{1mm}

\noindent\textbf{\emph{Substep 3.1.}} \emph{Convergence of $\hat u(t,\boldsymbol X_{\cdot\wedge t-}^n) - \hat u(0,\boldsymbol X)$.} Since $\hat u$ is continuous on $(\boldsymbol{\hat\Lambda},\hat d_\infty)$ and \eqref{d_infty_1} holds, we deduce that
\[
\hat u(t,\boldsymbol X_{\cdot\wedge t-}^n) - \hat u(0,\boldsymbol X) \ \underset{n\rightarrow\infty}{\overset{\P\text{-a.s.}}{\longrightarrow}} \ \hat u(t,\boldsymbol X_{\cdot\wedge t}) - \hat u(0,\boldsymbol X).
\]
\noindent\textbf{\emph{Substep 3.2.}} \emph{Convergence of $I^n$.} Recalling \eqref{K_1(x)}, we see that for a.e. $\omega\in\Omega$ the set $K_1(\boldsymbol X(\omega))$ is a relatively compact subset of $\boldsymbol{\hat\Lambda}$. Since $\partial_t^H\hat u$ is continuous on $(\boldsymbol{\hat\Lambda},\hat d_\infty)$, for such an $\omega$ we see that $\partial_t^H\hat u(s,\boldsymbol X_{\cdot\wedge s}^n(\omega))$ is uniformly bounded with respect to $s$ and $n$. In addition, using again the continuity of $\partial_t^H\hat u$ and \eqref{d_infty_2}, we deduce that for a.e. $\omega\in\Omega$ the sequence of maps $s\mapsto\partial_t^H\hat u(s,\boldsymbol X_{\cdot\wedge s}^n(\omega))$ converges pointwise to $s\mapsto\partial_t^H\hat u(s,\boldsymbol X_{\cdot\wedge s}(\omega))$. In conclusion, we can apply Lebesgue's dominated convergence theorem, which yields
\[
I^n \ = \ \int_0^t \partial_t^H \hat u\big(s,\boldsymbol X_{\cdot\wedge s}^n\big)\,ds \ \underset{n\rightarrow\infty}{\overset{\P\text{-a.s.}}{\longrightarrow}} \ \int_0^t \partial_t^H \hat u\big(s,\boldsymbol X_{\cdot\wedge s}\big)\,ds.
\]
\noindent\textbf{\emph{Substep 3.3.}} \emph{Convergence of $J^{n,1}$.} In order to study the convergence of the sum of stochastic integrals in \eqref{J^n,1}, we write each $X^i$ as $V^i+M^i$, where $V^i$ is a bounded variation process and $M^i$ is a continuous local martingale, so that the stochastic integral with respect to $X^i$ can be written as the sum of two integrals with respect to $V^i$ and $M^i$, respectively.

\vspace{1mm}

\noindent\emph{Convergence of $\int_0^t Z_s^{n,i}\,dV_s^i$.} Recalling \eqref{K_2(x)}, we see that for a.e.
$\omega\in\Omega$ the set $K_2(\boldsymbol X(\omega))$ is a relatively compact subset of $\boldsymbol{\hat\Lambda}$. Then, reasoning as in \textbf{\emph{Substep 3.2}}, using the continuity of $\partial_{x_i}^V\hat u$ and also \eqref{d_infty_3}, we see that we can apply Lebesgue's dominated convergence theorem, from which we get
\[
\int_0^t Z_s^{n,i}\,dV_s^i \ \underset{n\rightarrow\infty}{\overset{\P\text{-a.s.}}{\longrightarrow}} \ \int_0^t \partial_{x_i}^V \hat u\big(s,\boldsymbol X_{\cdot\wedge s}\big)\,dV_s^i.
\]
\noindent\emph{Convergence of $\int_0^t Z_s^{n,i}\,dM_s^i$.} Let $\langle M^i\rangle$ denote the quadratic variation of $M^i$. Reasoning as in the proof of the convergence of $\int_0^t Z_s^{n,i}\,dV_s^i$, we obtain by Lebesgue's dominated convergence theorem
\[
  \int_0^t \big|Z_s^{n,i} - \hat u\big(s,\boldsymbol X_{\cdot\wedge s}
  \big)\big|^2\,d\langle M^i\rangle_s \
  \underset{n\rightarrow\infty}{\overset{\P\text{-a.s.}}{\longrightarrow}} \ 0.
\]
By Proposition 2.26, Chapter 3, in \cite{KaratzasShreve}, we deduce the convergence in probability (even u.c.p.)
\[
\int_0^t Z_s^{n,i}\,dM_s^i \ \underset{n\rightarrow\infty}{\overset{\P}{\longrightarrow}} \ \int_0^t \partial_{x_i}^V \hat u\big(s,\boldsymbol X_{\cdot\wedge s}\big)\,dM_s^i.
\]
\noindent\textbf{\emph{Substep 3.4.}} \emph{Convergence of $J^{n,2}$.} For every $n\in\N$ and $i,j=1,\ldots,d$, consider the process $\langle X^i,X^j\rangle^n$ defined as
\[
\langle X^i,X^j\rangle_s^n \ := \ \sum_{\ell=1}^{2^n} \big(X_{t_\ell^n}^i - X_{t_{\ell-1}^n}^i\big)\big(X_{t_\ell^n}^j - X_{t_{\ell-1}^n}^j\big)\,1_{[t_{\ell-1}^n,T]}(s), \qquad \forall\,s\in[0,T].
\]
Notice that
\[
J_\ell^{n,2} \ = \ \frac{1}{2}\sum_{i,j=1}^d \int_{[t_{\ell-1}^n,t_\ell^n)} \partial_{x_ix_j}^V \hat u\big(s,\boldsymbol X_{\cdot\wedge s}\big)\,d\langle X^i,X^j\rangle_s^n
\]
and, therefore,
\[
J^{n,2} \ = \ \frac{1}{2}\sum_{i,j=1}^d \int_{[0,(1-2^{-n})t)} \partial_{x_ix_j}^V \hat u\big(s,\boldsymbol X_{\cdot\wedge s}\big)\,d\langle X^i,X^j\rangle_s^n.
\]
From Theorem 23 in Section II.6 of \cite{protter05}, it is easy
to deduce that 
\[
\sup_{s\in[0,t]}\big|\langle X^i,X^j\rangle_s^n - \langle X^i,X^j\rangle_s\big| \ \underset{n\rightarrow\infty}{\overset{\P}{\longrightarrow}} \ 0,
\]
where $\langle X^i,X^j\rangle$ is the covariation of $X^i$ and $X^j$. Then, up to a subsequence,
there exists a $\P$-null set $N$ such that, whenever $\omega\notin N$, $\langle X^i,X^j\rangle_s^n(\omega)\rightarrow\langle X^i,X^j\rangle_s(\omega)$, for all $s\in[0,t]$. For such an $\omega$, this implies that the measure $d\langle X^i,X^j\rangle^n(\omega)$ converges weakly to the measure $d\langle X^i,X^j\rangle^n(\omega)$. Since $\partial_{x_ix_j}^V \hat u$ is continuous, we deduce (up to a subsequence)
\[
J^{n,2} \ \underset{n\rightarrow\infty}{\overset{\P\text{-a.s.}}{\longrightarrow}} \ \frac{1}{2}\sum_{i,j=1}^d \int_0^t \partial_{x_ix_j}^V \hat u\big(s,\boldsymbol X_{\cdot\wedge s}\big)\,d\langle X^i,X^j\rangle_s.
\]
\noindent\textbf{\emph{Substep 3.5.}} \emph{Convergence of $J^{n,3}$.} We have (we set $\langle X^i\rangle^n:=\langle X^i,X^i\rangle^n$, for every $i=1,\ldots,d$)
\[
|J^{n,3}| \ \leq \ \frac{1}{4}\sum_{i,j=1}^d \sup_{\ell=1,\ldots,2^n-1}\Big|\partial_{x_ix_j}^V \hat u\big(t_\ell^n,\boldsymbol X_{\cdot\wedge t_{\ell-1}^n}^n\big) - \partial_{x_ix_j}^V \hat u\big(t_\ell^n,\boldsymbol X_{\cdot\wedge t_{\ell-1}^n}\big)\Big|\big(\langle X^i\rangle_T^n + \langle X^j\rangle_T^n\big).
\]
Recalling \eqref{K_2(x)}, we see that for a.e. $\omega\in\Omega$ the set $K_2(\boldsymbol X(\omega))$ is a relatively compact subset of $\boldsymbol{\hat\Lambda}$. Since $\partial_{x_ix_j}^V \hat u$ is continuous on $(\boldsymbol{\hat\Lambda},\hat d_\infty)$, for a.e. $\omega\in\Omega$ there exists a (non-decreasing) modulus of continuity $\rho_{i,j}^\omega\colon[0,+\infty)\rightarrow[0,+\infty)$ such that
\[
|J^{n,3}(\omega)| \leq \frac{1}{4}\sum_{i,j=1}^d \sup_{\ell=1,\ldots,2^n-1} \rho_{i,j}^\omega\big(\big\|\boldsymbol X_{\cdot\wedge t_{\ell-1}^n}^n(\omega) - \boldsymbol X_{\cdot\wedge t_{\ell-1}^n}(\omega)\big\|_\infty\big)\big(\langle X^i\rangle_T^n(\omega) + \langle X^j\rangle_T^n(\omega)\big).
\]
From \textbf{\emph{Substep 3.4}} we know that, up to a subsequence, $\langle X^i\rangle_T^n(\omega)$ converges to $\langle X^i\rangle_T(\omega)$, for a.e. $\omega\in\Omega$ and for every $i=1,\ldots,d$. On the other hand, by \eqref{d_infty_2} we have that $\sup_{\ell=1,\ldots,2^n-1}\rho_{i,j}^\omega(\|\boldsymbol X_{\cdot\wedge t_{\ell-1}^n}^n(\omega)-\boldsymbol X_{\cdot\wedge t_{\ell-1}^n}(\omega)\|_\infty)\leq\rho_{i,j}^\omega(\rho_{\boldsymbol X(\omega)}(2^{-n}))$, for a.e. $\omega\in\Omega$. Hence, up to a subsequence,
\[
J^{n,3} \ \underset{n\rightarrow\infty}{\overset{\P\text{-a.s.}}{\longrightarrow}} \ 0.
\]
\noindent\textbf{\emph{Substep 3.6.}} \emph{Convergence of $J^{n,4}$.} We have (as in \textbf{\emph{Substep 3.5}} we set $\langle X^i\rangle^n:=\langle X^i,X^i\rangle^n$, for every $i=1,\ldots,d$)
\begin{align*}
|J^{n,4}| \ \leq \ \frac{1}{4}\sum_{i,j=1}^d \big(\langle X^i\rangle_T^n + \langle X^j\rangle_T^n\big)\int_0^1 \sup_{\ell=1,\ldots,2^n-1}\Big|\partial_{x_ix_j}^V \hat u\big(t_\ell^n,\boldsymbol X_{\cdot\wedge t_{\ell-1}^n}^n\big)& \\
- \ \partial_{x_ix_j}^V \hat u\big(t_\ell^n,\boldsymbol X_{\cdot\wedge t_{\ell-1}^n}^n + a\big(\boldsymbol X_{t_\ell^n} - &\boldsymbol X_{t_{\ell-1}^n}\big)1_{[t_\ell^n,T]}\big)\Big|\,da.
\end{align*}
Recalling \eqref{K_3(x)}, we see that for a.e. $\omega\in\Omega$ the set $K_3(\boldsymbol X(\omega))$ is a relatively compact subset of $\boldsymbol{\hat\Lambda}$. Since $\partial_{x_ix_j}^V\hat u$ is continuous on $(\boldsymbol{\hat\Lambda},\hat d_\infty)$, for a.e. $\omega\in\Omega$ there exists a (non-decreasing) modulus of continuity $\bar\rho_{i,j}^\omega\colon[0,+\infty)\rightarrow[0,+\infty)$ such that
\[
|J^{n,4}(\omega)| \leq \frac{1}{4}\sum_{i,j=1}^d \big(\langle X^i\rangle_T^n(\omega) + \langle X^j\rangle_T^n(\omega)\big)\int_0^1 \sup_{\ell=1,\ldots,2^n-1}\bar\rho_{i,j}^\omega\big(a\big|\boldsymbol X_{t_\ell^n}(\omega) - \boldsymbol X_{t_{\ell-1}^n}(\omega)\big|\big)\,da.
\]
We know from \textbf{\emph{Substep 3.4}} that, up to a subsequence, $\langle X^i\rangle_T^n(\omega)$ converges to $\langle X^i\rangle_T(\omega)$, for a.e. $\omega\in\Omega$ and for every $i=1,\ldots,d$. On the other hand, $\sup_{\ell=1,\ldots,2^n-1}\bar\rho_{i,j}^\omega(a|\boldsymbol X_{t_\ell^n}(\omega) - \boldsymbol X_{t_{\ell-1}^n}(\omega)|)\leq\bar\rho_{i,j}^\omega(\rho_{\boldsymbol X(\omega)}(2^{-n}))$, for a.e. $\omega\in\Omega$, where we recall that $\rho_{\boldsymbol X(\omega)}$ is the modulus of continuity of the trajectory $\boldsymbol X(\omega)$. Hence, up to a subsequence,
\[
J^{n,4} \ \underset{n\rightarrow\infty}{\overset{\P\text{-a.s.}}{\longrightarrow}} \ 0.
\]
\end{proof}

\section{Consistency}
\label{AppB}

\begin{proof}[Proof of Lemma \ref{L:Consistency}]
The claim concerning the horizontal derivatives follows directly from their definition (Definition \ref{D:FunctionalDerivatives}-(i)).

It remains to prove the claim concerning the vertical derivatives. To this end, let $\boldsymbol X=(\boldsymbol X_t)_{t\in[0,T]}$, with $\boldsymbol X=(X^1,\ldots,X^d)$, be a $d$-dimensional continuous semimartingale on some filtered probability space $(\Omega,{\mathcal F},({\mathcal F}_t)_{t\in[0,T]},\P)$, where $({\mathcal F}_t)_{t\in[0,T]}$ satisfies the usual conditions. Since $\hat u_1,\hat u_2\in\boldsymbol C^{1,2}(\boldsymbol{\hat\Lambda})$, by Theorem \ref{T:ItoLifted}
 the functional It\^o formulae 
\begin{align*}
\hat u_1(t,\boldsymbol X) \ &= \ \hat u_1(0,\boldsymbol X) + \int_0^t \partial_t^H\hat u_1(s,\boldsymbol X)\,ds + \frac{1}{2}\sum_{i,j=1}^d \int_0^t \partial^V_{x_i x_j} \hat u_1(s,\boldsymbol X)\,d[X^i,X^j]_s \\
&\quad \ + \sum_{i=1}^d \int_0^t \partial^V_{x_i}\hat u_1(s,\boldsymbol X)\,dX_s^i, \hspace{2.4cm} \text{for all }\,0\leq t\leq T,\,\,\P\text{-a.s.}
\end{align*}
and
\begin{align*}
\hat u_2(t,\boldsymbol X) \ &= \ \hat u_2(0,\boldsymbol X) + \int_0^t \partial_t^H\hat u_2(s,\boldsymbol X)\,ds + \frac{1}{2}\sum_{i,j=1}^d \int_0^t \partial^V_{x_i x_j} \hat u_2(s,\boldsymbol X)\,d[X^i,X^j]_s \\
&\quad \ + \sum_{i=1}^d \int_0^t \partial^V_{x_i}\hat u_2(s,\boldsymbol X)\,dX_s^i, \hspace{2.4cm} \text{for all }\,0\leq t\leq T,\,\,\P\text{-a.s.}
\end{align*}
hold. Recalling that $\hat u_1$ and $\hat u_2$, together with their horizontal derivatives, coincide on continuous paths, identifying bounded variation and local martingale parts in the above formulae, (up to a $\P$-null set), for every $i,j=1,\ldots,d$ and any $t\in[0,T]$,
we have
\begin{equation}\label{FullSupport}
\partial^V_{x_i}\hat u_1(t,\boldsymbol X) \ = \ \partial^V_{x_i}\hat u_2(t,\boldsymbol X), \qquad\quad \partial^V_{x_i x_j} \hat u_1(t,\boldsymbol X) \ = \ \partial^V_{x_i x_j} \hat u_2(t,\boldsymbol X).
\end{equation}
Now, recall that $C([0,T];\R^d)$ is endowed with the uniform topology and let $\mathcal B(C([0,T];\R^d))$ denote the Borel $\sigma$-algebra on $C([0,T];\R^d)$. We also recall that given a probability measure $\mu\colon\mathcal B(C([0,T];\R^d))\rightarrow[0,1]$, the \emph{support} of $\mu$ is the smallest closed set $\mathcal S_\mu\subset C([0,T];\R^d)$ such that $\mu(\mathcal S_\mu)=1$. Consider a continuous semimartingale $\boldsymbol X=(\boldsymbol X_s)_{s\in[0,T]}$ whose law has support equal to $C([0,T];\R^d)$. An example of such an $\boldsymbol X$ is given by $\boldsymbol X_s:=\boldsymbol\eta+\boldsymbol W_s$, $s\in[0,T]$, where $\boldsymbol W=(\boldsymbol W_s)_{s\in[0,T]}$ is a $d$-dimensional Brownian motion, while $\boldsymbol\eta\colon\Omega\rightarrow\R^d$ is independent of $\boldsymbol W$ and has a standard normal multivariate distribution. Then, for every fixed $t\in[0,T]$, using equalities \eqref{FullSupport} with such a semimartingale $\boldsymbol X$, and exploiting the continuity of $\partial^V_{\boldsymbol x} \hat u_1$, $\partial^V_{\boldsymbol x} \hat u_2$, $\partial^V_{\boldsymbol x\boldsymbol x} \hat u_1$, $\partial^V_{\boldsymbol x\boldsymbol x} \hat u_2$, we obtain
\[
\partial^V_{\boldsymbol x} \hat u_1(t,\boldsymbol x) \ = \ \partial^V_{\boldsymbol x} \hat u_2(t,\boldsymbol x), \qquad\quad \partial^V_{\boldsymbol x\boldsymbol x} \hat u_1(t,\boldsymbol x) \ = \ \partial^V_{\boldsymbol x\boldsymbol x} \hat u_2(t,\boldsymbol x),
\]
for every $\boldsymbol x\in C([0,T];\R^d)$. Since the above equalities hold for every $t\in[0,T]$, the claim follows.
\end{proof}

\section{Smooth variational principle on $\boldsymbol\Lambda$}
\label{AppVar}

\subsection{Proof of Lemma \ref{L:rho_infty}}
\label{AppVar1}

\begin{proof}[Proof of Lemma \ref{L:rho_infty}] We split the proof into several steps.

\vspace{1mm}

\noindent\textsc{Step I}. \emph{Proof of item }1). We first notice that $\hat\kappa_\infty^{(t_0,\boldsymbol x_0)}$ is a non-anticipative map. Now, let $(t,\boldsymbol{\hat x})\in\boldsymbol{\hat\Lambda}$, $h\in\R\backslash\{0\}$, and $i=1,\ldots,d$ (recall that $\mathbf e_1,\ldots,\mathbf e_d$
 denotes the standard orthonormal basis of $\R^d$),
 then we have
\begin{align*}
&\frac{\hat\kappa_\infty^{(t_0,\boldsymbol x_0)}(t,\boldsymbol{\hat x} + h\,\mathbf e_i\,1_{[t,T]}) - \hat\kappa_\infty^{(t_0,\boldsymbol x_0)}(t,\boldsymbol{\hat x})}{h} \\
&= \ \int_{\R^d} \big\|\boldsymbol{\hat x}(\cdot\wedge t) - \boldsymbol x_0(\cdot\wedge t_0) - \mathbf z\,1_{[t,T]}\big\|_\infty\,\frac{\zeta(\mathbf z + h\,\mathbf e_i) - \zeta(\mathbf z)}{h}\,d\mathbf z \\
&\hspace{3cm} \overset{h\rightarrow0}{\longrightarrow} \ \int_{\R^d} \big\|\boldsymbol{\hat x}(\cdot\wedge t) - \boldsymbol x_0(\cdot\wedge t_0) - \mathbf z\,1_{[t,T]}\big\|_\infty\,\partial_{z_i}\zeta(\mathbf z)\,d\mathbf z,
\end{align*}
where $\partial_{z_i}\zeta(\mathbf z)$ denotes the partial derivative of $\zeta$ in the $\mathbf e_i$-direction at the point $\mathbf z$, which is given by $-z_i\,\zeta(\mathbf z)$. This proves that $\hat\kappa_\infty^{(t_0,\boldsymbol x_0)}$ admits first-order vertical derivatives at every $(t,\boldsymbol{\hat x})\in\boldsymbol{\hat\Lambda}$. In a similar way we can prove that $\hat\kappa_\infty^{(t_0,\boldsymbol x_0)}$ also admits second-order vertical derivatives at every $(t,\boldsymbol{\hat x})\in\boldsymbol{\hat\Lambda}$.

\vspace{1mm}

\noindent\textsc{Step II}. \emph{Proof of item }2). We begin noticing that
\begin{align*}
&|\hat\kappa_\infty^{(t_0,\boldsymbol x_0)}(t,\boldsymbol{\hat x} + h\,\mathbf e_i\,1_{[t,T]}) - \hat\kappa_\infty^{(t_0,\boldsymbol x_0)}(t,\boldsymbol{\hat x})| \\
&= \ \bigg|\int_{\R^d} \big\|\boldsymbol{\hat x}(\cdot\wedge t) - \boldsymbol x_0(\cdot\wedge t_0) - (\mathbf z - h\,\mathbf e_i)\,1_{[t,T]}\big\|_\infty\,\zeta(\mathbf z)\,d\mathbf z \\
&\quad \ - \int_{\R^d} \big\|\boldsymbol{\hat x}(\cdot\wedge t) - \boldsymbol x_0(\cdot\wedge t_0) - \mathbf z\,1_{[t,T]}\big\|_\infty\,\zeta(\mathbf z)\,d\mathbf z \bigg| \ \leq \ |h\,\mathbf e_i| \ = \ |h|,
\end{align*}
where we have used the fact that $\int_{\R^d}\zeta(\mathbf z)\,d\mathbf z=1$. It is then easy to see that, for every $i=1,\ldots,d$, $\partial_{x_i}^V\hat\kappa_\infty^{(t_0,\boldsymbol x_0)}$ is bounded by the constant $1$. Proceeding along the same lines as for $\hat\kappa_\infty^{(t_0,\boldsymbol x_0)}$, we deduce that $|\partial^V_{x_i}\hat\kappa_\infty^{(t_0,\boldsymbol x_0)}(t,\boldsymbol{\hat x} + h\,\mathbf e_j\,1_{[t,T]}) - \partial^V_{x_i}\hat\kappa_\infty^{(t_0,\boldsymbol x_0)}(t,\boldsymbol{\hat x})|$ is bounded by $|h|\int_{\R^d}|\partial_{z_i}\zeta(\mathbf z)|\,d\mathbf z=|h|\int_{\R^d}|z_i|\,\zeta(\mathbf z)\,d\mathbf z=|h|\sqrt{\frac{2}{\pi}}$. This allows to prove that, for every $i,j=1,\ldots,d$, $\partial^V_{x_i x_j}\hat\kappa_\infty^{(t_0,\boldsymbol x_0)}$ is bounded by $\sqrt{\frac{2}{\pi}}$.

\vspace{1mm}

\noindent\textsc{Step III}. \emph{Proof of item} 3). We begin noting that
\[
\hat\kappa_\infty^{(t_0,\boldsymbol x_0)}(t,\boldsymbol{\hat x}) \ = \ \int_{\R^d}\big\|\boldsymbol{\hat x}(\cdot\wedge t) - \boldsymbol x_0(\cdot\wedge t_0) - \mathbf z\,1_{[t,T]}\big\|_\infty\,\zeta(\mathbf z)\,d\mathbf z - \int_{\R^d}|\mathbf z|\,\zeta(\mathbf z)\,d\mathbf z \ \geq \ - C_\zeta,
\]
with $C_\zeta:=\int_{\R^d}|\mathbf z|\,\zeta(\mathbf z)\,d\mathbf z$, which proves the first part of item 3).

Now, let $(t,\boldsymbol x)\in\boldsymbol\Lambda$. Using the fact that $\zeta$ is a radial function, we have (when $d=1$, $(-\infty,0]\times\R^0$ and $[0,+\infty)\times\R^0$ stand for $(-\infty,0]$ and $[0,+\infty)$, respectively)
\begin{align}\label{Inequality_z_-z}
&\int_{\R^d}\big\|\boldsymbol x(\cdot\wedge t) - \boldsymbol x_0(\cdot\wedge t_0) - \mathbf z1_{[t,T]}\big\|_\infty\,\zeta(\mathbf z)\,d\mathbf z \notag \\
&= \ \int_{[0,+\infty)\times\R^{d-1}}\big\|\boldsymbol x(\cdot\wedge t) - \boldsymbol x_0(\cdot\wedge t_0) - \mathbf z1_{[t,T]}\big\|_\infty\,\zeta(\mathbf z)\,d\mathbf z \notag \\
&\quad \ + \int_{(-\infty,0]\times\R^{d-1}}\big\|\boldsymbol x(\cdot\wedge t) - \boldsymbol x_0(\cdot\wedge t_0) - \mathbf z1_{[t,T]}\big\|_\infty\,\zeta(\mathbf z)\,d\mathbf z	 \notag \\
&= \ \int_{[0,+\infty)\times\R^{d-1}}\big\|\boldsymbol x(\cdot\wedge t) - \boldsymbol x_0(\cdot\wedge t_0) - \mathbf z1_{[t,T]}\big\|_\infty\,\zeta(\mathbf z)\,d\mathbf z \notag \\
&\quad \ + \int_{[0,+\infty)\times\R^{d-1}}\big\|\boldsymbol x(\cdot\wedge t) - \boldsymbol x_0(\cdot\wedge t_0) + \mathbf z1_{[t,T]}\big\|_\infty\,\zeta(\mathbf z)\,d\mathbf z.
\end{align}
Now, we observe that, for every $\mathbf z\in\R^d$, we have
\begin{align}\label{norm_infty}
&\big\|\boldsymbol x(\cdot\wedge t) - \boldsymbol x_0(\cdot\wedge t_0) - \mathbf z1_{[t,T]}\big\|_\infty \notag \\
&= \ \max\Big\{\big\|\boldsymbol x(\cdot\wedge t) - \boldsymbol x_0(\cdot\wedge t\wedge t_0)\big\|_\infty,\max_{t\leq s\leq T}\big|\boldsymbol x(t) - \boldsymbol x_0(s\wedge t_0) - \mathbf z\big|\Big\}
\end{align}
and similarly for $\|\boldsymbol x(\cdot\wedge t) - \boldsymbol x_0(\cdot\wedge t_0) + \mathbf z1_{[t,T]}\|_\infty$. Moreover, by the elementary inequality $|\mathbf x - \mathbf z| + |\mathbf x + \mathbf z|\geq2|\mathbf x|$, valid for every $\mathbf x,\mathbf z\in\R^d$, we have
\begin{align*}
&\max_{t\leq s\leq T}\big|\boldsymbol x(t) - \boldsymbol x_0(s\wedge t_0) - \mathbf z\big| + \max_{t\leq s\leq T}\big|\boldsymbol x(t) - \boldsymbol x_0(s\wedge t_0) + \mathbf z\big| \\
&\geq \ \max_{t\leq s\leq T}\Big\{\big|\boldsymbol x(t) - \boldsymbol x_0(s\wedge t_0) - \mathbf z\big| + \big|\boldsymbol x(t) - \boldsymbol x_0(s\wedge t_0) + \mathbf z\big|\Big\} \\
&\geq \ 2\max_{t\leq s\leq T}\big|\boldsymbol x(t) - \boldsymbol x_0(s\wedge t_0)\big|.
\end{align*}
Then, using the elementary fact that if $a+b\geq2c$ it holds that $\max\{\ell,a\}+\max\{\ell,b\}\geq2\max\{\ell,c\}$, for every $a,b,c,\ell\in\R$, we find
\begin{align*}
&\max\Big\{\big\|\boldsymbol x(\cdot\wedge t) - \boldsymbol x_0(\cdot\wedge t\wedge t_0)\big\|_\infty,\max_{t
\leq s\leq T}\big|\boldsymbol x(t) - \boldsymbol x_0(s\wedge t_0) - \mathbf z\big|\Big\} \\
&+ \max\Big\{\big\|\boldsymbol x(\cdot\wedge t) - \boldsymbol x_0(\cdot\wedge t\wedge t_0)\big\|_\infty,\max_{t\leq s\leq T}\big|\boldsymbol x(t) - \boldsymbol x_0(s\wedge t_0) + \mathbf z\big|\Big\} \\
&\geq \ 2\max\Big\{\big\|\boldsymbol x(\cdot\wedge t) - \boldsymbol x_0(\cdot\wedge t\wedge t_0)\big\|_\infty,\max_{t\leq s\leq T}\big|\boldsymbol x(t) - \boldsymbol x_0(s\wedge t_0)\big|\Big\}.
\end{align*}
By \eqref{norm_infty} it follows that the last quantity coincides with $2\,\|\boldsymbol x(\cdot\wedge t) - \boldsymbol x_0(\cdot\wedge t_0)\|_\infty$. Therefore, by \eqref{Inequality_z_-z} we obtain (also recalling that $C_\zeta:=\int_{\R^d}|\mathbf z|\,\zeta(\mathbf z)\,d\mathbf z$)
\begin{align*}
\kappa_\infty^{(t_0,\boldsymbol x_0)}(t,\boldsymbol x) + C_\zeta \ &= \ \int_{\R^d}\big\|\boldsymbol x(\cdot\wedge t) - \boldsymbol x_0(\cdot\wedge t_0) - \mathbf z1_{[t,T]}\big\|_\infty\,\zeta(\mathbf z)\,d\mathbf z \\
&\geq \ 2\int_{[0,+\infty)\times\R^{d-1}}\big\|\boldsymbol x(\cdot\wedge t) - \boldsymbol x_0(\cdot\wedge t_0)\big\|_\infty\,\zeta(\mathbf z)\,d\mathbf z \\
&= \ \big\|\boldsymbol x(\cdot\wedge t) - \boldsymbol x_0(\cdot\wedge t_0)\big\|_\infty,
\end{align*}
which concludes the proof of item 3). Finally, the explicit expression of the constant $C_\zeta$, reported in \eqref{C_zeta}, follows from the fact that $C_\zeta=\mu_{\chi^2(d),\frac{1}{2}}$, where $\mu_{\chi^2(d),\frac{1}{2}}$ denotes the moment of order $1/2$ of a $\chi^2$-distribution with $d$ degrees of freedom.

\vspace{1mm} 

\noindent\textsc{Step IV}. \emph{Proof of the second inequality in \eqref{Ineq12}.} The second inequality in \eqref{Ineq12} follows easily from an application of the triangular inequality, namely noting that $\|\boldsymbol x(\cdot\wedge t) - \boldsymbol x_0(\cdot\wedge t_0) - \mathbf z1_{[t,T]}\|_\infty\leq\|\boldsymbol x(\cdot\wedge t) - \boldsymbol x_0(\cdot\wedge t_0)\|_\infty+|\mathbf z|$.

\vspace{1mm}

\noindent\textsc{Step V}. \emph{Proof of the first inequality in \eqref{Ineq12} for the case $\|\boldsymbol x(\cdot\wedge t) - \boldsymbol x_0(\cdot\wedge t_0)\|_\infty>2C_\zeta$.}\\
When $\|\boldsymbol x(\cdot\wedge t) - \boldsymbol x_0(\cdot\wedge t_0)\|_\infty>2C_\zeta$, we have, by item 3),
\begin{align*}
\kappa_\infty^{(t_0,\boldsymbol x_0)}(t,\boldsymbol x) \ &\geq \ \|\boldsymbol x(\cdot\wedge t) - \boldsymbol x_0(\cdot\wedge t_0)\|_\infty - C_\zeta \\
&= \ \|\boldsymbol x(\cdot\wedge t) - \boldsymbol x_0(\cdot\wedge t_0)\|_\infty\bigg(1 - \frac{C_\zeta}{\|\boldsymbol x(\cdot\wedge t) - \boldsymbol x_0(\cdot\wedge t_0)\|_\infty}\bigg) \\
&\geq \ \frac{1}{2}\,\|\boldsymbol x(\cdot\wedge t) - \boldsymbol x_0(\cdot\wedge t_0)\|_\infty \\
&\geq \ \frac{1}{2}\,\big(\|\boldsymbol x(\cdot\wedge t) - \boldsymbol x_0(\cdot\wedge t_0)\|_\infty^{d+1}\wedge\|\boldsymbol x(\cdot\wedge t) - \boldsymbol x_0(\cdot\wedge t_0)\|_\infty\big),
\end{align*}
which proves the first inequality in \eqref{Ineq12} with $\alpha_d:=\frac{1}{2}$, for the case $\|\boldsymbol x(\cdot\wedge t) - \boldsymbol x_0(\cdot\wedge t_0)\|_\infty>2C_\zeta$.

\vspace{1mm}

\noindent\textsc{Step VI}. \emph{Proof of the first inequality in \eqref{Ineq12} for the case $\|\boldsymbol x(\cdot\wedge t) - \boldsymbol x_0(\cdot\wedge t_0)\|_\infty\leq2C_\zeta$.}

\vspace{1mm}

\noindent\textsc{Step VI-1}. Our aim is to prove that for every fixed $d$ there exists some constant $\alpha_d>0$ such that
\begin{align}\label{FirstIneq}
&\int_{\R^d}\max\big\{a,|\mathbf y - \mathbf z|\big\}\,\zeta(\mathbf z)\,d\mathbf z - \int_{\R^d}|\mathbf z|\,\zeta(\mathbf z)\,d\mathbf z \\
&\hspace{2.5cm}\geq \ \alpha_d\min\{a^{d+1},a\} + \alpha_d\min\{|\mathbf y|^{d+1},|\mathbf y|\}, \qquad \forall\,(a,\mathbf y)\in[0,2C_\zeta]\times\R^d. \notag
\end{align}
As a matter of fact, suppose for a moment that \eqref{FirstIneq} holds true. Then, applying \eqref{FirstIneq} with $a:=\|\boldsymbol x(\cdot\wedge t) - \boldsymbol x_0(\cdot\wedge t\wedge t_0)\|_\infty$ and $\mathbf y_s:=\boldsymbol x(t) - \boldsymbol x_0(s\wedge t_0)$, for every $s\in[t,T]$, and taking the maximum over $s\in[t,T]$, we find (using \eqref{norm_infty})
\begin{align*}
\kappa_\infty^{(t_0,\boldsymbol x_0)}(t,\boldsymbol x) \ &= \ \int_{\R^d}\max\Big\{a,\max_{t\leq s\leq T}|\mathbf y_s - \mathbf z|\Big\}\,\zeta(\mathbf z)\,d\mathbf z - \int_{\R^d}|\mathbf z|\,\zeta(\mathbf z)\,d\mathbf z \notag \\
&= \ \int_{\R^d}\max_{t\leq s\leq T}\big\{\max\{a,|\mathbf y_s - \mathbf z|\}\big\}\,\zeta(\mathbf z)\,d\mathbf z - \int_{\R^d}|\mathbf z|\,\zeta(\mathbf z)\,d\mathbf z \notag \\
&\geq \ \max_{t\leq s\leq T}\bigg\{\int_{\R^d} \max\big\{a,|\mathbf y_s - \mathbf z|\big\}\,\zeta(\mathbf z)\,d\mathbf z\bigg\} - \int_{\R^d}|\mathbf z|\,\zeta(\mathbf z)\,d\mathbf z \notag \\
&= \ \max_{t\leq s\leq T}\bigg\{\int_{\R^d} \max\big\{a,|\mathbf y_s - \mathbf z|\big\}\,\zeta(\mathbf z)\,d\mathbf z - \int_{\R^d}|\mathbf z|\,\zeta(\mathbf z)\,d\mathbf z\bigg\} \notag \\
&\geq \ \max_{t\leq s\leq T}\Big\{\alpha_d\min\{a^{d+1},a\} + \alpha_d\min\{|\mathbf y_s|^{d+1},|\mathbf y_s|\}\Big\} \notag \\
&= \ \alpha_d\min\{a^{d+1},a\} + \alpha_d\max_{t\leq s\leq T}\Big\{\min\{|\mathbf y_s|^{d+1},|\mathbf y_s|\}\Big\} \notag \\
&= \ \alpha_d\min\{a^{d+1},a\} + \alpha_d\min\Big\{\max_{t\leq s\leq T}|\mathbf y_s|^{d+1},\max_{t\leq s\leq T}|\mathbf y_s|\Big\}.
\end{align*}
Hence, by the elementary inequality
\[
\min\{a^{d+1},a\} + \min\{b^{d+1},b\} \ \geq \ \min\big\{\max\{a^{d+1},b^{d+1}\},\max\{a,b\}\big\}, \qquad \forall\,a,b\geq0,
\]
we conclude that
\begin{align*}
\kappa_\infty^{(t_0,\boldsymbol x_0)}(t,\boldsymbol x) \ &\geq \ \alpha_d\,\min\Big\{\max\Big\{a^{d+1},\max_{t\leq s\leq T}|\mathbf y_s|^{d+1}\Big\},\max\Big\{a,\max_{t\leq s\leq T}|\mathbf y_s|\Big\}\Big\} \\
&= \ \alpha_d\,\min\big\{\|\boldsymbol x(\cdot\wedge t) - \boldsymbol x_0(\cdot\wedge t_0)\|_\infty^{d+1},\|\boldsymbol x(\cdot\wedge t) - \boldsymbol x_0(\cdot\wedge t_0)\|_\infty\big\},
\end{align*}
where the last equality follows from \eqref{norm_infty} with $\mathbf z=0$. This yields the first inequality in \eqref{Ineq12}. It remains to prove \eqref{FirstIneq}.

\vspace{1mm} 

\noindent\textsc{Step VI-2.} \emph{Proof of \eqref{FirstIneq}.} For every
positive integer $d$ and $a\geq0$, let $G_a\colon\R^d\rightarrow\R$ be given by
\begin{equation}\label{G_a_definition}
G_a(\mathbf y) \ := \ \int_{\R^d}\max\big\{a,|\mathbf y - \mathbf z|\big\}\,\zeta(\mathbf z)\,d\mathbf z - \int_{\R^d}|\mathbf z|\,\zeta(\mathbf z)\,d\mathbf z, \qquad \forall\,\mathbf y\in\R^d.
\end{equation}
Moreover, let $F_d\colon[0,+\infty)\rightarrow\R$ be defined as (differently to the notation used for $G_a$, we emphasize the dependence of $F_d$ on the dimension $d$; we do this because of statement \eqref{F_condition} below which changes with $d$)
\begin{equation}\label{F_definition}
F_d(a) \ := \ G_a(\mathbf 0) \ = \ \int_{\R^d}\max\big\{a,|\mathbf z|\big\}\,\zeta(\mathbf z)\,d\mathbf z - \int_{\R^d}|\mathbf z|\,\zeta(\mathbf z)\,d\mathbf z, \qquad \forall\,a\in[0,+\infty).
\end{equation}
Notice that $G_a$ and $F_d$ are convex functions on their domains.

Let us fix some notations. We denote by $\partial_{\mathbf y}G_a(\mathbf y)$ and $\partial_{\mathbf y\mathbf y} G_a(\mathbf y)$ (resp. $F_d'(a)$, $F_d''(a)$, $\dots$, $F_d^{(n)}(a)$) the gradient and Hessian (resp. first-order derivative, second-order derivative, $\dots$, $n$-th order derivative) of $G_a$ at $\mathbf y$ (resp. $F_d$ at $a$). When $a=0$, $F_d'(a)$, $F_d''(a)$, $\dots$, $F_d^{(n)}(a)$ are right-derivatives. We also denote by $I$ the $d\times d$ identity matrix. Finally, given $A$ and $B$ in $\mathcal S(d)$ (the set of symmetric $d\times d$ matrices), the inequality $B\leq A$ means that the symmetric matrix $A-B$ is positive semi-definite.

Our aim is to prove the following: \emph{for every $d$, there exist constants $\beta_d>0$ and $L_d>0$ such that}
\begin{align}\label{G_a_condition}
\text{\emph{$\forall\,a\in[0,2C_\zeta]$, $G_a\in C^2(\R^d)$, $\partial_{\mathbf y}G_a(\mathbf 0)=0$, $\partial_{\mathbf y\mathbf y} G_a(\mathbf 0)\geq\beta_d I$ and}} \\
\partial_{\mathbf y\mathbf y} G_a(\mathbf y) - \partial_{\mathbf y\mathbf y} G_a(\mathbf 0) \ \geq \ - L_d\,|\mathbf y|\,I, \qquad \forall\,\mathbf y\in\R^d \notag
\end{align}
\emph{and}
\begin{align}\label{F_condition}
\text{\emph{$F_d\in C^{d+1}([0,+\infty))$, $F_d'(0)=\cdots=F_d^{(d)}(0)=0$, $F_d^{(d+1)}(0)\geq\beta_d$ and}} \\
F_d^{(d+1)}(a) - F_d^{(d+1)}(0) \ \geq \ -L_d\,a, \qquad \forall\,a\geq0. \notag
\end{align}
Suppose for a moment that \eqref{G_a_condition} and \eqref{F_condition} hold. Then, by \eqref{G_a_condition} we show below
 that there exist some constants $\delta_d,\tilde\delta_d\in(0,1]$ such that
\begin{align}\label{Ineq_G_a}
G_a(\mathbf y) \ &\geq \ G_a(\mathbf 0) + \frac{1}{4}\,\beta_d\,|\mathbf y|^2, \qquad\hspace{1.3cm} \forall\,|\mathbf y|\leq\delta_d,\;\forall\,a\in[0,2C_\zeta], \\
F_d(a) \ &\geq \ F_d(0) + \frac{1}{2(d+1)!}\,\beta_d\,a^{d+1}, \qquad \forall\,a\in[0,\tilde\delta_d]. \label{Ineq_F_d}
\end{align}
As a matter of fact, for every fixed $\mathbf y\in\R^d$, set $\varphi_a(\lambda):=G_a(\lambda\mathbf y)$, for every $\lambda\in\R$. Since $\varphi_a\in C^2(\R)$, the Taylor expression given by
\[
\varphi_a(1) \ = \ \varphi_a(0) + \varphi_a'(0)
 + \frac{1}{2}\varphi_a''(0) + \int_0^1 (1 - \lambda)\,(\varphi_a''(\lambda) - \varphi_a''(0))\,d\lambda,
\]
which written in terms of $G_a$ becomes (denoting by $\langle\cdot,\cdot\rangle$ the scalar product in $\R^d$)
\begin{align*}
G_a(\mathbf y) \ &= \ G_a(\mathbf 0) + \langle \partial_{\mathbf y}G_a(\mathbf 0),\mathbf y\rangle + \frac{1}{2}\langle \partial_{\mathbf y\mathbf y} G_a(\mathbf 0)\mathbf y,\mathbf y\rangle \\
&\quad \ + \int_0^1 (1 - \lambda)\,\langle (\partial_{\mathbf y\mathbf y} G_a(\lambda\mathbf y) - \partial_{\mathbf y\mathbf y} G_a(\mathbf 0))\mathbf y,\mathbf y\rangle\,d\lambda \\
&\geq \ G_a(\mathbf 0) + \langle\partial_{\mathbf y}G_a(\mathbf 0),\mathbf y\rangle + \frac{1}{2}\beta_d|\mathbf y|^2 - L_d|\mathbf y|^3\int_0^1\lambda\,(1-\lambda)\,d\lambda \\
&= \ G_a(\mathbf 0) + \frac{1}{2}\beta_d|\mathbf y|^2 - \frac{1}{6}L_d|\mathbf y|^3.
\end{align*}
Hence
\[
G_a(\mathbf y) \ \geq \ G_a(\mathbf 0) + \frac{1}{4}\,\beta_d\,|\mathbf y|^2, \qquad \forall\,|\mathbf y|\leq\delta_d,\text{ where }\delta_d:=1\wedge\bigg(\frac{3}{2}\frac{\beta_d}{L_d}\bigg).
\]
This proves \eqref{Ineq_G_a}. Similarly, we consider the following
Taylor expression for $F_d$:
\[
F_d(a) \ = \ \sum_{k=0}^{d+1} \frac{F_d^{(k)}(0)}{k!}\,a^k + \frac{1}{d!}\int_0^a \big(F_d^{(d+1)}(b) - F_d^{(d+1)}(0)\big)\,(a - b)^d\,db. 
\]
By \eqref{F_condition} we obtain
\begin{align*}
F_d(a) \ &= \ F_d(0) + \frac{F_d^{(d+1)}(0)}{(d+1)!}\,a^{d+1} + \frac{1}{d!}\int_0^a \big(F_d^{(d+1)}(b) - F_d^{(d+1)}(0)\big)\,(a - b)^d\,db \\
&\geq \ F_d(0) + \frac{\beta_d}{(d+1)!}\,a^{d+1} - \frac{L_d}{d!}\int_0^a b\,(a - b)^d\,db \\
&= \ F_d(0) + \frac{\beta_d}{(d+1)!}\,a^{d+1} - \frac{L_d}{(d+2)!}\,a^{d+2}.
\end{align*}
Hence
\[
F_d(a) \ \geq \ F_d(0) + \frac{1}{2}\frac{\beta_d}{(d+1)!}\,a^{d+1}, \qquad \forall\,a\in[0,\tilde\delta_d],\text{ where }\tilde\delta_d:=1\wedge\bigg(\frac{d+2}{2}\frac{\beta_d}{L_d}\bigg).
\]
This proves \eqref{Ineq_F_d}. Now, we notice that from \eqref{Ineq_G_a} we have
\[
G_a(\mathbf y) \ \geq \ G_a(\mathbf 0) + \frac{1}{4}\,\beta_d\,|\mathbf y|^{d+1}, \qquad \forall\,|\mathbf y|\leq\delta_d,\;\forall\,a\in[0,2C_\zeta].
\]
Moreover, since $G_a$ is a convex function, it follows that
\begin{align}\label{G_a>=G_a}
G_a(\mathbf y) \ &\geq \ \min\Big(G_a(\mathbf 0) + \frac{1}{4}\,\beta_d\,|\mathbf y|^{d+1},G_a(\mathbf 0) + \frac{1}{4}\,\beta_d\,\delta_d^d\,|\mathbf y|\Big) \notag \\
&\geq \ \min\Big(G_a(\mathbf 0) + \frac{1}{4}\,\beta_d\,\delta_d^d\,|\mathbf y|^{d+1},G_a(\mathbf 0) + \frac{1}{4}\,\beta_d\,\delta_d^d\,|\mathbf y|\Big) \notag \\
&= \ G_a(\mathbf 0) + \frac{1}{4}\,\beta_d\,\delta_d^d\,\big(|\mathbf y|^{d+1}\wedge|\mathbf y|\big), \qquad\qquad \forall\,\mathbf y\in\R^d.
\end{align}
Proceeding along the same lines, we deduce by \eqref{Ineq_F_d} that
\begin{equation}\label{F>=F}
F_d(a) \ \geq \ F_d(0) + \frac{1}{2}\,\frac{\beta_d}{(d+1)!}\,\tilde\delta_d^d\,\big(a^{d+1}\wedge a\big), \qquad \forall\,a\geq0.
\end{equation}
So, in particular, since $G_a(\mathbf 0)=F_d(a)$ and $F_d(0)=0$, we obtain, from \eqref{G_a>=G_a} and \eqref{F>=F},
\[
G_a(\mathbf y) \geq \frac{1}{2}\,\frac{\beta_d}{(d+1)!}\,\tilde\delta_d^d\,(a^{d+1}\wedge a) + \frac{1}{2}\,\frac{\beta_d}{(d+1)!}\,\delta_d^d\,\big(|\mathbf y|^{d+1}\wedge|\mathbf y|\big), \quad \forall\,(a,\mathbf y)\in[0,2C_\zeta]\times\R^d,
\]
which proves \eqref{FirstIneq} with $\alpha_d:=\frac{1}{2}\frac{\beta_d}{(d+1)!}(\tilde\delta_d^d\wedge\delta_d^d)$ for the case $a\leq2C_\zeta$. It remains to prove \eqref{G_a_condition} and \eqref{F_condition}.

\vspace{1mm}

\noindent\textsc{Step VI-3.} \emph{Proof of \eqref{F_condition}.} From the definition \eqref{F_definition} of $F_d$ we see that $F_d$ is continuous. Moreover, by direct calculation we find
\begin{align*}
\lim_{h\rightarrow0^+}\frac{F_d(a+h) - F_d(a)}{h} \ &= \ \int_{|\mathbf z|\leq a}\zeta(\mathbf z)\,d\mathbf z, \qquad \forall\,a\geq0, \\
\lim_{h\rightarrow0^+}\frac{F_d(a-h) - F_d(a)}{-h} \ &= \ \int_{|\mathbf z|\leq a}\zeta(\mathbf z)\,d\mathbf z, \qquad \forall\,a>0.
\end{align*}
Hence, the first derivative of $F_d$ exists everywhere and is given by $F_d'(a)=\int_{|\mathbf z|\leq a}\zeta(\mathbf z)\,d\mathbf z$, $\forall\,a\geq0$. Notice that $F_d'(0)=0$. We also see that $F_d'$ is continuous on $[0,+\infty)$. Now, for every $r>0$ let $S_{d-1}(r)$ denote the surface area of the boundary of the ball $\{\mathbf z\in\R^d\colon|\mathbf z|\leq r\}$, which is given by $S_{d-1}(r)=\frac{2\pi^{d/2}}{\Gamma(d/2)}r^{d-1}$, where $\Gamma(\cdot)$ is the Gamma function. Then, recalling that $\zeta$ is a radial function and using $d$-dimensional spherical coordinates (see for instance Appendix C.3 in \cite{Evans}), we get
\[
F_d'(a) \ = \ \int_0^a \frac{1}{(2\pi)^\frac{d}{2}}\,\textup{e}^{-\frac{1}{2}r^2}\,S_{d-1}(r)\,dr, \qquad \forall\,a\geq0.
\]
So, in particular, the second derivative of $F_d$ exists everywhere and is given by
\[
F_d''(a) \ = \ \frac{1}{(2\pi)^{\frac{d}{2}}}\,\textup{e}^{-\frac{1}{2}a^2}\,S_{d-1}(a) \ = \ \frac{2^{1-\frac{d}{2}}}{\Gamma\left(\frac{d}{2}\right)}\,a^{d-1}\,\textup{e}^{-\frac{1}{2}a^2}, \qquad \forall\,a\geq0.
\]
We deduce that $F_d\in C^\infty([0,+\infty))$. We also observe that every derivative of $F_d$ is bounded, so in particular $F_d^{(d+1)}$ is Lipschitz. As a consequence, there exists $L_d>0$ such that
\[
F_d^{(d+1)}(a) - F_d^{(d+1)}(0) \ \geq \ -L_d\,a, \qquad \forall\,a\geq0.
\]
Finally, let us prove by induction on $d$ that $F_d'(0)=\cdots=F_d^{(d)}(0)=0$ and $F_d^{(d+1)}(0)>0$. For $d=1$ we have, by direct calculation, $F_1(0)=F_1'(0)=0$ and $F_1''(0)=1/\sqrt{2\pi}>0$. Let us now suppose that the claim holds true for $F_d$, for some $d\geq1$, and let us prove it for $F_{d+1}$. By the explicit expressions of $F_{d+1}$ and $F_{d+1}'$ we see that $F_{d+1}(0)=F_{d+1}'(0)=0$. Moreover
\[
F_{d+1}''(a) \ = \ C_{d+1}\,a^d\,\textup{e}^{-\frac{1}{2}a^2},
\]
where $C_{d+1}:=2^{1-(d+1)/2}/\Gamma((d+1)/2)>0$. So, in particular, $F_{d+1}''(0)=0$. Now, we observe that
\[
F_{d+1}'''(a) \ = \ d\,C_{d+1}\,a^{d-1}\,\textup{e}^{-\frac{1}{2}a^2} - C_{d+1}\,a^{d+1}\,\textup{e}^{-\frac{1}{2}a^2} \ = \ \frac{C_{d+1}}{C_d}\,(d-a^2)\,F_d''(a),
\]
where $C_d:=2^{1-d/2}/\Gamma(d/2)>0$. Therefore
\[
F_{d+1}^{\text{\textup{iv}}}(a) \ = \ \frac{C_{d+1}}{C_d}\,\big((d-a^2)\,F_d'''(a) - 2\,a\,F_d''(a)\big).
\]
Moreover, by the general Leibniz rule, we have
\[
F_{d+1}^{(3+n)}(a) \ = \ \frac{C_{d+1}}{C_d} \sum_{k=0}^{n} \binom{n}{k}\,(d-a^2)^{(n-k)}\,F_d^{(2+k)}(a), \qquad \text{for every }n\geq2,
\]
where $(d-a^2)^{(n-k)}$ denotes the $(n-k)$-th derivative of the map $a\mapsto d-a^2$. Since $(d-a^2)^{(n-k)}$ is identically equal to zero whenever $n-k\geq3$, it follows that
\begin{align*}
F_{d+1}^{(3+n)}(a) \ &= \ \frac{C_{d+1}}{C_d} \bigg(\binom{n}{n-2}\,(d-a^2)^{(2)}\,F_d^{(2+n-2)}(a) \\
&\quad \ + \binom{n}{n-1}\,(d-a^2)^{(1)}\,F_d^{(2+n-1)}(a) + \binom{n}{n}\,(d-a^2)\,F_d^{(2+n)}(a)\bigg) \\
&= \ \frac{C_{d+1}}{C_d} \big(-n\,(n-1)\,F_d^{(2+n-2)}(a) - 2\,n\,a\,F_d^{(2+n-1)}(a) + (d-a^2)\,F_d^{(2+n)}(a)\big).
\end{align*}
In conclusion, we have
\begin{align*}
F_{d+1}'''(0) \ &= \ \frac{C_{d+1}}{C_d}\,d\,F_d''(0), \\
F_{d+1}^{\text{\textup{iv}}}(0) \ &= \ \frac{C_{d+1}}{C_d}\,d\,F_d'''(0), \\
F_{d+1}^{(3+n)}(0) \ &= \ \frac{C_{d+1}}{C_d} \big(-n\,(n-1)\,F_d^{(2+n-2)}(0) + d\,F_d^{(2+n)}(0)\big), \qquad \text{for every }n\geq2.
\end{align*}
From the formulae above it is straightforward to see that the claim holds. This concludes the proof of \eqref{F_condition}.

\vspace{1mm} 

\noindent\textsc{Step VI-4.} \emph{Proof of \eqref{G_a_condition}.} From the definition \eqref{G_a_definition} of $G_a$ we see that $G_a\in C^\infty(\R^d)$. Moreover, we have, for every $i,j=1,\ldots,d$,
\begin{align*}
\partial_{y_i}G_a(\mathbf y) &= -\int_{\R^d}\max\big\{a,|\mathbf y - \mathbf z|\big\}\,\partial_{z_i}\zeta(\mathbf z)\,d\mathbf z = \int_{\R^d}\max\big\{a,|\mathbf y - \mathbf z|\big\}\,z_i\,\zeta(\mathbf z)\,d\mathbf z, \\
\partial_{y_i y_j} G_a(\mathbf y) &= \int_{\R^d}\max\big\{a,|\mathbf y - \mathbf z|\big\}\,\partial_{z_i z_j}\zeta(\mathbf z)\,d\mathbf z = \int_{\R^d}\max\big\{a,|\mathbf y - \mathbf z|\big\}\,\big(z_i\,z_j - \delta_{ij}\big)\,\zeta(\mathbf z)\,d\mathbf z,
\end{align*}
where $\delta_{ij}$ is the Kronecker delta. Since $\zeta$ is a radial function, we have $\zeta(\mathbf z)=\zeta(-\mathbf z)$, for every $\mathbf z\in\R^d$, therefore $\partial_{y_i}G_a(\mathbf 0)=0$.

We now prove that for every fixed $d$ there exists $L_d>0$ such that, for every $a\geq0$, we have
\[
\partial_{\mathbf y\mathbf y} G_a(\mathbf y) - \partial_{\mathbf y\mathbf y} G_a(\mathbf 0) \ \geq \ - L_d|\mathbf y|I, \qquad \forall\,\mathbf y\in\R^d,
\]
which can be equivalently written as
\begin{equation}\label{G_a_Lip}
\sum_{i,j=1}^d \big(\partial_{y_i y_j} G_a(\mathbf y) - \partial_{y_i y_j} G_a(\mathbf 0)\big)w_i w_j \ \geq \ - L_d|\mathbf y||\mathbf w|^2, \qquad \forall\,\mathbf y,\mathbf w\in\R^d,
\end{equation}
where $y_i$ (resp. $w_i$) denotes the $i$-th component of $\mathbf y$ (resp. $\mathbf w$). We start noticing
 that, for every $i,j=1,\ldots,d$, we have (we use the elementary inequality $|\max\{a,b+c\}-\max\{a,c\}|\leq|b|$, valid for every $a,c\geq0$ and $b\in\R$, with $b=|\mathbf y-\mathbf z|-|\mathbf z|$ and $c=|\mathbf z|$)
\begin{align*}
\big|\partial_{y_i y_j} G_a(\mathbf y) - \partial_{y_i y_j} G_a(\mathbf 0)\big| &\leq \int_{\R^d}\big|\max\{a,|\mathbf y - \mathbf z|\} - \max\{a,|\mathbf z|\}\big|\,\big|\partial_{z_i z_j}\zeta(\mathbf z)\big|\,d\mathbf z \\
&= \int_{\R^d}\big|\max\{a,|\mathbf y - \mathbf z| - |\mathbf z| + |\mathbf z|\} - \max\{a,|\mathbf z|\}\big|\,\big|\partial_{z_i z_j}\zeta(\mathbf z)\big|\,d\mathbf z \\
&\leq \int_{\R^d}|\mathbf y|\,|\partial_{z_i z_j}\zeta(\mathbf z)|\,d\mathbf z \ = \ \frac{L_d}{d}\,|\mathbf y|,
\end{align*}
with $\frac{L_d}{d}=\int_{\R^d}|\partial_{z_i z_j}\zeta(\mathbf z)|\,d\mathbf z$. Then, for every $\mathbf w\in\R^d$, we obtain
\[
\bigg|\sum_{i,j=1}^d \big(\partial_{y_i y_j} G_a(\mathbf y) - \partial_{y_i y_j} G_a(\mathbf 0)\big)w_i w_j\bigg| \ \leq \ \frac{L_d}{d}\,|\mathbf y|\sum_{i,j=1}^d|w_i||w_j| \ \leq \ L_d\,|\mathbf y|\,|\mathbf w|^2,
\]
which proves \eqref{G_a_Lip}. 

Finally, we prove that for every fixed $d$ there exists $\hat\beta_d>0$ such that, for every $a\in[0,2C_\zeta]$,
\begin{equation}\label{G_a_D^2}
\partial_{\mathbf y\mathbf y} G_a(\mathbf 0) \ \geq \ \hat\beta_d I.
\end{equation}
As a matter of fact, for every $\mathbf w\in\R^d$, we have
\begin{align*}
&\langle \partial_{\mathbf y\mathbf y} G_a(\mathbf 0)\mathbf w,\mathbf w\rangle = \sum_{i,j=1}^d \partial_{y_i y_j} G_a(\mathbf 0)\,w_i\,w_j	 = \sum_{i,j=1}^d w_i\,w_j \int_{\R^d}\max\big\{a,|\mathbf z|\big\}\,\big(z_i\,z_j - \delta_{ij}\big)\,\zeta(\mathbf z)\,d\mathbf z	\\
&= \sum_{i=1}^d w_i^2 \int_{\R^d}\max\big\{a,|\mathbf z|\big\}\,\big(z_1^2 - 1\big)\,\zeta(\mathbf z)\,d\mathbf z - \sum_{i\neq j} w_i\,w_j \int_{\R^d}\max\big\{a,|\mathbf z|\big\}\,z_1\,z_2\,\zeta(\mathbf z)\,d\mathbf z.
\end{align*}
Now, notice that $\int_{\R^d}\max\{a,|\mathbf z|\}\,z_1\,z_2\,\zeta(\mathbf z)\,d\mathbf z=0$, for every $a\geq0$. Hence
\begin{align*}
\langle\partial_{\mathbf y\mathbf y} G_a(\mathbf 0)\mathbf w,\mathbf w\rangle \ &= \ |\mathbf w|^2 \int_{\R^d}\max\big\{a,|\mathbf z|\big\}\,\big(z_1^2 - 1\big)\,\zeta(\mathbf z)\,d\mathbf z \\
&= \ |\mathbf w|^2 \frac{1}{d}\sum_{i=1}^d\int_{\R^d}\max\big\{a,|\mathbf z|\big\}\,\big(z_i^2 - 1\big)\,\zeta(\mathbf z)\,d\mathbf z \\
&= \ |\mathbf w|^2 \frac{1}{d} \int_{\R^d}\max\big\{a,|\mathbf z|\big\}\,\big(|\mathbf z|^2 - d\big)\,\zeta(\mathbf z)\,d\mathbf z.
\end{align*}
Let $H\colon[0,+\infty)\rightarrow\R$ be defined as
\[
H(a) \ := \ \int_{\R^d}\max\big\{a,|\mathbf z|\big\}\,\big(|\mathbf z|^2 - d\big)\,\zeta(\mathbf z)\,d\mathbf z, \qquad \forall\,a\in[0,+\infty).
\]
Notice that \eqref{G_a_D^2} follows if we prove the following (actually, it would be enough to require that $H(a)>0$ for every $a\geq0$ and $H$ decreasing; \eqref{G_a_D^2_proof} is a sufficient condition for this):
\begin{equation}\label{G_a_D^2_proof}
\text{\emph{$H(0)>0$, $\lim_{a\rightarrow+\infty}H(a)=0$, $H$ is a strictly decreasing function}}.
\end{equation}
As a matter of fact, if \eqref{G_a_D^2_proof} holds, then
\[
\inf_{a\in[0,2C_\zeta]} H(a) \ = \ H(2C_\zeta) \ > \ 0,
\]
from which \eqref{G_a_D^2} follows with $\hat\beta_d=\frac{1}{d}H(2C_\zeta)$.

It remains to prove \eqref{G_a_D^2_proof}. Denoting by $\mu_{\chi^2(d),p}$ the moment of order $p>0$ of a $\chi^2$-distribution with $d$ degrees of freedom, and recalling that $\Gamma(\cdot)$ is the Gamma function, we have
\[
H(0) \ = \ \mu_{\chi^2(d),\frac{3}{2}} - d\,\mu_{\chi^2(d),\frac{1}{2}} \ = \ 2^{\frac{3}{2}}\,\frac{\Gamma\big(\frac{d}{2} + \frac{3}{2}\big)}{\Gamma\big(\frac{d}{2}\big)} - d\,2^{\frac{1}{2}}\,\frac{\Gamma\big(\frac{d}{2} + \frac{1}{2}\big)}{\Gamma\big(\frac{d}{2}\big)} \ = \ \mu_{\chi^2(d),\frac{1}{2}}  \ > \ 0.
\]
Concerning the function $H$, we also have that (we perform the change of variables $\mathbf z=a\mathbf w$ under the integral sign)
\[
\lim_{a\rightarrow+\infty} H(a) \ = \ \lim_{a\rightarrow+\infty} \int_{\R^d} a^2\,\max\big\{1,|\mathbf w|\big\}\,\big(a^2\,|\mathbf w|^2 - d\big)\,\zeta(a\mathbf w)\,d\mathbf w \ = \ 0,
\]
where the limit follows from an application of the Lebesgue dominated convergence theorem.

Now, proceeding as in Step VI-3 for the function $F$, we deduce that $H\in C^\infty([0,+\infty))$ and, for every $a\geq0$,
\[
H'(a) \ = \ \int_{|\mathbf z|\leq a}\big(|\mathbf z|^2 - d\big)\,\zeta(\mathbf z)\,d\mathbf z, \qquad H''(a) \ = \ \frac{a^2 - d}{(2\pi)^{\frac{d}{2}}}\,\textup{e}^{-\frac{1}{2}a^2}\,S_{d-1}(a),
\]
where $S_{d-1}(a)$ denotes the surface area of the boundary of the ball $\{\mathbf z\in\R^d\colon|\mathbf z|\leq a\}$. Notice that
\[
H'(0) \ = \ 0, \qquad\qquad \lim_{a\rightarrow+\infty} H'(a) \ = \ \int_{\R^d}\big(|\mathbf z|^2 - d\big)\,\zeta(\mathbf z)\,d\mathbf z \ = \ \mu_{\chi^2(d),1} - d \ = \ 0.
\]
Then, we deduce from the sign of $H''$ that $H'(a)<0$, for every $a>0$. This implies that $H$ is a strictly decreasing function and concludes the proof.
\end{proof}

\subsection{Proof of Lemma \ref{L:Smoothing_rho_infty}}
\label{AppVar2}

\begin{proof}[Proof of Lemma \ref{L:Smoothing_rho_infty}]
We split the proof into three steps.

\vspace{1mm}

\noindent\textsc{Step I}. \emph{Proof of item }1).
We first notice that $\hat\chi_\infty^{(t_0,\boldsymbol x_0)}$ is a non-anticipative map. Now, let $(t,\boldsymbol{\hat x})\in\boldsymbol{\hat\Lambda}$ and $\delta>0$, with $t<T$ and $\delta\leq T-t$, then we have
\begin{align*}
&\frac{\hat\chi_\infty^{(t_0,\boldsymbol x_0)}(t+\delta,\boldsymbol{\hat x}(\cdot\wedge t)) - \hat\chi_\infty^{(t_0,\boldsymbol x_0)}(t,\boldsymbol{\hat x})}{\delta} \\
&= \ \int_0^{+\infty} \frac{\hat\kappa_\infty^{(t_0,\boldsymbol x_0)}\big((t + s)\wedge T,\boldsymbol{\hat x}(\cdot\wedge t)\big)}{1 + C_\zeta + \hat\kappa_\infty^{(t_0,\boldsymbol x_0)}\big((t + s)\wedge T,\boldsymbol{\hat x}(\cdot\wedge t)\big)}\,\frac{\eta(s - \delta) - \eta(s)}{\delta}\,ds \\
&\quad \ - \frac{1}{\delta}\int_0^\delta \frac{\hat\kappa_\infty^{(t_0,\boldsymbol x_0)}\big((t + s)\wedge T,\boldsymbol{\hat x}(\cdot\wedge t)\big)}{1 + C_\zeta + \hat\kappa_\infty^{(t_0,\boldsymbol x_0)}\big((t + s)\wedge T,\boldsymbol{\hat x}(\cdot\wedge t)\big)}\,\eta(s - \delta)\,ds.
\end{align*}
Notice that (writing $\frac{\eta(s - \delta) - \eta(s)}{\delta}=-\int_{s-\delta}^s \eta'(\tilde s)\,d\tilde s$, using Fubini's theorem and the integrability of $\eta'$, namely $\int_0^{+\infty}|\eta'(s)|\,ds<+\infty$)
\begin{align*}
&\int_0^{+\infty} \frac{\hat\kappa_\infty^{(t_0,\boldsymbol x_0)}\big((t + s)\wedge T,\boldsymbol{\hat x}(\cdot\wedge t)\big)}{1 + C_\zeta + \hat\kappa_\infty^{(t_0,\boldsymbol x_0)}\big((t + s)\wedge T,\boldsymbol{\hat x}(\cdot\wedge t)\big)}\,\frac{\eta(s - \delta) - \eta(s)}{\delta}\,ds \\
&\hspace{3.5cm} \overset{\delta\rightarrow0^+}{\longrightarrow} \ - \int_0^{+\infty} \frac{\hat\kappa_\infty^{(t_0,\boldsymbol x_0)}\big((t + s)\wedge T,\boldsymbol{\hat x}(\cdot\wedge t)\big)}{1 + C_\zeta + \hat\kappa_\infty^{(t_0,\boldsymbol x_0)}\big((t + s)\wedge T,\boldsymbol{\hat x}(\cdot\wedge t)\big)}\,\eta'(s)\,ds.
\end{align*}
We also have, using that $\frac{\hat\kappa_\infty^{(t_0,\boldsymbol x_0)}}{1+C_\zeta+\hat\kappa_\infty^{(t_0,\boldsymbol x_0)}}\leq1$ and $\eta(s-\delta)=(s-\delta)\,\textup{e}^{-(s-\delta)}\geq-\delta\,\textup{e}^\delta=\eta(-\delta)$, for every $0\leq s\leq\delta$,
\[
\bigg|\frac{1}{\delta}\int_0^\delta \frac{\hat\kappa_\infty^{(t_0,\boldsymbol x_0)}\big((t + s)\wedge T,\boldsymbol{\hat x}(\cdot\wedge t)\big)}{1 + C_\zeta + \hat\kappa_\infty^{(t_0,\boldsymbol x_0)}\big((t + s)\wedge T,\boldsymbol{\hat x}(\cdot\wedge t)\big)}\,\eta(s - \delta)ds\,\bigg| \ \leq \ \delta\,\textup{e}^\delta.
\]
This implies that
\[
\frac{1}{\delta}\int_0^\delta \frac{\hat\kappa_\infty^{(t_0,\boldsymbol x_0)}\big((t + s)\wedge T,\boldsymbol{\hat x}(\cdot\wedge t)\big)}{1 + C_\zeta + \hat\kappa_\infty^{(t_0,\boldsymbol x_0)}\big((t + s)\wedge T,\boldsymbol{\hat x}(\cdot\wedge t)\big)}\eta(s - \delta)ds \ \overset{\delta\rightarrow0^+}{\longrightarrow} \ 0.
\]
This proves that the horizontal derivative of $\hat\chi_\infty^{(t_0,\boldsymbol x_0)}$ exists everywhere on $[0,T)\times D([0,T];\R^d)$ and is given by
\begin{equation}\label{HorDer_t<T}
\partial_t^H \hat\chi_\infty^{(t_0,\boldsymbol x_0)}(t,\boldsymbol{\hat x}) \ = \ - \int_0^{+\infty} \frac{\hat\kappa_\infty^{(t_0,\boldsymbol x_0)}\big((t + s)\wedge T,\boldsymbol{\hat x}(\cdot\wedge t)\big)}{1 + C_\zeta + \hat\kappa_\infty^{(t_0,\boldsymbol x_0)}\big((t + s)\wedge T,\boldsymbol{\hat x}(\cdot\wedge t)\big)}\,\eta'(s)\,ds.
\end{equation}
Moreover, at $t=T$ we have
\begin{align}\label{HorDer_t=T}
\partial_t^H \hat\chi_\infty^{(t_0,\boldsymbol x_0)}(T,\boldsymbol{\hat x}) \ &= \ \lim_{t\rightarrow T^-} \partial_t^H \hat\chi_\infty^{(t_0,\boldsymbol x_0)}(t,\boldsymbol{\hat x}) \notag \\
&= \ - \lim_{t\rightarrow T^-}  \int_0^{+\infty} \frac{\hat\kappa_\infty^{(t_0,\boldsymbol x_0)}\big((t + s)\wedge T,\boldsymbol{\hat x}(\cdot\wedge t)\big)}{1 + C_\zeta + \hat\kappa_\infty^{(t_0,\boldsymbol x_0)}\big((t + s)\wedge T,\boldsymbol{\hat x}(\cdot\wedge t)\big)}\,\eta'(s)\,ds \notag \\
&= \ - \int_0^{+\infty} \frac{\hat\kappa_\infty^{(t_0,\boldsymbol x_0)}\big(T,\boldsymbol{\hat x}(\cdot\wedge T-)\big)}{1 + C_\zeta + \hat\kappa_\infty^{(t_0,\boldsymbol x_0)}\big(T,\boldsymbol{\hat x}(\cdot\wedge T-)\big)}\,\eta'(s)\,ds,
\end{align}
where $\boldsymbol{\hat x}(\cdot\wedge T-)$ is given by \eqref{xt-} with $t=T$. In conclusion, the horizontal derivative of $\hat\chi_\infty^{(t_0,\boldsymbol x_0)}$ exists everywhere on $\boldsymbol{\hat\Lambda}$.

Let us now consider the vertical derivatives of $\hat\chi_\infty^{(t_0,\boldsymbol x_0)}$. Given $(t,\boldsymbol{\hat x})\in\boldsymbol{\hat\Lambda}$, $h\in\R\backslash\{0\}$, and $i=1,\ldots,d$, we have (recalling that $\mathbf e_1,\ldots,\mathbf e_d$ denotes the standard orthonormal basis of $\R^d$)
\begin{align*}
&\frac{\hat\chi_\infty^{(t_0,\boldsymbol x_0)}(t,\boldsymbol{\hat x} + h\,\mathbf e_i\,1_{[t,T]}) - \hat\chi_\infty^{(t_0,\boldsymbol x_0)}(t,\boldsymbol{\hat x})}{h} \\
&= \int_0^{+\infty}\frac{\frac{\hat\kappa_\infty^{(t_0,\boldsymbol x_0)}\big((t + s)\wedge T,\boldsymbol{\hat x}(\cdot\wedge t) + h\,\mathbf e_i\,1_{[t,T]}\big)}{1 + C_\zeta + \hat\kappa_\infty^{(t_0,\boldsymbol x_0)}\big((t + s)\wedge T,\boldsymbol{\hat x}(\cdot\wedge t) + h\,\mathbf e_i\,1_{[t,T]}\big)} - \frac{\hat\kappa_\infty^{(t_0,\boldsymbol x_0)}\big((t + s)\wedge T,\boldsymbol{\hat x}(\cdot\wedge t)\big)}{1 + C_\zeta + \hat\kappa_\infty^{(t_0,\boldsymbol x_0)}\big((t + s)\wedge T,\boldsymbol{\hat x}(\cdot\wedge t)\big)}}{h}\eta(s)ds.
\end{align*}
By item 1) of Lemma \ref{L:rho_infty}, we have that
\begin{align*}
\frac{\frac{\hat\kappa_\infty^{(t_0,\boldsymbol x_0)}\big((t + s)\wedge T,\boldsymbol{\hat x}(\cdot\wedge t) + h\,\mathbf e_i\,1_{[t,T]}\big)}{1 + C_\zeta + \hat\kappa_\infty^{(t_0,\boldsymbol x_0)}\big((t + s)\wedge T,\boldsymbol{\hat x}(\cdot\wedge t) + h\,\mathbf e_i\,1_{[t,T]}\big)} - \frac{\hat\kappa_\infty^{(t_0,\boldsymbol x_0)}\big((t + s)\wedge T,\boldsymbol{\hat x}(\cdot\wedge t)\big)}{1 + C_\zeta + \hat\kappa_\infty^{(t_0,\boldsymbol x_0)}\big((t + s)\wedge T,\boldsymbol{\hat x}(\cdot\wedge t)\big)}}{h}& \\
\overset{h\rightarrow0}{\longrightarrow} \ \frac{(1 + C_\zeta)\,\partial^V_{x_i}\hat\kappa_\infty^{(t_0,\boldsymbol x_0)}\big((t + s)\wedge T,\boldsymbol{\hat x}(\cdot\wedge t)\big)}{\big(1 + C_\zeta + \hat\kappa_\infty^{(t_0,\boldsymbol x_0)}\big((t + s)\wedge T,\boldsymbol{\hat x}(\cdot\wedge t)\big)\big)^2}&.
\end{align*}
Then, by the Lebesgue dominated convergence theorem, we obtain
\begin{align*}
\int_0^{+\infty} \frac{\frac{\hat\kappa_\infty^{(t_0,\boldsymbol x_0)}\big((t + s)\wedge T,\boldsymbol{\hat x}(\cdot\wedge t) + h\,\mathbf e_i\,1_{[t,T]}\big)}{1 + C_\zeta + \hat\kappa_\infty^{(t_0,\boldsymbol x_0)}\big((t + s)\wedge T,\boldsymbol{\hat x}(\cdot\wedge t) + h\,\mathbf e_i\,1_{[t,T]}\big)} - \frac{\hat\kappa_\infty^{(t_0,\boldsymbol x_0)}\big((t + s)\wedge T,\boldsymbol{\hat x}(\cdot\wedge t)\big)}{1 + C_\zeta + \hat\kappa_\infty^{(t_0,\boldsymbol x_0)}\big((t + s)\wedge T,\boldsymbol{\hat x}(\cdot\wedge t)\big)}}{h}\,\eta(s)\,ds& \\
\overset{h\rightarrow0}{\longrightarrow} \ \int_0^{+\infty} \frac{(1 + C_\zeta)\,\partial^V_{x_i}\hat\kappa_\infty^{(t_0,\boldsymbol x_0)}\big((t + s)\wedge T,\boldsymbol{\hat x}(\cdot\wedge t)\big)}{\big(1 + C_\zeta + \hat\kappa_\infty^{(t_0,\boldsymbol x_0)}\big((t + s)\wedge T,\boldsymbol{\hat x}(\cdot\wedge t)\big)\big)^2}\,\eta(s)\,ds&.
\end{align*}
This proves that $\hat\chi_\infty^{(t_0,\boldsymbol x_0)}$ admits first-order vertical derivatives at every $(t,\boldsymbol{\hat x})\in\boldsymbol{\hat\Lambda}$, which are given by
\begin{equation}\label{VerDer}
\partial^V_{x_i}\hat\chi_\infty^{(t_0,\boldsymbol x_0)}(t,\boldsymbol{\hat x}) \ = \ \int_0^{+\infty} \frac{(1 + C_\zeta)\,\partial^V_{x_i}\hat\kappa_\infty^{(t_0,\boldsymbol x_0)}\big((t + s)\wedge T,\boldsymbol{\hat x}(\cdot\wedge t)\big)}{\big(1 + C_\zeta + \hat\kappa_\infty^{(t_0,\boldsymbol x_0)}\big((t + s)\wedge T,\boldsymbol{\hat x}(\cdot\wedge t)\big)\big)^2}\,\eta(s)\,ds,
\end{equation}
for every $(t,\boldsymbol{\hat x})\in\boldsymbol{\hat\Lambda}$ and every $i=1,\ldots,d$. In a similar way we can prove that $\hat\chi_\infty^{(t_0,\boldsymbol x_0)}$ also admits second-order vertical derivatives at every $(t,\boldsymbol{\hat x})\in\boldsymbol{\hat\Lambda}$. In particular, it holds that
\begin{align}\label{VerDer2nd}
&\partial^V_{x_i x_j}\hat\chi_\infty^{(t_0,\boldsymbol x_0)}(t,\boldsymbol{\hat x}) = \int_0^{+\infty} \frac{(1 + C_\zeta)\partial^V_{x_i x_j}\hat\kappa_\infty^{(t_0,\boldsymbol x_0)}\big((t + s)\wedge T,\boldsymbol{\hat x}(\cdot\wedge t)\big)}{\big(1 + C_\zeta + \hat\kappa_\infty^{(t_0,\boldsymbol x_0)}\big((t + s)\wedge T,\boldsymbol{\hat x}(\cdot\wedge t)\big)\big)^2}\,\eta(s)\,ds \\
&\!- 2\!\int_0^{+\infty}\! \frac{(1 + C_\zeta)\partial^V_{x_i}\hat\kappa_\infty^{(t_0,\boldsymbol x_0)}\big((t + s)\wedge T,\boldsymbol{\hat x}(\cdot\wedge t)\big)\partial^V_{x_j}\hat\kappa_\infty^{(t_0,\boldsymbol x_0)}\big((t + s)\wedge T,\boldsymbol{\hat x}(\cdot\wedge t)\big)}{\big(1 + C_\zeta + \hat\kappa_\infty^{(t_0,\boldsymbol x_0)}\big((t + s)\wedge T,\boldsymbol{\hat x}(\cdot\wedge t)\big)\big)^3}\eta(s)ds, \notag
\end{align}
for every $(t,\boldsymbol{\hat x})\in\boldsymbol{\hat\Lambda}$ and every $i,j=1,\ldots,d$.

\vspace{1mm}

\noindent\textsc{Step II}. \emph{Proof of item }2). We prove the continuity of the map $\big((t_0,\boldsymbol x_0),(t,\boldsymbol{\hat x})\big)\mapsto\hat\chi_\infty^{(t_0,\boldsymbol x_0)}(t,\boldsymbol{\hat x})$ on $\boldsymbol\Lambda\times\boldsymbol{\hat\Lambda}$, which, in particular, implies the continuity of the map $(t,\boldsymbol{\hat x})\mapsto\hat\chi_\infty^{(t_0,\boldsymbol x_0)}(t,\boldsymbol{\hat x})$ on $\boldsymbol{\hat\Lambda}$. In a similar way, we can prove the continuity on $\boldsymbol{\hat\Lambda}$ of the pathwise derivatives $(t,\boldsymbol{\hat x})\mapsto\partial_t^H\hat\chi_\infty^{(t_0,\boldsymbol x_0)}(t,\boldsymbol{\hat x})$, $(t,\boldsymbol{\hat x})\mapsto\partial_{x_i}^V\hat\chi_\infty^{(t_0,\boldsymbol x_0)}(t,\boldsymbol{\hat x})$, $(t,\boldsymbol{\hat x})\mapsto\partial_{x_i x_j}^V\hat\chi_\infty^{(t_0,\boldsymbol x_0)}(t,\boldsymbol{\hat x})$, for every $i,j=1,\ldots,d$.

Let us prove that the map $\big((t_0,\boldsymbol x_0),(t,\boldsymbol{\hat x})\big)\mapsto\hat\chi_\infty^{(t_0,\boldsymbol x_0)}(t,\boldsymbol{\hat x})$ is continuous on $\boldsymbol\Lambda\times\boldsymbol{\hat\Lambda}$. Fix $(t,\boldsymbol{\hat x}),(s,\boldsymbol{\hat y})\in\boldsymbol{\hat\Lambda}$ and $(t_0,\boldsymbol x_0),(s_0,\boldsymbol y_0)\in\boldsymbol\Lambda$. We observe that, for every $\mathbf z\in\R^d$, we have
\begin{align*}
&\big\|\boldsymbol{\hat x}(\cdot\wedge t) - \boldsymbol x_0(\cdot\wedge t_0) - \mathbf z1_{[t,T]}\big\|_\infty \\
&= \ \max\Big(\sup_{0\leq r<t}\big|\boldsymbol{\hat x}(r) - \boldsymbol x_0(r\wedge t_0)\big|,\max_{t\leq r\leq T}\big|\boldsymbol{\hat x}(t) - \boldsymbol x_0(r\wedge t_0) - \mathbf z\big|\Big) \\
&= \ \max\Big(\big\|\boldsymbol{\hat x}(\cdot\wedge t-) - \boldsymbol x_0(\cdot\wedge t\wedge t_0)\big\|_\infty,\big\|\boldsymbol{\hat x}(t) - \boldsymbol x_0(t\vee(r\wedge t_0)) - \mathbf z\big\|_\infty\Big),
\end{align*}
where $\boldsymbol{\hat x}(\cdot\wedge t-)$ is given by \eqref{xt-}. Similarly
\begin{align*}
&\big\|\boldsymbol{\hat y}(\cdot\wedge s) - \boldsymbol y_0(\cdot\wedge s_0) - \mathbf z1_{[s,T]}\big\|_\infty \\
&= \ \max\Big(\big\|\boldsymbol{\hat y}(\cdot\wedge s-) - \boldsymbol y_0(\cdot\wedge s\wedge s_0)\big\|_\infty,\big\|\boldsymbol{\hat y}(s) - \boldsymbol y_0(s\vee(r\wedge s_0)) - \mathbf z\big\|_\infty\Big).
\end{align*}
As a consequence, using the elementary inequality $\max\{a,b\}-\max\{c,d\}\leq\max\{a-c,b-d\}$, valid for every $a,b,c,d\in\R$, we obtain (recall that the function $\hat\kappa_\infty^{(t_0,\boldsymbol x_0)}$ is defined by \eqref{kappa})
\begin{align*}
&\hat\kappa_\infty^{(t_0,\boldsymbol x_0)}(t,\boldsymbol{\hat x}) - \hat\kappa_\infty^{(s_0,\boldsymbol y_0)}(s,\boldsymbol{\hat y}) \\
&= \int_{\R^d} \Big(\big\|\boldsymbol{\hat x}(\cdot\wedge t) - \boldsymbol{x}_0(\cdot\wedge t_0) - \mathbf z\,1_{[t,T]}\big\|_\infty - \big\|\boldsymbol{\hat y}(\cdot\wedge s) - \boldsymbol{y}_0(\cdot\wedge s_0) - \mathbf z\,1_{[s,T]}\big\|_\infty\Big)\,\zeta(\mathbf z)\,d\mathbf z \\
&= \int_{\R^d} \bigg(\max\Big(\big\|\boldsymbol{\hat x}(\cdot\wedge t-) - \boldsymbol x_0(\cdot\wedge t\wedge t_0)\big\|_\infty,\big\|\boldsymbol{\hat x}(t) - \boldsymbol x_0(t\vee(r\wedge t_0)) - \mathbf z\big\|_\infty\Big) \\
&\quad - \max\Big(\big\|\boldsymbol{\hat y}(\cdot\wedge s-) - \boldsymbol y_0(\cdot\wedge s\wedge s_0)\big\|_\infty,\big\|\boldsymbol{\hat y}(s) - \boldsymbol y_0(s\vee(r\wedge s_0)) - \mathbf z\big\|_\infty\Big)\bigg)\,\zeta(\mathbf z)\,d\mathbf z \\
&\leq \int_{\R^d} \bigg(\max\Big(\big\|\boldsymbol{\hat x}(\cdot\wedge t-) - \boldsymbol x_0(\cdot\wedge t\wedge t_0)\big\|_\infty - \big\|\boldsymbol{\hat y}(\cdot\wedge s-) - \boldsymbol y_0(\cdot\wedge s\wedge s_0)\big\|_\infty,\big\|\boldsymbol{\hat x}(t) \\
&\quad - \boldsymbol x_0(t\vee(r\wedge t_0)) - \mathbf z\big\|_\infty - \mathbf z\big| - \big\|\boldsymbol{\hat y}(s) - \boldsymbol y_0(s\vee(r\wedge s_0)) - \mathbf z\big\|_\infty\Big)\bigg)\,\zeta(\mathbf z)\,d\mathbf z.
\end{align*}
Now, notice that
\begin{align*}
&\big\|\boldsymbol{\hat x}(\cdot\wedge t-) - \boldsymbol x_0(\cdot\wedge t\wedge t_0)\big\|_\infty - \big\|\boldsymbol{\hat y}(\cdot\wedge s-) - \boldsymbol y_0(\cdot\wedge s\wedge s_0)\big\|_\infty \\
&\hspace{3cm}\leq \ \big\|\boldsymbol{\hat x}(\cdot\wedge t-) - \boldsymbol x_0(\cdot\wedge t\wedge t_0) - \boldsymbol{\hat y}(\cdot\wedge s-) + \boldsymbol y_0(\cdot\wedge s\wedge s_0)\big\|_\infty \\
&\hspace{3cm}\leq \ \big\|\boldsymbol{\hat x}(\cdot\wedge t-) - \boldsymbol{\hat y}(\cdot\wedge s-)\big\|_\infty + \big\|\boldsymbol x_0(\cdot\wedge t\wedge t_0) - \boldsymbol y_0(\cdot\wedge s\wedge s_0)\big\|_\infty.
\end{align*}
On the other hand, we have
\begin{align*}
&\big\|\boldsymbol{\hat x}(t) - \boldsymbol x_0(t\vee(r\wedge t_0)) - \mathbf z\big\|_\infty - \big\|\boldsymbol{\hat y}(s) - \boldsymbol y_0(s\vee(r\wedge s_0)) - \mathbf z\big\|_\infty \\
&\hspace{3cm}\leq \ \big\|\boldsymbol{\hat x}(t) - \boldsymbol x_0(t\vee(\cdot\wedge t_0)) - \boldsymbol{\hat y}(s) + \boldsymbol y_0(s\vee(\cdot\wedge s_0))\big\|_\infty \\
&\hspace{3cm}\leq \ \big|\boldsymbol{\hat x}(t) - \boldsymbol{\hat y}(s)\big| + \big\|\boldsymbol x_0(t\vee(\cdot\wedge t_0)) - \boldsymbol y_0(s\vee(\cdot\wedge s_0))\big\|_\infty.
\end{align*}
In conclusion, we obtain
\begin{align*}
\big|\hat\kappa_\infty^{(t_0,\boldsymbol x_0)}(t,\boldsymbol{\hat x}) - \hat\kappa_\infty^{(s_0,\boldsymbol y_0)}(s,\boldsymbol{\hat y})\big| \ \leq \ \big\|\boldsymbol{\hat x}(\cdot\wedge t-) - \boldsymbol{\hat y}(\cdot\wedge s-)\big\|_\infty + \big|\boldsymbol{\hat x}(t) - \boldsymbol{\hat y}(s)\big| & \\
+ \ \big\|\boldsymbol x_0(\cdot\wedge t\wedge t_0) - \boldsymbol y_0(\cdot\wedge s\wedge s_0)\big\|_\infty + \big\|\boldsymbol x_0(t\vee(\cdot\wedge t_0)) - \boldsymbol y_0(s\vee(\cdot\wedge s_0))\big\|_\infty &.
\end{align*}
This implies that
\begin{align*}
&\bigg|\frac{\hat\kappa_\infty^{(t_0,\boldsymbol x_0)}(t,\boldsymbol{\hat x})}{1 + C_\zeta + \hat\kappa_\infty^{(t_0,\boldsymbol x_0)}(t,\boldsymbol{\hat x})} - \frac{\hat\kappa_\infty^{(s_0,\boldsymbol y_0)}(s,\boldsymbol{\hat y})}{1 + C_\zeta + \hat\kappa_\infty^{(s_0,\boldsymbol y_0)}(s,\boldsymbol{\hat y})}\bigg| \notag \\
&= \frac{(1 + C_\zeta)\,\big|\hat\kappa_\infty^{(t_0,\boldsymbol x_0)}(t,\boldsymbol{\hat x}) - \hat\kappa_\infty^{(s_0,\boldsymbol y_0)}(s,\boldsymbol{\hat y})\big|}{\big(1 + C_\zeta + \hat\kappa_\infty^{(t_0,\boldsymbol x_0)}(t,\boldsymbol{\hat x})\big)\big(1 + C_\zeta + \hat\kappa_\infty^{(s_0,\boldsymbol y_0)}(s,\boldsymbol{\hat y})\big)} \notag \\
&\leq (1 + C_\zeta)\frac{\big\|\boldsymbol{\hat x}(\cdot\wedge t-) - \boldsymbol{\hat y}(\cdot\wedge s-)\big\|_\infty + \big|\boldsymbol{\hat x}(t) - \boldsymbol{\hat y}(s)\big|}{\big(1 + C_\zeta + \hat\kappa_\infty^{(t_0,\boldsymbol x_0)}(t,\boldsymbol{\hat x})\big)\big(1 + C_\zeta + \hat\kappa_\infty^{(s_0,\boldsymbol y_0)}(s,\boldsymbol{\hat y})\big)} \\
&\quad + (1 + C_\zeta)\frac{\big\|\boldsymbol x_0(\cdot\wedge t\wedge t_0) - \boldsymbol y_0(\cdot\wedge s\wedge s_0)\big\|_\infty + \big\|\boldsymbol x_0(t\vee(\cdot\wedge t_0)) - \boldsymbol y_0(s\vee(\cdot\wedge s_0))\big\|_\infty}{\big(1 + C_\zeta + \hat\kappa_\infty^{(t_0,\boldsymbol x_0)}(t,\boldsymbol{\hat x})\big)\big(1 + C_\zeta + \hat\kappa_\infty^{(s_0,\boldsymbol y_0)}(s,\boldsymbol{\hat y})\big)}. \notag 
\end{align*}
By the latter inequality, we obtain (denoting $\boldsymbol{\hat x}^t(\cdot):=\boldsymbol{\hat x}(\cdot\wedge t)$ and $\boldsymbol{\hat y}^s(\cdot):=\boldsymbol{\hat y}(\cdot\wedge s)$)
\begin{align}\label{Cont_chi_proof}
&\big|\hat\chi_\infty^{(t_0,\boldsymbol x_0)}(t,\boldsymbol{\hat x}) - \hat\chi_\infty^{(s_0,\boldsymbol y_0)}(s,\boldsymbol{\hat y})\big| \\
&\leq \ \int_0^{+\infty} \bigg|\frac{\hat\kappa_\infty^{(t_0,\boldsymbol x_0)}\big((t + r)\wedge T,\boldsymbol{\hat x}^t\big)}{1 + C_\zeta + \hat\kappa_\infty^{(t_0,\boldsymbol x_0)}\big((t + r)\wedge T,\boldsymbol{\hat x}^t\big)} - \frac{\hat\kappa_\infty^{(s_0,\boldsymbol y_0)}\big((s + r)\wedge T,\boldsymbol{\hat y}^s\big)}{1 + C_\zeta + \hat\kappa_\infty^{(s_0,\boldsymbol y_0)}\big((s + r)\wedge T,\boldsymbol{\hat y}^s\big)}\bigg|\,\eta(r)\,dr \notag \\
&\leq \ \int_0^{+\infty} \frac{(1 + C_\zeta)\,\big\|\boldsymbol{\hat x}^t(\cdot\wedge((t+r)\wedge T)-) - \boldsymbol{\hat y}^s(\cdot\wedge((s+r)\wedge T)-)\big\|_\infty}{\big(1 + C_\zeta + \hat\kappa_\infty^{(t_0,\boldsymbol x_0)}((t + r)\wedge T,\boldsymbol{\hat x}^t)\big)\big(1 + C_\zeta + \hat\kappa_\infty^{(s_0,\boldsymbol y_0)}((s + r)\wedge T,\boldsymbol{\hat y}^s)\big)}\,\eta(r)\,dr \notag \\
&+ \int_0^{+\infty} \frac{(1 + C_\zeta)\,\big|\boldsymbol{\hat x}^t((t+r)\wedge T) - \boldsymbol{\hat y}^s((s+r)\wedge T)\big|}{\big(1 + C_\zeta + \hat\kappa_\infty^{(t_0,\boldsymbol x_0)}((t + r)\wedge T,\boldsymbol{\hat x}^t)\big)\big(1 + C_\zeta + \hat\kappa_\infty^{(s_0,\boldsymbol y_0)}((s + r)\wedge T,\boldsymbol{\hat y}^s)\big)}\,\eta(r)\,dr \notag \\
&+ \int_0^{+\infty} \frac{(1 + C_\zeta)\,\big\|\boldsymbol x_0(\cdot\wedge(t + r)\wedge t_0) - \boldsymbol y_0(\cdot\wedge(s + r)\wedge s_0)\big\|_\infty}{\big(1 + C_\zeta + \hat\kappa_\infty^{(t_0,\boldsymbol x_0)}((t + r)\wedge T,\boldsymbol{\hat x}^t)\big)\big(1 + C_\zeta + \hat\kappa_\infty^{(s_0,\boldsymbol y_0)}((s + r)\wedge T,\boldsymbol{\hat y}^s)\big)}\,\eta(r)\,dr \notag \\
&+ \int_0^{+\infty} \frac{(1 + C_\zeta)\,\big\|\boldsymbol x_0(((t + r)\wedge T)\vee(\cdot\wedge t_0)) - \boldsymbol y_0(((s + r)\wedge T)\vee(\cdot\wedge s_0))\big\|_\infty}{\big(1 + C_\zeta +  \hat\kappa_\infty^{(t_0,\boldsymbol x_0)}((t + r)\wedge T,\boldsymbol{\hat x}^t)\big)\big(1 + C_\zeta + \hat\kappa_\infty^{(s_0,\boldsymbol y_0)}((s + r)\wedge T,\boldsymbol{\hat y}^s)\big)}\,\eta(r)\,dr, \notag
\end{align}
where $\boldsymbol{\hat x}^t(\cdot\wedge((t+r)\wedge T)-)$ and $\boldsymbol{\hat y}^s(\cdot\wedge((s+r)\wedge T)-)$ are defined as in \eqref{xt-}. Notice that, for every fixed $r\geq0$,
\begin{equation}\label{Cont_chi_proof1}
\big|\boldsymbol{\hat x}^t((t+r)\wedge T) - \boldsymbol{\hat y}^s((s+r)\wedge T)\big| \ = \ \big|\boldsymbol{\hat x}(t) - \boldsymbol{\hat y}(s)\big|.
\end{equation}
On the other hand, concerning the term $\|\boldsymbol{\hat x}^t(\cdot\wedge((t+r)\wedge T)-) - \boldsymbol{\hat y}^s(\cdot\wedge((s+r)\wedge T)-)\|_\infty$ we distinguish four cases, for every fixed $r>0$:
\begin{itemize}
\item if both $t<T$ and $s<T$, then $\|\boldsymbol{\hat x}^t(\cdot\wedge((t+r)\wedge T)-) - \boldsymbol{\hat y}^s(\cdot\wedge((s+r)\wedge T)-)\|_\infty=\|\boldsymbol{\hat x}(\cdot\wedge t) - \boldsymbol{\hat y}(\cdot\wedge s)\|_\infty$;
\item if $t=T$ and $s<T$, it holds that
\begin{align*}
&\big\|\boldsymbol{\hat x}^t(\cdot\wedge((t+r)\wedge T)-) - \boldsymbol{\hat y}^s(\cdot\wedge((s+r)\wedge T)-)\big\|_\infty \\
&= \ \big\|\boldsymbol{\hat x}(\cdot\wedge T-) - \boldsymbol{\hat y}(\cdot\wedge s)\big\|_\infty \ \leq \ \big\|\boldsymbol{\hat x}(\cdot) - \boldsymbol{\hat y}(\cdot\wedge s)\big\|_\infty;
\end{align*}
\item similarly to the previous case, if $t<T$ and $s=T$, then $\|\boldsymbol{\hat x}^t(\cdot\wedge((t+r)\wedge T)-) - \boldsymbol{\hat y}^s(\cdot\wedge((s+r)\wedge T)-)\|_\infty\leq\|\boldsymbol{\hat x}(\cdot\wedge t) - \boldsymbol{\hat y}(\cdot)\|_\infty$;
\item finally, if both $t=T$ and $s=T$, it holds that
\begin{align*}
&\big\|\boldsymbol{\hat x}^t(\cdot\wedge((t+r)\wedge T)-) - \boldsymbol{\hat y}^s(\cdot\wedge((s+r)\wedge T)-)\big\|_\infty \\
&= \ \big\|\boldsymbol{\hat x}(\cdot\wedge T-) - \boldsymbol{\hat y}(\cdot\wedge T-)\big\|_\infty \ \leq \ \big\|\boldsymbol{\hat x}(\cdot) - \boldsymbol{\hat y}(\cdot)\big\|_\infty.
\end{align*}
\end{itemize}
In conclusion, for every $r>0$ and any $t,s\in[0,T]$, it holds that,
\begin{equation}\label{Cont_chi_proof2}
\big\|\boldsymbol{\hat x}^t(\cdot\wedge((t+r)\wedge T)-) - \boldsymbol{\hat y}^s(\cdot\wedge((s+r)\wedge T)-)\big\|_\infty \ \leq \ \big\|\boldsymbol{\hat x}(\cdot\wedge t) - \boldsymbol{\hat y}(\cdot\wedge s)\big\|_\infty,
\end{equation}
where $\boldsymbol{\hat x}(\cdot\wedge t)=\boldsymbol{\hat x}(\cdot)$ if $t=T$ and, similarly, $\boldsymbol{\hat y}(\cdot\wedge s)=\boldsymbol{\hat y}(\cdot)$ if $s=T$.\\
Plugging \eqref{Cont_chi_proof1} and \eqref{Cont_chi_proof2} into \eqref{Cont_chi_proof}, we conclude that
\begin{align*}
&\big|\hat\chi_\infty^{(t_0,\boldsymbol x_0)}(t,\boldsymbol{\hat x}) - \hat\chi_\infty^{(s_0,\boldsymbol y_0)}(s,\boldsymbol{\hat y})\big| \\
&\leq \ \int_0^{+\infty} \frac{(1 + C_\zeta)\,\big\|\boldsymbol{\hat x}(\cdot\wedge t) - \boldsymbol{\hat y}(\cdot\wedge s)\big\|_\infty}{\big(1 + C_\zeta + \hat\kappa_\infty^{(t_0,\boldsymbol x_0)}((t + r)\wedge T,\boldsymbol{\hat x}^t)\big)\big(1 + C_\zeta + \hat\kappa_\infty^{(s_0,\boldsymbol y_0)}((s + r)\wedge T,\boldsymbol{\hat y}^s)\big)}\,\eta(r)\,dr \notag \\
&+ \int_0^{+\infty} \frac{(1 + C_\zeta)\,\big|\boldsymbol{\hat x}(t) - \boldsymbol{\hat y}(s)\big|}{\big(1 + C_\zeta + \hat\kappa_\infty^{(t_0,\boldsymbol x_0)}((t + r)\wedge T,\boldsymbol{\hat x}^t)\big)\big(1 + C_\zeta + \hat\kappa_\infty^{(s_0,\boldsymbol y_0)}((s + r)\wedge T,\boldsymbol{\hat y}^s)\big)}\,\eta(r)\,dr \notag \\
&+ \int_0^{+\infty} \frac{(1 + C_\zeta)\,\big\|\boldsymbol x_0(\cdot\wedge(t + r)\wedge t_0) - \boldsymbol y_0(\cdot\wedge(s + r)\wedge s_0)\big\|_\infty}{\big(1 + C_\zeta + \hat\kappa_\infty^{(t_0,\boldsymbol x_0)}((t + r)\wedge T,\boldsymbol{\hat x}^t)\big)\big(1 + C_\zeta + \hat\kappa_\infty^{(s_0,\boldsymbol y_0)}((s + r)\wedge T,\boldsymbol{\hat y}^s)\big)}\,\eta(r)\,dr \notag \\
&+ \int_0^{+\infty} \frac{(1 + C_\zeta)\,\big\|\boldsymbol x_0(((t + r)\wedge T)\vee(\cdot\wedge t_0)) - \boldsymbol y_0(((s + r)\wedge T)\vee(\cdot\wedge s_0))\big\|_\infty}{\big(1 + C_\zeta + \hat\kappa_\infty^{(t_0,\boldsymbol x_0)}((t + r)\wedge T,\boldsymbol{\hat x}^t)\big)\big(1 + C_\zeta + \hat\kappa_\infty^{(s_0,\boldsymbol y_0)}((s + r)\wedge T,\boldsymbol{\hat y}^s)\big)}\,\eta(r)\,dr. \notag
\end{align*}
This proves the continuity of the map $\big((t_0,\boldsymbol x_0),(t,\boldsymbol{\hat x})\big)\mapsto\hat\chi_\infty^{(t_0,\boldsymbol x_0)}(t,\boldsymbol{\hat x})$ on $\boldsymbol\Lambda\times\boldsymbol{\hat\Lambda}$.

\vspace{1mm}

\noindent\textsc{Step III}. \emph{Proof of items }3) \emph{and} 4). Regarding item 3), the bound for the horizontal derivative follows from formulae \eqref{HorDer_t<T}-\eqref{HorDer_t=T}, from the fact that $\hat\kappa_\infty^{(t_0,\boldsymbol x_0)}/(1+C_\zeta+\hat\kappa_\infty^{(t_0,\boldsymbol x_0)})\leq1$, and also from equality $\int_0^{+\infty}|\eta'(s)|\,ds=\frac{2}{\textup{e}}$. On the other hand, the bounds concerning the vertical derivatives follow from formulae \eqref{VerDer} and \eqref{VerDer2nd}, from the bounds on the vertical derivatives of $\hat\kappa_\infty^{(t_0,\boldsymbol x_0)}$ given in item 2) of Lemma \ref{L:rho_infty}, and also from inequality $1/(1+C_\zeta+\hat\kappa_\infty^{(t_0,\boldsymbol x_0)})\leq1$ (here we use that $\hat\kappa_\infty^{(t_0,\boldsymbol x_0)}\geq-C_\zeta$, see item 3 of Lemma \ref{L:rho_infty}).

Finally, regarding item 4), the first inequality is a direct consequence of the two inequalities in \eqref{Ineq12}. On the other hand, the second inequality follows from the second inequality in \eqref{Ineq12} together with the fact that $1/(1+C_\zeta+\hat\kappa_\infty^{(t_0,\boldsymbol x_0)})\leq1$, and also from the inequality $\hat\kappa_\infty^{(t_0,\boldsymbol x_0)}/(1+C_\zeta+\hat\kappa_\infty^{(t_0,\boldsymbol x_0)})\leq1$.
\end{proof}

\section{Cylindrical approximation}
\label{AppC}

In the present Appendix we state two results already proved in \cite{cosso_russoStrict,cosso_russoStrong-Visc}, namely Theorem 3.5 in \cite{cosso_russoStrict} and Theorem 3.12 in \cite{cosso_russoStrong-Visc}, which correspond respectively to Lemma \ref{L:Classical} and Lemma \ref{L:Smoothing} below. Notice however that in \cite{cosso_russoStrict,cosso_russoStrong-Visc} the pathwise derivatives are defined in an alternative manner. For this reason, in order to help the reader, we prefer to report the proof of these two results in the present setting.

\subsection{The deterministic calculus via regularization}

We begin recalling some results from the deterministic calculus of regularization, as developed in Section 3.2 of \cite{digirfabbrirusso13} and Section 2.2 of \cite{cosso_russoStrict}, for which we refer for all the details. The only difference with respect to \cite{digirfabbrirusso13} and \cite{cosso_russoStrict} being that here we consider $\R^d$-valued paths (with $d$ not necessarily equal to $1$), even if, as usual, we rely on the one-dimensional theory, as we work component by component.

Firstly, for every $t\geq0$ and any function $\boldsymbol f\colon[0,t]\rightarrow\R^d$ we define the following extensions to the entire real line:
\[
\boldsymbol f_{(0,t]}(s) \ := \
\begin{cases}
\boldsymbol 0, & s>t, \\
\boldsymbol f(s), & s\in[0,t], \\
\boldsymbol f(0), & s<0,
\end{cases} \qquad\qquad
\boldsymbol f_{[0,t]}(s) \ := \
\begin{cases}
\boldsymbol f(t), & s>t, \\
\boldsymbol f(s), & s\in[0,t], \\
\boldsymbol 0, & s<0.
\end{cases}
\]

\begin{defn}
\label{D:DeterministicIntegral}
Let $\boldsymbol f\colon[0,t]\rightarrow\R^d$ and $g\colon[0,t]\rightarrow\R$ be c\`{a}dl\`{a}g functions. When the  limit
\[
\int_{[0,t]}g(s)\,d^-\boldsymbol f(s) \ := \ \lim_{\eps\rightarrow0^+}\int_{-\infty}^{+\infty} g_{(0,t]}(s)\,\frac{\boldsymbol f_{[0,t]}(s+\eps)-\boldsymbol f_{[0,t]}(s)}{\eps}\,ds
\]
exists and it is finite, we denote it by $\int_{[0,t]} g\,d^-\boldsymbol f$ and call it \textbf{forward integral of $g$ with respect to $\boldsymbol f$}.
\end{defn}

We recall from \cite{cosso_russoStrict}, Proposition 2.11, the following integration by parts formula, which will be used several times in this Appendix.

\begin{prop}
\label{P:IntbyParts}
Let $\boldsymbol f\colon[0,t]\rightarrow\R^d$ and $g\colon[0,t]\rightarrow\R$ be c\`{a}dl\`{a}g functions, with $g$ being of bounded variation. We have
the \textbf{integration by parts formula} 
\[
\int_{[0,t]}  g(s)\,d^-\boldsymbol f(s) \ = \ g(t) \,\boldsymbol f(t)  - \int_{(0,t]} \boldsymbol f(s)\,dg(s),
\]
where $\int_{(0,t]} \boldsymbol f(s)\,dg(s)$ is a Lebesgue-Stieltjes integral on $(0,t]$.
\end{prop}

\subsection{Cylinder terminal condition $\xi$}

\begin{lem}\label{L:Classical}
Suppose that $\xi$ is a cylinder (or tame) function, in the sense that it admits the representation
\begin{align*}
\xi(\boldsymbol x) \ &= \ g\bigg(\int_{[0,T]}\psi_0(t)\,d^-\boldsymbol x(t),\ldots,\int_{[0,T]}\psi_n(t)\,d^-\boldsymbol x(t)\bigg), \qquad \boldsymbol x\in C([0,T];\R^d),
\end{align*}
for some non-negative integer $n$, where we have the following.
\begin{itemize}
\item $g\colon\R^{d(n+1)}\rightarrow\R$ is of class $C^2(\R^{d(n+1)})$ and, together with its first and second-order partial derivatives, satisfies a polynomial growth condition;
\item $\psi_0,\ldots,\psi_n\colon[0,T]\rightarrow\R$ are of class $C^2([0,T])$.
\end{itemize}
Then, the function $v$ defined by \eqref{v} is in $\boldsymbol C^{1,2}(\boldsymbol\Lambda)$ and is a classical (smooth) solution of equation \eqref{PPDE}.
\end{lem}
\begin{proof}
Let $(t,\boldsymbol x)\in\boldsymbol\Lambda$ and consider $\boldsymbol W^{t,\boldsymbol x}=(\boldsymbol W_s^{t,\boldsymbol x})_{s\in[0,T]}$ given by \eqref{W^t,x}. From the definition of $v$ in \eqref{v}, we have
\begin{align*}
v(t,\boldsymbol x) \ &= \ \E\bigg[g\bigg(\int_{[0,T]}\psi_0(s)\,d^-\boldsymbol W_s^{t,\boldsymbol x},\ldots,\int_{[0,T]}\psi_n(s)\,d^-\boldsymbol W_s^{t,\boldsymbol x}\bigg)\bigg] \\
&= \ \E\bigg[g\bigg(\int_{[0,t]}\psi_0(s)\,d^-\boldsymbol x(s) + \int_t^T\psi_0(s)\,d\boldsymbol W_s,\ldots\bigg)\bigg],
\end{align*}
where the second equality follows from the fact that the forward integral coincides with the It\^o integral when the integrator is the Brownian motion (or, more generally, a continuous semimartingale), see for instance Proposition 6, Section 3.3, in \cite{RussoValloisSurvey}.

In order to prove that $v\in\boldsymbol C^{1,2}(\boldsymbol\Lambda)$, we consider the following map $\hat v\colon\boldsymbol{\hat\Lambda}\rightarrow\R$ consistent with $v$:
\[
\hat v(t,\boldsymbol{\hat x}) \ = \ \E\bigg[g\bigg(\int_{[0,t]}\psi_0(s)\,d^-\boldsymbol{\hat x}(s) + \int_t^T\psi_0(s)\,d\boldsymbol W_s,\ldots\bigg)\bigg],
\]
for all $(t,\boldsymbol{\hat x})\in\boldsymbol{\hat\Lambda}$. Notice that
\[
\hat v(t,\boldsymbol{\hat x}) \ = \ \hat V\bigg(t,\int_{[0,t]}\psi_0(s)\,d^-\boldsymbol{\hat x}(s)\,,\,\ldots\bigg),
\]
where $\hat V\colon[0,T]\times\R^{d(n+1)}\rightarrow\R$ is given by
\[
\hat V(t,\mathbf z) \ = \ \E\bigg[g\bigg(\mathbf z_0 + \int_t^T \psi_0(s)\, d\boldsymbol W_s\,,\,\ldots\,,\,\mathbf z_n + \int_t^T \psi_n(s)\, d\boldsymbol W_s\bigg)\bigg],
\]
for all $t\in[0,T]$ and $\mathbf z\in\R^{d(n+1)}$, with $\mathbf z=(\mathbf z_0,\ldots,\mathbf z_n)$ and $\mathbf z_0,\ldots,\mathbf z_n\in\R^d$. Let $\sigma\colon[0,T]\rightarrow\R^{d(n+1)\times d}$ be given by
\begin{equation}\label{sigma}
\sigma(t) \ = \
\left[
\begin{array}{c}\psi_0(t)\,I \\\psi_1(t)\,I \\\vdots \\
\psi_n(t)\,I\end{array}
\right]
\end{equation}
where $I$ denotes the $d\times d$ identity matrix. It is well-known (see, for instance, Theorem 5.6.1 in \cite{friedman75vol1}) that $\hat V\in C^{1,2}([0,T]\times\R^{d(n+1)})$ and satisfies (here $\partial_t\hat V(t,\mathbf z)$ denotes the time derivative, while $\partial_{\mathbf z\mathbf z}\hat V(t,\mathbf z)$ is the Hessian matrix of spatial derivatives)
\begin{equation}\label{PDE_V}
\begin{cases}
\vspace{2mm}
\displaystyle \partial_t\hat V(t,\mathbf z) + \frac{1}{2} \text{tr}\big[\sigma(t)\sigma\trans(t)\partial_{\mathbf z\mathbf z}\hat V(t,\mathbf z)\big] \ = \ 0, \qquad &(t,\mathbf z)\in[0,T)\times\R^{d(n+1)}, \\
\hat V(T,\mathbf z) \ = \ g(\mathbf z), &\,\mathbf z\in\R^{d(n+1)}.
\end{cases}
\end{equation}
Let us find the expression of the pathwise derivatives of $\hat v$ in terms of $\hat V$. Concerning the horizontal derivative at $(t,\boldsymbol{\hat x})\in\boldsymbol{\hat\Lambda}$, with $t<T$, we have
\begin{align*}
&\partial_t^H\hat v(t,\boldsymbol{\hat x}) \ = \ \lim_{\delta\rightarrow0^+}\frac{\hat v(t+\delta,\boldsymbol{\hat x}(\cdot\wedge t)) - \hat v(t,\boldsymbol{\hat x})}{\delta} \\
&= \ \lim_{\delta\rightarrow0^+}\frac{1}{\delta}\bigg\{\hat V\bigg(t+\delta,\int_{[0,t+\delta]}\psi_0(s)\,d^-\boldsymbol{\hat x}(s\wedge t),\ldots\bigg) - \hat V\bigg(t,\int_{[0,t]}\psi_0(s)\,d^-\boldsymbol{\hat x}(s),\ldots\bigg)\bigg\} \\
&= \ \lim_{\delta\rightarrow0^+}\frac{1}{\delta}\bigg\{\hat V\bigg(t+\delta,\int_{[0,t]}\psi_0(s)\,d^-\boldsymbol{\hat x}(s),\ldots\bigg) - \hat V\bigg(t,\int_{[0,t]}\psi_0(s)\,d^-\boldsymbol{\hat x}(s),\ldots\bigg)\bigg\} \\
&= \ \partial_t\hat V\bigg(t,\int_{[0,t]}\psi_0(s)\,d^-\boldsymbol{\hat x}(s),\ldots\bigg).
\end{align*}
Moreover, at $t=T$ we have
\begin{align*}
\partial_t^H\hat v(T,\boldsymbol{\hat x}) \ &= \ \lim_{t\rightarrow T-} \partial_t^H\hat v(t,\boldsymbol{\hat x}) \ = \ \lim_{t\rightarrow T-} \partial_t\hat V\bigg(t,\int_{[0,t]}\psi_0(s)\,d^-\boldsymbol{\hat x}(s),\ldots\bigg) \\
&= \ \partial_t\hat V\bigg(T,\int_{[0,T]}\psi_0(s)\,d^-\boldsymbol{\hat x}(s) - \psi_0(T)\,\big(\boldsymbol{\hat x}(T) - \boldsymbol{\hat x}(T-)\big),\ldots\bigg),
\end{align*}
where the latter equality follows from the integration by parts formula of Proposition \ref{P:IntbyParts}. Concerning the first-order vertical derivatives at $(t,\boldsymbol{\hat x})\in\boldsymbol{\hat\Lambda}$,  for every $i=1,\ldots,d$, we have
\begin{align*}
\partial^V_{x_i}\hat v(t,\boldsymbol{\hat x}) \ &= \ \lim_{h\rightarrow0}\frac{\hat v(t,\boldsymbol{\hat x} + h\,\mathbf e_i\,1_{[t,T]}) - \hat v(t,\boldsymbol{\hat x})}{h} \\
&= \ \lim_{h\rightarrow0} \frac{1}{h} \bigg\{\hat V\bigg(t,\int_{[0,t]}\psi_0(s)\,d^-\boldsymbol{\hat x}(s) + \psi_0(t)\,h\,\mathbf e_i\,,\,\ldots\bigg) \\
&\quad \ - \hat V\bigg(t,\int_{[0,t]}\psi_0(s)\,d^-\boldsymbol{\hat x}(s)\,,\,\ldots\bigg)\bigg\} \\
&= \ \bigg\langle\sigma_i(t),\partial_{\mathbf z}\hat V\bigg(t,\int_{[0,t]}\psi_0(s)\,d^-\boldsymbol{\hat x}(s)\,,\,\ldots\bigg)\bigg\rangle,
\end{align*}
where $\langle\cdot,\cdot\rangle$ denotes the scalar product in $\R^{d(n+1)}$, $\sigma_i(t)$ is the $i$-th column of the matrix $\sigma(t)$ in \eqref{sigma}, and $\partial_{\mathbf z}\hat V$ is the gradient of spatial derivatives of $\hat V$.

Concerning the second-order vertical derivatives, it holds that
\[
\partial^V_{\boldsymbol x\boldsymbol x} \hat v(t,\boldsymbol{\hat x}) \ = \ \sigma\trans(t)\partial_{\mathbf z\mathbf z}\hat V\bigg(t,\int_{[0,t]}\psi_0(s)\,d^-\boldsymbol{\hat x}(s)\,,\,\ldots\bigg)\sigma(t).
\]
Since $\hat V\in C^{1,2}([0,T]\times\R^{d(n+1)})$, we deduce that $\hat v\in\boldsymbol C^{1,2}(\boldsymbol{\hat\Lambda})$, so that $v\in\boldsymbol C^{1,2}(\boldsymbol\Lambda)$. Finally, since $\hat V$ is a classical (smooth) solution of equation \eqref{PDE_V}, using the relations between the pathwise derivatives of $\hat v$ (and hence of $v$) and the derivatives of $\hat V$, we deduce that $v$ is a classical (smooth) solution of equation \eqref{PPDE}.
\end{proof}

\subsection{Cylindrical approximation}

\begin{lem}\label{L:Smoothing}
Suppose that $\xi\colon C([0,T];\R^d)\rightarrow\R$ is continuous and satisfies the polynomial growth condition
\[
|\xi(\boldsymbol x)| \ \leq \ M(1 + \|\boldsymbol x\|_\infty^p), \qquad \text{for all }\boldsymbol x\in C([0,T];\R^d),
\]
for some positive constants $M$ and $p$. Then, there exists a sequence $\{\xi_n\}_n$, with $\xi_n$ being a map from $C([0,T];\R^d)$ into $\R$, such that the following
holds.
\begin{enumerate}[\upshape I)]
\item $\{\xi_n\}_n$ converges pointwise to $\xi$ as $n\rightarrow+\infty$.
\item If $\xi$ is bounded then $\xi_n$ is bounded uniformly with respect to $n$.
\item For every $n$, $\xi_n$ is given by
\[
\xi_n(\boldsymbol x) \ = \ g_n\bigg(\int_{[0,T]}\psi_0(t)\,d^-\boldsymbol x(t),\ldots,\int_{[0,T]}\psi_n(t)\,d^-\boldsymbol x(t)\bigg),
\]
where the properties below hold:
\begin{enumerate}[\upshape i)]
\item for every $n$, $g_n\colon\R^{d(n+1)}\rightarrow\R$ is of class $C^\infty(\R^{d(n+1)})$, with partial derivatives of every order satisfying a polynomial growth condition; moreover, $g_n$ satisfies
\[
|g_n(\mathbf z)| \ \leq \ M'\big(1 + |\mathbf z|^{p'}\big), \qquad \text{for all }\mathbf z\in\R^{d(n+1)},
\]
for some positive constants $M'$ and $p'$, not depending on $n$;
\item the functions $\psi_\ell\colon[0,T]\rightarrow\R$, $\ell\geq0$, satisfy what follows:
\begin{enumerate}[\upshape a)]
\item $\psi_\ell$ is of class $C^\infty([0,T])$;
\item $\psi_\ell$ is uniformly bounded with respect to $\ell$;
\item the first derivative of $\psi_\ell$ is bounded in $L^1([0,T])$, uniformly with respect to $\ell$.
\end{enumerate}
\end{enumerate}
\end{enumerate}
\end{lem}
\begin{proof} \textbf{Step I.} Let $\{e_\ell\}_{\ell\geq0}$ be the following orthonormal basis of $L^2([0,T];\R)$:
\[
e_0(t) \ = \ \frac{1}{\sqrt{T}}, \qquad e_{2\ell-1}(t) \ = \ \sqrt{\frac{2}{T}}\sin\bigg(2\ell\pi\frac{t}{T}\bigg), \qquad e_{2\ell}(t) \ = \ \sqrt{\frac{2}{T}}\cos\bigg(2\ell\pi\frac{t}{T}\bigg),
\]
for all $\ell\geq1$ and $t\in[0,T]$. Consider the linear operator $\Lambda\colon C([0,T];\R^d)\rightarrow C([0,T];\R^d)$ given by
\[
(\Lambda\boldsymbol x)(t) \ = \ \boldsymbol x(T)\frac{t}{T}, \qquad \text{for all }t\in[0,T],\,\,\boldsymbol x\in C([0,T];\R^d).
\]
Observe that $(\boldsymbol x-\Lambda\boldsymbol x)(0) = (\boldsymbol x-\Lambda\boldsymbol x)(T)=0$. Now, for every $n\geq0$, define
\begin{align*}
s_n(\boldsymbol x - \Lambda\boldsymbol x) \ &= \ \sum_{\ell=0}^n (\boldsymbol x - \Lambda\boldsymbol x)_\ell \, e_\ell, \\
\sigma_n(\boldsymbol x - \Lambda\boldsymbol x) \ &= \ \frac{s_0(\boldsymbol x - \Lambda\boldsymbol x) + \cdots + s_n(\boldsymbol x - \Lambda\boldsymbol x)}{n+1} \ = \ \sum_{\ell=0}^n \frac{n+1-\ell}{n+1} (\boldsymbol x - \Lambda\boldsymbol x)_\ell \, e_\ell,
\end{align*}
where $(\boldsymbol x - \Lambda\boldsymbol x)_\ell$ is given by
\[
\int_0^T (\boldsymbol x(t) - (\Lambda\boldsymbol x)(t)) e_\ell(t) dt \ = \ \int_0^T \boldsymbol x(t) e_\ell(t) dt - \boldsymbol x(T){\mathcal E}_\ell(T) + x(T)\frac{1}{T}\int_0^T {\mathcal E}_\ell(t)dt,
\]
with ${\mathcal E}_\ell$ being a primitive of $e_\ell$. In particular, we take
\[
{\mathcal E}_0(t) \ = \ \frac{t}{\sqrt{T}} - \frac{\sqrt{T}}{2}, \;\; {\mathcal E}_{2\ell-1}(t) \ = \ - \sqrt{\frac{T}{2}}\frac{1}{\ell\pi}\cos\bigg(2\ell\pi\frac{t}{T}\bigg), \;\; {\mathcal E}_{2\ell}(t) \ = \ \sqrt{\frac{T}{2}}\frac{1}{\ell\pi}\sin\bigg(2\ell\pi\frac{t}{T}\bigg),
\]
for all $\ell\geq1$ and $t\in[0,T]$. Then
\[
(\boldsymbol x - \Lambda\boldsymbol x)_\ell \ = \ \int_0^T \boldsymbol x(t) e_\ell(t) dt - \boldsymbol x(T){\mathcal E}_\ell(T) \ = \ - \int_{[0,T]}{\mathcal E}_\ell(t)\,d^-\boldsymbol x(t),
\]
for all $\ell\geq0$. By Fej\'er's theorem (see for instance Theorem III.3.4 in \cite{zygmund02}), we have
\[
\|\sigma_n(\boldsymbol x-\Lambda\boldsymbol x)-(\boldsymbol x-\Lambda\boldsymbol x)\|_\infty \ \overset{n\rightarrow\infty}{\longrightarrow} \ 0 \qquad\quad \text{ and } \qquad\quad \|\sigma_n(\boldsymbol x-\Lambda\boldsymbol x)\|_\infty \ \leq \ \|\boldsymbol x-\Lambda\boldsymbol x\|_\infty,
\]
for all $\boldsymbol x\in C([0,T];\R^d)$. Consider the linear operator $T_n\colon C([0,T];\R^d)\rightarrow C([0,T];\R^d)$ given by ($e_{-1}(t):=\tfrac{t}{T}$, for all $t\in[0,T]$)
\begin{align*}
T_n\boldsymbol x \ &= \ \Lambda\boldsymbol x + \sigma_n(\boldsymbol x-\Lambda\boldsymbol x) - \big(\sigma_n(\boldsymbol x-\Lambda\boldsymbol x)\big)(0) \\
&= \ \boldsymbol x(T)e_{-1} + \sum_{\ell=1}^n \frac{n+1-\ell}{n+1} (\boldsymbol x - \Lambda\boldsymbol x)_\ell \, (e_\ell - e_\ell(0)),
\end{align*}
for all $n\geq0$, where for the latter equality we used the fact that $e_0$ is constant. Then, for any $\boldsymbol x\in C([0,T];\R^d)$, $\|T_n\boldsymbol x-\boldsymbol x\|_\infty\rightarrow0$, as $n$ tends to infinity. Furthermore, there exists a positive constant $C$, independent of $n$, such that
\begin{equation}
\label{E:UniformBoundT_n}
\|T_n\boldsymbol x\|_\infty \ \leq \ C\|\boldsymbol x\|_\infty, \qquad \text{for all }\boldsymbol x\in C([0,T];\R^d),\,\,n\geq0.
\end{equation}
Then, we define $\tilde\xi_n(\boldsymbol x):=\xi(T_n\boldsymbol x)$. Notice that
\begin{align*}
\tilde\xi_n(\boldsymbol x) \ &= \ \xi\bigg(\boldsymbol x(T)e_{-1} + \sum_{\ell=1}^n \frac{n+1-\ell}{n+1} (\boldsymbol x - \Lambda\boldsymbol x)_\ell \, (e_\ell - e_\ell(0))\bigg) \\
&= \ \xi\bigg(\boldsymbol x(T)e_{-1} - \sum_{\ell=1}^n \frac{n+1-\ell}{n+1} \bigg(\int_{[0,T]}{\mathcal E}_\ell(t)\,d^-\boldsymbol x(t)\bigg) \, (e_\ell - e_\ell(0))\bigg) \\
&= \ \tilde g_n\bigg(\boldsymbol x(T)\,,\,\int_{[0,T]}{\mathcal E}_1(t)\,d^-\boldsymbol x(t)\,,\,\ldots\,,\,\int_{[0,T]}{\mathcal E}_n(t)\,d^-\boldsymbol x(t)\bigg) \\
&= \ \tilde g_n\bigg(\int_{[0,T]}1\,d^-\boldsymbol x(t)\,,\,\int_{[0,T]}{\mathcal E}_1(t)\,d^-\boldsymbol x(t)\,,\,\ldots\,,\,\int_{[0,T]}{\mathcal E}_n(t)\,d^-\boldsymbol x(t)\bigg),
\end{align*}
where the last inequality follows from the identity $\boldsymbol x(T)=\int_{[0,T]}1\,d^-\boldsymbol x(t)$, while $\tilde g_n$ is a map from $\R^{d(n+1)}$ into $\R$ given by
\begin{equation}\label{tilde_g_n}
\tilde g_n(\mathbf z) \ := \ \xi\bigg(\mathbf z_0 e_{-1} - \sum_{\ell=1}^n \frac{n+1-\ell}{n+1}\,\mathbf z_\ell\,(e_\ell - e_\ell(0))\bigg),
\end{equation}
for all $\mathbf z\in\R^{d(n+1)}$, with $\mathbf z=(\mathbf z_0,\ldots,\mathbf z_n)$ and $\mathbf z_0,\ldots,\mathbf z_n\in\R^d$. From now on we denote
\begin{equation}\label{varphi}
\psi_0(t) \ = \ 1, \qquad\qquad\qquad\qquad \psi_\ell(t) \ = \ {\mathcal E}_\ell(t), \qquad \ell\geq1.
\end{equation}

\vspace{1mm}

\noindent\textbf{Step II.} We begin introducing the double sequence $\{g_{n,k}\}_{n\geq0,k\geq1}$, with $g_{n,k}\colon\R^{d(n+1)}\rightarrow\R$ given by
\[
g_{n,k}(\mathbf z) \ = \ \int_{\R^{d(n+1)}} \tilde g_n(\mathbf z - \mathbf w)\,\zeta_{n,k}(\mathbf w)\,d\mathbf w,
\]
where $\zeta_{n,k}(\mathbf z) = k^{d(n+1)}\zeta_n(k\,\mathbf z)$, for all $\mathbf z\in\R^{d(n+1)}$, with
\[
\zeta_n(\mathbf z) \ = \ c_n \prod_{\ell=0}^n \exp\bigg(\frac{1}{\mathbf z_\ell^2 - 2^{-2\ell}}\bigg) 1_{\{|\mathbf z_\ell|<2^{-\ell}\}}, \qquad \text{for all }\mathbf z=(\mathbf z_0,\ldots,\mathbf z_n)\in\R^{d(n+1)}.
\]
The constant $c_n>0$ is such that $\int_{\R^{d(n+1)}} \zeta_n(\mathbf z)\,d\mathbf z=1$. We also introduce the double sequence $\{\xi_{n,k}\}_{n\geq0,k\geq1}$, with $\xi_{n,k}\colon C([0,T];\R^d)\rightarrow\R$ given by
\[
\xi_{n,k}(\boldsymbol x) \ = \ g_{n,k}\bigg(\int_{[0,T]}\psi_0(t)\,d^-\boldsymbol x(t),\ldots,\int_{[0,T]}\psi_n(t)\,d^-\boldsymbol x(t)\bigg), \qquad \boldsymbol x\in C([0,T];\R^d),
\]
with $\psi_0,\ldots,\psi_n$ as in \eqref{varphi}. Our aim is to apply Lemma D.1 in \cite{cosso_russoStrong-Visc}. To this end, we need to prove the following items:
\begin{enumerate}[a)]
\item $\xi_{n,k}$ is continuous;
\item for every $\boldsymbol x\in\boldsymbol C([0,T];\R^d)$, $|\xi_{n,k}(\boldsymbol x)-\tilde \xi_n(\boldsymbol x)|\rightarrow0$ as $k\rightarrow+\infty$;
\item $\{\xi_{n,k}\}_{n\geq0,k\geq1}$ is equicontinuous on compact sets.
\end{enumerate}
Suppose for a moment that items a)-b)-c) hold true. Then, by Lemma D.1 in \cite{cosso_russoStrong-Visc} we deduce the existence of a subsequence $\{\xi_{n,k_n}\}_n$ converging pointwise to $\xi$. Hence, we set $\xi_n:=\xi_{n,k_n}$ and $g_n:=g_{n,k_n}$. It is then easy to see that $\{\xi_n\}_n$ is the claimed sequence. It remains to prove a)-b)-c). As items a) and b) can be easily proved, we only report the proof of item c).

\vspace{1mm}

\noindent\textbf{Step III.} Let us prove item c). We begin noticing that $g_{n,k}$ can be rewritten as follows:
\[
g_{n,k}(\mathbf z) \ = \ \int_{\boldsymbol E_n} \tilde g_n(\mathbf z - \mathbf w)\,\zeta_{n,k}(\mathbf w)\,d\mathbf w,
\]
where
\[
\boldsymbol E_n \ := \ \big\{\mathbf z=(\mathbf z_0,\ldots,\mathbf z_n)\in\R^{d(n+1)}\colon|\mathbf z_\ell|\leq 2^{-\ell},\,\ell=0,\ldots,n\big\}.
\]
Then, the claim follows if we prove that for any compact set $\boldsymbol K\subset C([0,T];\R^d)$ there exists a modulus of continuity $\rho_{\boldsymbol K}$ such that, for every $n$,
\begin{align}
\label{UCtilde_g_n}
&\bigg|\tilde g_n\bigg(\int_{[0,T]}\psi_0(t)\,d^-\boldsymbol x(t) + \mathbf z_0,\ldots,\int_{[0,T]}\psi_n(t)\,d^-\boldsymbol x(t) + \mathbf z_n\bigg) \\
&- \tilde g_n\bigg(\int_{[0,T]}\psi_0(t)\,d^-\boldsymbol x'(t) + \mathbf z_0,\ldots,\int_{[0,T]}\psi_n(t)\,d^-\boldsymbol x'(t) + \mathbf z_n\bigg)\bigg| \ \leq \ \rho_{\boldsymbol K}(\|\boldsymbol x - \boldsymbol x'\|_\infty), \notag
\end{align}
for all $\boldsymbol x,\boldsymbol x'\in\boldsymbol K$ and $\mathbf z=(\mathbf z_0,\ldots,\mathbf z_n)\in\boldsymbol E_n$. Let us prove \eqref{UCtilde_g_n}. Given a compact set $\boldsymbol K\subset C([0,T];\R^d)$, we denote
\begin{align*}
\boldsymbol{\mathcal K} \ := \ \bigg\{\boldsymbol x\in C([0,T];\R^d)\colon \boldsymbol x \ = \ T_n\boldsymbol{\bar x} + \mathbf z_0 e_{-1} - \sum_{\ell=1}^n \frac{n+1-\ell}{n+1} \mathbf z_\ell (e_\ell - e_\ell(0)),& \\
\text{for some }\boldsymbol{\bar x}\in\boldsymbol K,\,\mathbf z\in\boldsymbol E_n,\,n\geq0&\bigg\}.
\end{align*}
Recalling \eqref{tilde_g_n}, we see that if $\boldsymbol{\mathcal K}$ is relatively compact then \eqref{UCtilde_g_n} follows from the uniform continuity of $\xi$ on compact sets (which in turn follows from the continuity of $\xi$). In order to prove that $\boldsymbol{\mathcal K}$ is relatively compact, we observe that $\boldsymbol{\mathcal K}\subset\boldsymbol{\mathcal K}_1+\boldsymbol{\mathcal K}_2:=\{\boldsymbol x\in C([0,T];\R^d)\colon\boldsymbol x \ = \ \boldsymbol x_1+\boldsymbol x_2,\,\boldsymbol x_1\in\boldsymbol{\mathcal K}_1,\,\boldsymbol x_2\in\boldsymbol{\mathcal K}_2\}$, where
\begin{align*}
\boldsymbol{\mathcal K}_1 := \big\{\boldsymbol x\in C([0,T];\R^d)\colon \boldsymbol x&=T_n\boldsymbol{\bar x},\,\text{for some }\boldsymbol{\bar x}\in\boldsymbol K,\,n\geq0\big\}, \\
\boldsymbol{\mathcal K}_2 := \bigg\{\!\boldsymbol x\in C([0,T];\R^d)\colon \boldsymbol x&=\mathbf z_0 e_{-1}\!-\!\sum_{\ell=1}^n \frac{n+1-\ell}{n+1} \mathbf z_\ell (e_\ell - e_\ell(0)),\,\text{for some }n,\,\mathbf z\in\boldsymbol E_n\!\bigg\}.
\end{align*}
If we prove that $\boldsymbol{\mathcal K}_1$ and $\boldsymbol{\mathcal K}_2$ are relatively compact, it follows that $\boldsymbol{\mathcal K}$ is also relatively compact.

\begin{itemize}
\item \emph{$\boldsymbol{\mathcal K}_1$ is relatively compact.} Let $\{\boldsymbol x_h\}_h$ be a sequence in $\boldsymbol{\mathcal K}_1$. Let us prove that, up to a subsequence, $\{\boldsymbol x_h\}_h$ converges. For each $h$, there exists $\boldsymbol{\bar x}_h\in\boldsymbol K$ and $n_h\geq0$ such that $\boldsymbol x_h \ = \ T_{n_h}\boldsymbol{\bar x}_h$. Suppose that, up to a subsequence, $n_h$ goes to infinity (the proof is simpler when $n_h$ is bounded). Since $\{\boldsymbol{\bar x}_h\}_h\subset\boldsymbol K$, there exists $\boldsymbol{\bar x}\in\boldsymbol K$ such that $\{\boldsymbol{\bar x}_h\}_h$ converges, up to a subsequence, to $\boldsymbol{\bar x}$. Then
\[
\|\boldsymbol x_h - \boldsymbol{\bar x}\|_\infty \ = \ \|T_{n_h}\boldsymbol{\bar x}_h - \boldsymbol{\bar x}\|_\infty \ \leq \ \|T_{n_h}\boldsymbol{\bar x}_h - T_{n_h}\boldsymbol{\bar x}\|_\infty + \|T_{n_h}\boldsymbol{\bar x} - \boldsymbol{\bar x}\|_\infty.
\]
By \eqref{E:UniformBoundT_n} and $\|T_{n_h}\boldsymbol{\bar x} - \boldsymbol{\bar x}\|_\infty\rightarrow0$, the claim follows.
\item \emph{$\boldsymbol{\mathcal K}_2$ is relatively compact.} Let $\{\boldsymbol x_h\}_h$ be a sequence in $\boldsymbol{\mathcal K}_2$. In order to prove that, up to a subsequence, $\{\boldsymbol x_h\}_h$ is convergent, we start
  noticing that
\[
\boldsymbol x_h \ = \ \mathbf z_{0,h} e_{-1} - \sum_{\ell=1}^{n_h} \frac{n_h+1-\ell}{n_h+1} \mathbf z_{\ell,h} (e_\ell - e_\ell(0)),
\]
for some $n_h\geq0$ and $\mathbf z_h=(\mathbf z_{0,h},\ldots,\mathbf z_{n_h,h})\in\boldsymbol E_{n_h}$. Suppose that, up to a subsequence, $n_h$ goes to infinity, otherwise the proof is simpler. Notice that, each sequence $\{\mathbf z_{\ell,h}\}_h$ converges, up to a subsequence, to some $\mathbf z_\ell$, with $|\mathbf z_\ell|\leq 2^{-\ell}$. By Cantor's diagonal argument, there exists a subsequence of $\{\boldsymbol x_h\}_h$, which we still denote $\{\boldsymbol x_h\}_h$, such that every $\{\mathbf z_{\ell,h}\}_h$ converges to $\mathbf z_\ell$. We construct this subsequence in such a way that, for every $h$, $|\mathbf z_{0,h}-\mathbf z_0|+\cdots+|\mathbf z_{n_h,h}-\mathbf z_{n_h}|\leq1/h$. It follows that $\|\boldsymbol x_h-\boldsymbol x\|_\infty\rightarrow0$, where $\boldsymbol x=\mathbf z_0 e_{-1} - \sum_{\ell=1}^\infty \mathbf z_\ell (e_\ell - e_\ell(0))$.
\end{itemize}
\end{proof}

\section*{Acknowledgements}

 The authors are very grateful to Prof. Mikhail Gomoyunov
for his careful reading of the first version of the paper, for his
comments and the very challenging questions.
They also acknowledge the Referees, the Associated Editor and the Editor in
 Chief for their stimulating comments.
The work of the second named author
was supported by a public grant as part of the
{\it Investissement d'avenir project, reference ANR-11-LABX-0056-LMH,
  LabEx LMH,}
in a joint call with Gaspard Monge Program for optimization, operations
research and their interactions with data sciences.

\bibliographystyle{plain}
\bibliography{references}

\end{document}